\documentclass[11pt,reqno]{amsart}

\usepackage{amssymb,amsthm,amsmath,amscd,caption,float,extarrows}
\usepackage[backref=page, colorlinks=true, linkcolor=magenta, citecolor=cyan]{hyperref}
\usepackage{mathtools}
\usepackage{soul}
\usepackage{times}

\RequirePackage[dvipsnames,usenames]{color}

\input{kmacros3.sty}
\input{quiver.sty}

\usepackage{tikz}
\usepackage{tikz-cd}
\usetikzlibrary{cd}
\usepackage{graphicx}
\usepackage[all,cmtip]{xy}

\usepackage{mabliautoref}

\usepackage{bm}
\usepackage{pifont}
\usepackage{upgreek}

\usepackage{eucal}
\usepackage{ulem}
\usepackage{stmaryrd}

\numberwithin{equation}{theorem}

\usepackage{fullpage}

\usepackage{setspace}

\usepackage{enumitem}
\usepackage{calc}

\usepackage{verbatim}
\usepackage{alltt}

\DeclareMathOperator{\Span}{span}

\newcommand{\p}{{\mathfrak p}}
\newcommand{\fm}{{\mathfrak m}}
\newcommand{\fP}{{\mathfrak P}}

\DeclareMathOperator{\ev}{eval}

\DeclareMathOperator{\Pure}{Pure}

\DeclareMathOperator{\cont}{c}
\DeclareMathOperator{\cl}{cl}

\newtheoremstyle{cited}{.5\baselineskip\@plus.2\baselineskip\@minus.2\baselineskip}{.5\baselineskip\@plus.2\baselineskip\@minus.2\baselineskip}{\itshape}{}{\bfseries}{\bfseries .}{5pt plus 1pt minus 1pt}{\thmname{#1}\thmnumber{ #2}\thmnote{ \normalfont#3}}

\theoremstyle{cited}
\newtheorem{citedthm}[theorem]{Theorem}

\newtheoremstyle{citeddef}{.5\baselineskip\@plus.2\baselineskip\@minus.2\baselineskip}{.5\baselineskip\@plus.2\baselineskip\@minus.2\baselineskip}{}{}{\bfseries}{\bfseries .}{5pt plus 1pt minus 1pt}{\thmname{#1}\thmnumber{ #2}\thmnote{ \normalfont#3}}
\theoremstyle{citeddef}

\makeatletter
\def\@tocline#1#2#3#4#5#6#7{\relax
  \ifnum #1>\c@tocdepth 
  \else
    \par \addpenalty\@secpenalty\addvspace{#2}%
    \begingroup \hyphenpenalty\@M
    \@ifempty{#4}{%
      \@tempdima\csname r@tocindent\number#1\endcsname\relax
    }{%
      \@tempdima#4\relax
    }%
    \parindent\z@ \leftskip#3\relax \advance\leftskip\@tempdima\relax
    \rightskip\@pnumwidth plus4em \parfillskip-\@pnumwidth
    #5\leavevmode\hskip-\@tempdima
      \ifcase #1
       \or\or \hskip 1em \or \hskip 2em \else \hskip 3em \fi%
      #6\nobreak\relax
    \hfill\hbox to\@pnumwidth{\@tocpagenum{#7}}\par
    \nobreak
    \endgroup
  \fi}

\makeatother

\begin{document}

\title{Variants on Frobenius Intersection Flatness  \protect\\ and Applications to Tate Algebras}
\author{Rankeya Datta}
\address{Department of Mathematics, University of Missouri, Columbia, MO 65212 USA}
\email{rankeya.datta@missouri.edu}
\author{Neil Epstein}
\address{Department of Mathematical Sciences, George Mason University, Fairfax, VA 22030, USA}
\email{nepstei2@gmu.edu}
\author{Karl Schwede}
\address{Department of Mathematics, University of Utah, Salt Lake City, UT 84112, USA}
\email{schwede@math.utah.edu}
\author{Kevin Tucker}
\address{Department of Mathematics, University of Illinois at Chicago, Chicago, IL 60607, USA}
\email{kftucker@uic.edu}


\begin{abstract}
The theory of singularities defined by Frobenius has been extensively developed for $F$-finite rings and for rings that are essentially of finite type over excellent local rings.  However, important classes of non-local excellent rings, such as Tate algebras and their quotients (affinoid algebras) do not fit into either setting. We investigate here a framework for moving beyond the $F$-finite setting, developing the theory of three related classes of regular rings defined by properties of Frobenius. In increasing order of strength, these are Frobenius Ohm-Rush (FOR), Frobenius intersection flat, and Frobenius Ohm-Rush trace (FORT).
 We show that Tate algebras are Frobenius intersection flat, from which it follows that reduced affinoid algebras have test elements using a result of Sharp. We also deduce new cases of the openness of the $F$-pure locus.
\end{abstract}

\maketitle

\setcounter{tocdepth}{1}
\tableofcontents

\section{Introduction}

The primary tool to detect and measure singularities in prime characteristic $p > 0$ commutative algebra and algebraic geometry is the Frobenius or $p$-th power endomorphism. Indeed, this is centrally motivated by Kunz's fundamental result  that a Noetherian ring is regular if and only if Frobenius is flat \cite{KunzCharacterizationsOfRegularLocalRings}.
The introduction of Frobenius splitting techniques \cite{HochsterRobertsFrobeniusLocalCohomology,MehtaRamanathanFrobeniusSplittingAndCohomologyVanishing} and later Hochster and Huneke's celebrated theory of tight closure \cite{HochsterHunekeTCandBrianconSkoda,HochsterHunekeInfiniteIntegralExtensions} have led to widespread applications. Nonetheless, from the outset, much of the theory of $F$-singularities has necessarily been focused on the $F$-finite setting (i.e., where the Frobenius is a finite morphism) -- or more generally for those rings which can be reduced to the $F$-finite setting, such as for rings (essentially) of finite type over a quasi-excellent local ring. It is easy to see why this has long been the case. For example, $F$-finiteness guarantees excellence \cite{KunzOnNoetherianRingsOfCharP} and the existence of a dualizing complex \cite{Gabber.tStruc}, allows for the application of finite duality to Frobenius, and is often essential to guarantee compatibility with localization and completion. Our goal in this paper is to investigate a framework for moving beyond the $F$-finite setting, and, moreover, to exhibit a prominent class of rings not addressed by existing methods for which this framework applies.

Suppose $R = S/I$, where $S$ is a regular ring of prime characteristic $p > 0$. For an $S$-module $M$ denote by $F_*M$ the corresponding $S$-module given by restriction of scalars along the Frobenius endomorphism, and let $J^{[p]}$ denote the expansion of an ideal $J \subseteq S$ via Frobenius.
Leveraging \cite{FedderFPureRat}, it has been understood dating back to at least \cite{HochsterHunekeFRegularityTestElementsBaseChange} that many properties of the $F$-singularities of $R$ generalize beyond the $F$-finite setting if $S$ is Frobenius intersection flat, meaning roughly that expansions of modules over Frobenius commutes with arbitrary intersections. More precisely, for a submodule $U$ of a module $L$, let $U F_*S$ denote the image of $U \otimes_S F_*S$ in $L \otimes_S F_*S$. We say that $S$ is
\emph{Frobenius intersection flat} provided for each collection of submodules $\{ U_i \}_i$ of a finitely generated $S$-module $L$, we have that 
\begin{equation}
\label{eq:fifsubmodcondition}
\left(\bigcap_i U_i\right) F_*S = \bigcap_i (U_i F_*S).
\end{equation} 
Note that, in the important case where $L = S$ and $\{ J_i \}$ is a collection of ideals, we are simply asking that 
\begin{equation}
\label{eq:foridealcontition}
\bigg( \bigcap_i J_i \bigg)^{[p]} = \bigcap_i \big( J_i^{[p]} \big)
\end{equation}
as $JF_*S = F_*J^{[p]}$ for an ideal $J \subseteq S$.
As an explicit example of the utility of this condition, Sharp \cite{SharpBigTestElements} has shown that test elements exist for reduced quotients of excellent regular rings that are Frobenius intersection flat.

Our first main result is that Tate algebras are Frobenius intersection flat, in turn guaranteeing the existence of test elements for reduced affinoid algebras. We believe this result provides the first new class of reduced excellent rings for which one now knows the existence of test elements since the works of Hochster and Huneke \cite{HochsterHunekeTCandBrianconSkoda, HochsterHunekeFRegularityTestElementsBaseChange} and Sharp \cite{SharpTestElementsforFpure}.  The point is that affinoid algebras are not necessarily essentially of finite type over an excellent semi-local ring or $F$-pure.

\begin{theoremA*}[{\autoref{cor:Tate-char-p-FIF-FORT}, \autoref{cor:affinoid-rings-test-elements}}]
    For any non-Archimedean normed field $(k, |\cdot|)$ of characteristic $p > 0$, $T_n(k)$ is Frobenius intersection flat.  Hence, any reduced affinoid algebra over $k$ has a test element.  
    \label{thm:thmA}
\end{theoremA*}

Infact, we are able to prove the existence of big test elements for reduced algebras of finite type over an affinoid algebra and for ideal-adic completions of reduced affinoid algebras as well (\autoref{cor:affinoid-rings-test-elements}, \autoref{cor:ideal-adic-completion-affinoid}). Recall that Tate algebras are fundamental objects in rigid analytic geometry, acting as a counterpart to polynomial rings in classical algebraic geometry. First introduced by Tate in \cite{TateRigidAnalyticSpaces}, Tate algebras can be thought of as the regular functions on the closed unit polydisc in the context of non-Archimedean geometry. Precisely, associated to a non-Archimedean field $(k, |\cdot|)$, the Tate algebra $T_n(k)$ is the ring of restricted power series in $n \geq 1$ variables -- those series whose coefficients' norms go to zero. The Tate algebra $T_n(k)$ is an excellent regular ring \cite{Kiehl-Tate}, but generally not local. Affinoid algebras over $k$ are quotients of $T_n(k)$ and form the local models of classical rigid analytic spaces. For more on Tate algebras, see \cite{Boschrigid} and \autoref{subsec.BackgroundOnTate}. A Tate algebra $T_n(k)$ is $F$-finite if and only if $k$ is \cite[Lemma 3.3.3]{DattaMurayamaFsolidity}
but we are not aware of a method to pass to the $F$-finite case (such as an analog of the $\Gamma$-construction of \cite{HochsterHunekeFRegularityTestElementsBaseChange}) that works in this setting.
As such, at present we do not know an alternate argument to exhibit test elements for reduced affinoid algebras.
Moreover, our results on Tate algebras are in fact stronger than stated above: we are able to show that for any non-Archimedean field $k$, $T_n(k) \to T_n(\widehat{\overline{k}})$ satisfies the intersection flatness condition (\autoref{thm:IF-completion-algebraic-closure}), the property analogous to Frobenius intersection flatness for arbitrary ring maps (\autoref{def:IF}). As a consequence, for any algebraic extension $k \subseteq \ell$ of non-Archimedean fields of arbitrary characteristic, it follows that $T_n(k) \to T_n(\ell)$ is intersection flat (\autoref{cor:IF-algebraic-extensions}).

The first, second and fourth authors have also previously studied variations on the intersection flatness condition in \cite{DattaEpsteinTucker}. Building on this earlier work, in this article we systematically develop these notions in the context of the Frobenius endomorphism, giving rise to two variants of Frobenius intersection flatness. For the first of these, we say that a regular ring $S$ of prime characteristic $p > 0$ is \emph{Frobenius Ohm-Rush} (FOR) provided only that Frobenius expansions of collections of ideals of $S$ commute with arbitrary intersections -- so that \eqref{eq:foridealcontition} must hold, but \eqref{eq:fifsubmodcondition} may not. Thus, Frobenius intersection flat rings are automatically FOR, but the FOR condition is substantially weaker in general (see \autoref{thm:FORT-F-intersection-flat} \ref{thm:FORT-F-intersection-flat.4}). Nonetheless, FOR rings  satisfy a number of desirable properties. They were used in \cite{KatzmanLyubeznikZhangOnDiscretenessAndRationality}, and in forthcoming work \cite{DESTPhantom} the authors will show that reduced quotients of FOR excellent regular rings have test elements (generalizing the result from \cite{SharpBigTestElements} used above). Moreover, FOR rings feature prominently in the next main result of this article which exhibits new cases of the openness of the $F$-pure locus.

Recall that a ring $R$ of characteristic $p>0$ is said to be $F$-pure provided the Frobenius map $F \colon R \to F_* R$ is pure (also called universally injective \cite[\href{https://stacks.math.columbia.edu/tag/058H}{Tag 058H}]{stacks-project}) as a map of $R$-modules. Importantly, \mbox{$F$-purity} of Noetherian rings is closely related to the notion of log canonical singularities in characteristic zero \cite{HaraWatanabeFRegFPure}. When $R$ is $F$-finite, $R$ is $F$-pure if and only if it is $F$-split, from which it is straightforward to see that the $F$-pure locus of $R$ is open. Using the $\Gamma$-construction to reduce to the $F$-finite case, \cite[Corollary 3.5]{Murayama:TheGammaConstructionAndAsymptotic} showed that the $F$-pure locus of $R$ is open provided it is of finite type over an excellent local ring (or even a $G$-ring). More recently, \cite[Cor. 7.15]{HochsterYaoGenericLocalDuality} establishes that the $F$-pure locus of an $S_2$ ring that is a homomorphic image of an excellent Cohen-Macaulay ring is open.  Soon after, Lyu proved the result in the more general setting that $R/P$ is $J$-0 for all primes, \cite[Theorem A.2.4]{LyuUniformBounds}.
While our result is not as general as Lyu's (which in turn relies on \cite{HochsterYaoGenericLocalDuality}), the proof is different and may be of independent interest. In fact, our result originally appeared in the earlier arXiv versions of \cite{DattaEpsteinTucker} from 2023.  


\begin{theoremB*}[{\autoref{thm:FPureLocusOpenGeneral}}]
    Suppose $S$ is a regular ring of prime characteristic $p > 0$ and $R = S/I$.  Suppose that either: (1) $S$ is FOR, or (2) the regular locus of $\Spec S/\frq$ contains a non-empty open set for each $\frq \in \Spec S$ (for instance, if $S$ is excellent).
    Then the $F$-pure locus of $\Spec R$ is open.  
\end{theoremB*}
\noindent
\cite[Thm.\ 3.4]{EnescuYaoLowerSemicont} proves a result similar to part (2) of Theorem B but with more restrictive hypotheses; see \autoref{rem:EnescuYaoSemicont}. More generally, for any fixed $x \in R$, we show the locus of prime ideals of $R$ at which $R \xrightarrow{1 \mapsto F^e_* x} F^e_* R$ is pure is an open set under the above assumptions on the ambient ring $S$ (\autoref{thm:opennessforquotientsofregular}). Building on the work of \cite{HochsterYaoGenericLocalDuality}, \cite[Thm. A.2.4]{LyuUniformBounds} recovers \autoref{thm:opennessforquotientsofregular} when $S$ satisfies condition (2) of Theorem~B. Lyu's result implies in particular that all excellent rings of prime characteristic have open $F$-pure locus. Nevertheless, our proof of \autoref{thm:opennessforquotientsofregular} is simpler than that of \cite{HochsterYaoGenericLocalDuality,LyuUniformBounds} in the setting where these results overlap; see also \autoref{rem:openness-results-timeline}. The study of maps $R \xrightarrow{1 \mapsto F^e_* x} F^e_* R$ leads us to compare various notions of $F$-regularity outside the $F$-finite setting, see \autoref{subsec:variant-strong-F-regular} and \autoref{subsec:Artin-Rees}. In addition, the openness of pure loci of the maps  $R \xrightarrow{1 \mapsto F^e_* x} F^e_* R$ and the local-to-global techniques developed in \cite{DattaEpsteinTucker} imply that all excellent regular rings of prime characteristic are close to being FOR in the following sense.

\begin{theoremC*}[\autoref{thm:FOR-upto-radical}]
    Let $S$ be an excellent regular ring of prime characteristic $p > 0$. Then for any collection of radical ideals $\{J_i\}$ of $S$, we have $\bigcap_i J_i^{[p]} = (\bigcap_i J_i)^{[p]}$.
\end{theoremC*}

Moving on to address the second variant on Frobenius intersection flatness stemming from \cite{DattaEpsteinTucker}, again let $S$ be a regular ring of prime characteristic $p > 0$. Recall that the trace ideal of $F_*x \in F_*S$ is the image $\Tr_{F_* S}(F_*x)$ of $\Hom_S(F_* S, S) \xrightarrow{\text{eval@}F_*x} S$. We say that $S$ is \emph{Frobenius Ohm-Rush trace} (FORT) if for each $F_* x \in F_* S$, we have that 
\[
    F_* x \in \Tr_{F_* S}(F_*x) \cdot F_* S.
\]
FORT rings are Frobenius intersection flat by \cite[Proposition 4.3.8]{DattaEpsteinTucker}, but the FORT condition is substantially stronger in general. Indeed, note that FORT rings are necessarily $F$-split (as can be seen by taking $x = 1$).  However, in \cite{DattaMurayamaTate}, the first author and Murayama give an example of a non-Archimedean field $k$ 
where $\Hom_{T_n(k)}(F_* T_n(k), T_n(k)) = 0$ and hence $T_n(k)$ is then Frobenius intersection flat but not $F$-split or FORT. Returning to the topic of Tate algebras, we can also exhibit cases where they satisfy the stronger FORT condition.

\begin{theoremD*}[\autoref{cor:Tate-char-p-FIF-FORT}]
    Suppose $k$ is a non-Archimedean field of characteristic $p > 0$ such that either: (1) $k$ is spherically complete, or (2) $F_* k$ has a dense $k$-subspace with a countable basis.
    Then $T_n(k)$ is FORT. 
\end{theoremD*}

\noindent
Note that Theorem D is both important in its own right and is also central to the proof of Theorem~A. Moreover, once again our arguments yield yet stronger results still. For any extension of non-Archimedean fields $k \subseteq \ell$ with either $k$ spherically complete or where $\ell$ has a dense subspace with a countable $k$-basis, we will see that $T_n(k) \to T_n(\ell)$ satisfies a trace property analogous to FORT.

We have so far highlighted results on the variants of Frobenius intersection flatness for non-local Noetherian rings of prime characteristic. Our final feature result is that we can give a precise characterization of Frobenius intersection flatness in the local setting.

\begin{theoremE*}[\autoref{thm:FORT-F-intersection-flat}]
    Let $(R,\fm)$ be a regular local ring of prime characteristic $p > 0$. Let $F_{\widehat{R}/R} \colon F_*R \otimes_R \widehat{R} \to F_*\widehat{R}$ denote the relative Frobenius of the completion map $R \to \widehat{R}$. Then $R$ is Frobenius intersection flat if and only if $F_{\widehat{R}/R}$ is a pure map of $\widehat{R}$-modules.
\end{theoremE*}

\noindent
Recall that the behavior of $F_{\widehat{R}/R}$ controls how far $R$ is from being excellent. Indeed, as a consequence of work of Radu and Andr{\'e} \cite{RaduUneClasseDAnneaux,AndreHomomorphismsRegulariers} it follows that in the situation of Theorem E, $(R,\fm)$ is excellent if and only if $F_{\widehat{R}/R}$ is a faithfully flat ring map. However, faithful flatness of $F_{\widehat{R}/R}$ appears to be stronger than its $\widehat{R}$-purity. Therefore, Theorem E allows one to view the Frobenius intersection flatness condition as being related to, and potentially a weakening of, the notion of excellence for prime characteristic regular local rings. For DVRs, i.e. regular local rings of dimension $1$, Theorem E in fact provides a new characterization of the excellence property because we deduce that Frobenius intersection flatness of a prime characteristic DVR is equivalent to its excellence (\autoref{thm:FORT-F-intersection-flat}).  Moreover, as another consequence of Theorem E and descent properties of Frobenius intersection flatness
    (\autoref{subsec:Descent-FIF}), we can also show that if $(R,\fram)$ has geometrically regular formal fibers and $S$ is a regular ring that is essentially of finite type over $R$, then $S$ is Frobenius intersection flat (\autoref{thm:descent-F-intersection-flat}). This establishes another new case of Frobenius intersection flatness that to the best of our knowledge has not appeared in the literature.

\subsection*{History}

Large parts of this paper originally appeared in arXiv versions 1 and 2 of the paper \cite{DattaEpsteinTucker} which was originally posted to the arXiv in 2023, by the first, second and fourth authors.  In particular, many of the main results of \autoref{sec:Frobenius Ohm-Rush trace} and \autoref{sec:openness-of-pure-loci} originally appeared there.  Those sections were moved to this paper to make the topics of the two papers more coherent, with this paper being focused on characteristic $p > 0$ while relying on the characteristic-free results of \cite{DattaEpsteinTucker}.

\section{Preliminaries}
We begin by collecting some notation.

\subsection{Conventions and abbreviations}

All rings in this paper are commutative.
We will specify whenever Noetherian hypothesis is needed in the statements of our results.
When we use the term `regular ring' it is implicit that such rings are Noetherian. Additionally, for us a `local ring' is not necessarily Noetherian. However, when we say a local ring $(R, \fram)$ is `complete', we mean it is Noetherian and $\fram$-adically complete.

The following abbreviations are used freely in the text.
\begin{enumerate}
    \item ORT for Ohm-Rush trace (\autoref{def:ORT}),
    \item FOR for Frobenius Ohm-Rush (\autoref{def:FORT}),
    \item FORT for Frobenius Ohm-Rush trace (\autoref{def:FORT}).
\end{enumerate}

\subsection{Ohm-Rush, Ohm-Rush trace, and intersection flatness}
\label{subsec.OR-ORT-IF}

We briefly recall the definitions of Ohm-Rush, Ohm-Rush trace modules, and intersection flatness for an $R$-module $M$, and refer the reader to \cite{DattaEpsteinTucker} for a detailed discussion.  For the purposes of this article, we will primarily be interested in the case that $M = F^e_* R$.

\begin{definition}
\label{def:OR}
    For an $R$-module $M$, the \emph{content} of $x \in M$ is 
    \[
     c_M(x) := \bigcap_{x \in IM} I
    \]
    and gives rise to the \emph{content function (on $M$)}  $c_M : M \to \{\text{ideals of $R$}\}$. $M$ is said to be \emph{Ohm-Rush} if $x \in c_M(x)M$ for all $x \in M$. Equivalently, $M$ is Ohm-Rush if for all $x \in M$, the set of $I$ such that $x \in IM$ has a unique smallest element.  
\end{definition}

\begin{definition}
\label{def:ORT}
    For an $R$-module $M$, the \emph{trace} of $x\in M$ is
\[
    \Tr_M(x) := \Image\big( \Hom_R(M, R) \xrightarrow{f \mapsto f(x)} R\big).
\]
     $M$ is called \emph{Ohm-Rush trace} (or \emph{ORT}) if $x \in \Tr_M(x) M$ for all $x \in M$.
\end{definition}

\begin{definition}
    \label{def:IF}
    For $R$-modules $L$ and $M$ and a submodule $U \subseteq L$, we denote by $UM$ the image of $U \otimes_R M$ in $L \otimes_R M$. 
    An $R$-module $M$ is called \emph{intersection flat} if for any finitely generated $R$-module $L$ and any collection of submodules $\{U_i\}_i$ of $L$, we have that 
    \[
        \Big( \bigcap_i U_i \Big) M = \bigcap_i \big( U_i M \big).
    \]
\end{definition}

Note that an intersection flat module is in particular flat by \cite[Proposition 5.5]{HochsterJeffriesintflatness}. We have the following implications between the above notions.
\begin{equation}
\label{eq:ORTImpliesIFImpliesOR}
    \text{(ORT)} \Rightarrow \text{(intersection flat)} \Rightarrow \text{(Ohm-Rush and flat)}
\end{equation}
See \cite[Proposition 4.3.8]{DattaEpsteinTucker} and \cite[Remark 4.2.3(a)]{DattaEpsteinTucker} (see also \cite{rg71,OhmRu-content}); moreover, all of the above implications are strict (see \autoref{thm:FORT-F-intersection-flat}). In case $(R,\fram)$ is a complete local ring, all three coincide \cite[Theorem 4.3.12]{DattaEpsteinTucker}.

\begin{remark} 
These notions have other characterizations that we will not need.  
Notably, a flat module is intersection flat if and only if it is  Mittag-Leffler, see \cite[Theorem 4.3.1]{DattaEpsteinTucker}.
Furthermore, a flat module is ORT if and only if it is strictly Mittag-Leffler \cite[Part II, Prop. 2.3.4]{rg71}.
\end{remark}

\subsection{Pure maps}  Recall  that given a ring $R$, a map of $R$-modules $M \to N$ is \emph{(cyclically) pure} if for all (cyclic) $R$-modules $P$, the induced map $M \otimes_R P \to N \otimes_R P$ is injective. A ring homomorphism $R \to S$ is \emph{(cyclically) pure} if it is (cyclically) pure as map of $R$-modules. If $M$ is a submodule of an $R$-module $N$ such that the inclusion $M \hookrightarrow N$ is pure, then we will often say that $M$ is \emph{pure in $N$}. 

Pure maps of modules are also called \emph{universally injective} maps of modules in the literature. While the latter terminology is more descriptive, we will primarily use the former terminology for its brevity.

\subsection{Weak versions of the excellence property}

While Noetherian rings in general can be badly behaved (the singular loci may not be closed), the excellence property guarantees good geometric behavior, see \cite[\href{https://stacks.math.columbia.edu/tag/07QS}{Tag 07QS}]{stacks-project}.  We will need  weak versions of excellence, namely the J-0 and J-1 properties that we now recall.  
\begin{definition}
    \label{def:subexcellent}
    Let $A$ be a Noetherian ring. We say that $A$ is \emph{J-0} if the regular locus of $A$ contains a
        non-empty open set of $\Spec(A)$; $A$ is \emph{J-1} if the regular locus of 
        $A$ is open in $\Spec(A)$.
\end{definition}

We record a property of J-0 and J-1 rings that will be relevant for us in \autoref{sec:openness-of-pure-loci}.

\begin{proposition}
    \label{prop:J0-J1-equivalence}
    Let $R$ be a Noetherian ring. Then the following are equivalent:
    \begin{enumerate}[label=\textnormal{(\arabic*)}]
        \item For all $\p \in \Spec(R)$, $R/\p$ is J-0.\label{prop:J0-J1-equivalence.1}
        \item For all $\p \in \Spec(R)$, $R/\p$ is J-1.\label{prop:J0-J1-equivalence.2} 
    \end{enumerate}
    Moreover, the equivalent conditions imply that $R$ is J-1.
\end{proposition}

\begin{proof}
    Since $R/\p$ is a domain, its regular locus is non-empty. Thus \ref{prop:J0-J1-equivalence.2}$\implies$\ref{prop:J0-J1-equivalence.1}. Conversely, suppose \ref{prop:J0-J1-equivalence.1} holds. Since every prime ideal $\mathfrak{P}$ of $R/\p$ is of the form $\bq/\p$ for a unique prime ideal $\bq$ of $R$ containing $\p$, we then get that $(R/\p)/\mathfrak{P} \cong R/\bq$ is J-0 for all prime ideals $\mathfrak{P}$ of $R/\p$. Then $R/\p$ is J-1 by \cite[\href{https://stacks.math.columbia.edu/tag/07P9}{Tag 07P9}]{stacks-project}, that is, \ref{prop:J0-J1-equivalence.1}$\implies$\ref{prop:J0-J1-equivalence.2}. The same result also implies that $R$ itself is J-1 if the equivalent conditions \ref{prop:J0-J1-equivalence.1} or \ref{prop:J0-J1-equivalence.2} hold.
\end{proof}

\subsection{The relative Frobenius} If $\varphi \colon R \to S$ is a homomorphism
of rings of prime characteristic $p > 0$, then for every integer $e \geq 0$, consider the co-Cartesian diagram
  \[
    \begin{tikzcd}[column sep=4em]
      R \rar{F_R^e}\dar[swap]{\varphi} & F^e_{R*}R
      \arrow[bend left=30]{ddr}{F^e_{R*}\varphi}
      \dar{\id_{F^e_{R*}R} \otimes_R \varphi }\\
      S \rar{F^e_R \otimes_R \id_S} \arrow[bend right=12,end
      anchor=west]{drr}[swap]{F^e_S} & F^e_{R*}R \otimes_R S
      \arrow[dashed]{dr}[description]{F^e_{S/R}}\\
      & & F^e_{S*}S,
    \end{tikzcd}
    \]
    where $F^e_R \colon R \to F^e_{R*}R$ (resp. $F^e_S \colon S \to F^e_{S*}S$)
    denote the $e$-th iterate of the Frobenius endomorphism on $R$ (resp. on $S$). 
    The \emph{$e$-th relative Frobenius homomorphism} associated to $\varphi$ is the ring homomorphism
  \[
    \begin{tikzcd}[column sep=1.475em,row sep=0]
      \mathllap{F^e_{S/R}\colon} F^e_{R*}R \otimes_R S \rar & F^e_{S*}S\\
      F^e_{R*}r \otimes s \rar[mapsto] & F^e_{S*}\varphi(r)s^{p^e}
    \end{tikzcd}
  \]
Here if $r \in R$, we denote the corresponding element of $F^e_*R$ by $F^e_*r$. If $e = 1$, we denote $F^1_{S/R}$ by $F_{S/R}$.

\begin{notation}
From now on, for a ring $R$ of prime characteristic $p > 0$, by $F^e_*R$ we will always 
mean $F^e_{R*} R$ for ease of notation. 
\end{notation}

The relative Frobenius map $F_{S/R}$ detects geometric properties of the fibers of
$\varphi \colon R \to S$ when $R$ and $S$ are Noetherian. This is summarized in the 
next result.

\begin{theorem}
\label{thm:Radu-Andre-Dumitrescu}
Let $\varphi \colon R \to S$ be a flat homomorphism of Noetherian rings of 
prime characteristic $p > 0$. Then we have the following:
\begin{enumerate}[label=\textnormal{(\arabic*)}]
    \item The fibers of $\varphi$ are geometrically regular (i.e. $\varphi$ is a 
    regular map) if and only if $F_{S/R}$ is flat.
    \label{thm:Radu-Andre-Dumitrescu.1}
    \item The fibers of $\varphi$ are geometrically reduced (i.e. $\varphi$ is a 
    reduced map) if and only if $F_{S/R}$ is pure as a map of $F_*R$-modules.
    \label{thm:Radu-Andre-Dumitrescu.2}
\end{enumerate}
\end{theorem}

\begin{proof}[Indication of proof of \autoref{thm:Radu-Andre-Dumitrescu}]
   \ref{thm:Radu-Andre-Dumitrescu.1} follows by \cite[Thm.\ 4]{RaduUneClasseDAnneaux} and
    \cite[Thm.\ 1]{AndreHomomorphismsRegulariers} while \ref{thm:Radu-Andre-Dumitrescu.2} follows by 
    \cite[Thm.\ 3]{DumitrescuReduceness}.
\end{proof}

Let $(R, \fm)$ be a Noetherian local ring of prime characteristic $p > 0$ and let $\widehat{R}$ denote the
$\fm$-adic completion of $R$. We will now discuss the structure of the relative Frobenius
$F_{\widehat{R}/R}$ associated with the canonical map $R \to \widehat{R}$.

\begin{proposition}
    \label{prop:rel-Frob-completion}
    Let $(R, \fm)$ be a Noetherian local ring of prime characteristic $p > 0$.
    Consider the relative Frobenius
    \[
    F_{\widehat{R}/R} \colon F_*R \otimes_R \widehat{R} \to F_*\widehat{R}.
    \]
    We have the following:
    \begin{enumerate}[label=\textnormal{(\arabic*)}]
        \item The ring $F_*R \otimes_R \widehat{R}$ is local\footnote{Recall that a local ring for us is just a ring with unique maximal ideal, which is not necessarily Noetherian.}whose maximal ideal $\eta$ is the expansion of the maximal ideal $F_*\fm$ of $F_*R$ along the ring  homomorphism $F_*R \to F_*R \otimes_R \widehat{R}$. In particular, $\eta$ is finitely generated.\label{prop:rel-Frob-completion.1}

        \item $\eta^{[p]} = \fm(F_*R \otimes_R \widehat{R})$, and so, the $\eta$-adic completion of $F_*R \otimes_R \widehat{R}$ coincides with the $\fm$-adic completion of $F_*R \otimes_R \widehat{R}$ and also with the $\fm\widehat{R}$-adic completion of $F_*R \otimes_R \widehat{R}$. \label{prop:rel-Frob-completion.2}

        \item The relative Frobenius $F_{\widehat{R}/R}$ can be identified with the canonical map from $F_*R \otimes_R \widehat{R}$ to its $\eta$-adic (or $\fm$-adic or $\fm\widehat{R}$-adic) completion. \label{prop:rel-Frob-completion.3}

        \item $R \to \widehat{R}$ has geometrically regular fibers if and only if $F_*R \otimes_R \widehat{R}$ is Noetherian. \label{prop:rel-Frob-completion.4}

        \item If $R$ has Krull dimension $1$ and $\widehat{R}$ is a domain, then $R \to \widehat{R}$ has geometrically regular fibers if and only if $F_{\widehat{R}/R}$ is injective.\label{prop:rel-Frob-completion.5}
    \end{enumerate}
\end{proposition}

\begin{proof}
    Consider the diagram
   \[
    \begin{tikzcd}[column sep=4em]
      R \rar{F_R}\dar[swap]{\varphi} & F_{*}R
      \arrow[bend left=30]{ddr}{F_{*}\varphi}
      \dar{\id_{F_*R} \otimes_R  \varphi }\\
      \widehat{R} \rar{F_R \otimes_R \id_{\widehat R}} \arrow[bend right=12,end
      anchor=west]{drr}[swap]{F_{\widehat{R}}} & F_{*}R \otimes_R \widehat{R}
      \arrow[dashed]{dr}[description]{F_{\widehat{R}/R}}\\
      & & F_{*}\widehat{R},
    \end{tikzcd}
    \]
    where $\varphi \colon R \to \widehat{R}$ denotes the canonical map.
    Then $F_{\widehat{R}/R}$ is a map of $R$-algebras, $F_*R$-algebras as well 
    as $\widehat{R}$-algebras.

    We omit the proofs of \ref{prop:rel-Frob-completion.1}, \ref{prop:rel-Frob-completion.2} and \ref{prop:rel-Frob-completion.3} as they are well known consequences of the definitions and the diagram above.
    
    \ref{prop:rel-Frob-completion.4} Suppose $F_*R \otimes_R \widehat{R}$ is Noetherian. By \ref{prop:rel-Frob-completion.3} $F_{\widehat{R}/R}$ is faithfully flat since the map from a Noetherian local ring
    to its completion at the maximal ideal is always faithfully flat. Then 
    $R \to \widehat{R}$ has geometrically regular fibers by the Radu-Andr\'e theorem
    (\autoref{thm:Radu-Andre-Dumitrescu}\ref{thm:Radu-Andre-Dumitrescu.1}. Conversely, suppose 
    $R \to \widehat{R}$ has geometrically regular fibers. By \autoref{thm:Radu-Andre-Dumitrescu}\ref{thm:Radu-Andre-Dumitrescu.1}, $F_{\widehat{R}/R}$ is faithfully flat. 
    Since $F_*\widehat{R}$ is a Noetherian
    ring, 
    $F_*R \otimes_R \widehat{R}$ is Noetherian by
    faithfully flat descent of the Noetherian property
    \cite[\href{https://stacks.math.columbia.edu/tag/033E}{Tag 033E}]{stacks-project}.
    
    \ref{prop:rel-Frob-completion.5} If $R \to \widehat{R}$ has geometrically regular fibers, then
    $F_{\widehat{R}/R}$ is injective since it is faithfully flat by \autoref{thm:Radu-Andre-Dumitrescu}\ref{thm:Radu-Andre-Dumitrescu.1}. Conversely, suppose $F_{\widehat{R}/R}$ is 
    injective. Since $\widehat{R}$ is a domain,  $F_*R \otimes_R \widehat{R}$, which can be identified with a subring of $F_*\widehat{R}$, is also a domain. By \ref{prop:rel-Frob-completion.1}, $F_*R \otimes_R \widehat{R}$ is local
    with finitely generated maximal ideal. Furthermore, since $\Spec(\widehat{R})$ and
    $\Spec(F_*R \otimes_R \widehat{R})$ are homeomorphic, we have
   $
    \dim(F_*R \otimes_R \widehat{R}) = \dim(\widehat{R}) = 1.
    $
    Thus, $F_*R \otimes_R \widehat{R}$ is a local domain of Krull dimension $1$. Therefore it has two prime ideals, the maximal ideal and the $(0)$ ideal,
    which are both finitely generated. Then $F_*R \otimes_R \widehat{R}$ is Noetherian by
    a theorem of Cohen \cite[\href{https://stacks.math.columbia.edu/tag/05KG}{Tag 05KG}]{stacks-project}. By \ref{prop:rel-Frob-completion.4} we then get that $R \to \widehat{R}$ has
    geometrically regular fibers.
\end{proof}

We can use \autoref{prop:rel-Frob-completion} to deduce the following Corollary.

\begin{corollary}
    \label{cor:rel-Frob-completions}
    Let $(R, \fm)$ be a Noetherian local ring of prime characteristic $p > 0$. The following are equivalent:
    \begin{enumerate}[label=\textnormal{(\arabic*)}]
        \item $R \to \widehat{R}$ has geometrically regular fibers.\label{cor:rel-Frob-completions.1}
        \item $F_{\widehat{R}/R}$ is a flat ring map.\label{cor:rel-Frob-completions.2}
        \item $F_{\widehat{R}/R}$ is a pure ring map.\label{cor:rel-Frob-completions.3}
    \end{enumerate}
\end{corollary}

\begin{proof}
    The equivalence of \ref{cor:rel-Frob-completions.1} and \ref{cor:rel-Frob-completions.2} is the Radu-Andr{\'e} theorem; see \autoref{thm:Radu-Andre-Dumitrescu}. The relative Frobenius is surjective on $\Spec$ \cite[\href{https://stacks.math.columbia.edu/tag/0BRA}{Tag 0BRA}]{stacks-project}. Thus \ref{cor:rel-Frob-completions.2} implies that $F_{\widehat{R}/R}$ is faithfully flat, and it is well-known that a faithfully flat ring map is also pure \cite[\href{https://stacks.math.columbia.edu/tag/05CK}{Tag 05CK}]{stacks-project}, that is, \ref{cor:rel-Frob-completions.2} $\implies$ \ref{cor:rel-Frob-completions.3}. Finally, assuming \ref{cor:rel-Frob-completions.3} we have that since $F_{*}\widehat{R}$ is a Noetherian ring, the ring $F_{*}R \otimes_R \widehat{R}$ is Noetherian because the property of being Noetherian descends along pure ring maps. Therefore \ref{cor:rel-Frob-completions.3} $\implies$ \ref{cor:rel-Frob-completions.1} by \autoref{prop:rel-Frob-completion}\ref{prop:rel-Frob-completion.4}.
\end{proof}

The following result that purity of certain maps can be checked after completion will also be useful to reduced to the complete local case.

\begin{lemma}
    \label{lem:F-purity-completions}
    Let $(R, \fm)$ be a Noetherian local ring of prime characteristic $p > 0$. Let $c \in R$. Then for an integer $e > 0$, consider the maps
    \[    \textrm{$\lambda^{R}_{c,e} : R \to F^e_{R*}R$ and  $\lambda^{\widehat{R}}_{c,e}: \widehat{R} \to F^e_{\widehat{R}*}\widehat{R}$},
    \]
    where in the first map $1 \mapsto F^e_{R*}c$ and in the second map $1 \mapsto F^e_{\widehat{R}*}c$ (by abuse of notation the image of $c$ in $\widehat{R}$ is also denoted as $c$). Then $\lambda^R_{c,e}$ is $R$-pure if and only
    if $\lambda^{\widehat R}_{c,e}$ is $\widehat{R}$-pure. As a consequence, $R$ is $F$-pure if and
    only if $\widehat R$ is $F$-pure.
\end{lemma}

\begin{proof}
The second assertion about $F$-purity follows from the first upon taking $c = 1$ and $e = 1$. So we prove the first assertion. Recall that if $M$ is a module over a Noetherian local ring $(R, \fm)$ then an $R$-linear
map
$
R \to M
$
is pure if and only if the induced map $E \to E \otimes_R M$ is injective, where $E = E_R(R/\fm)$ is and injective hull of the residue field of $R$. 
Moreover, $E$ naturally has the structure of a $\widehat{R}$-module because every element of $E$ is annihilated by a power of $\fm$, and with this module structure, we also have $E \cong E_{\widehat{R}}(\widehat{R}/\fm\widehat{R})$. 

Then by \cite[\href{https://stacks.math.columbia.edu/tag/0BNK}{Tag 0BNK}]{stacks-project}, for all
$\widehat{R}$-modules $M$ we get
\[
E \otimes_R M \cong E \otimes_R (\widehat{R} \otimes_{\widehat{R}} M) \cong (E \otimes_R \widehat{R}) \otimes_{\widehat{R}}M \cong E \otimes_{\widehat{R}} M.
\]
Now note that $\id_E \otimes_{\widehat R} \lambda^{\widehat R}_{c,e}$ can be expressed
as the composition
\[
E \xrightarrow{\id_E \otimes_R \lambda^{R}_{c,e}} E \otimes_R F^e_{R*}R \to E \otimes_R F^e_{\widehat{R}*}\widehat{R} \cong E \otimes_{\widehat{R}} F^e_{\widehat{R}*}\widehat{R}.
\]
The map $F^e_{R*}R \to F^e_{\widehat{R}*}\widehat{R}$ is a faithfully flat ring homomorphism, and hence pure as a map of $F^e_{R*}R$-modules. By restriction of scalars, $F^e_{R*}R \to F^e_{\widehat{R}*}\widehat{R}$ is then also pure as a map of
$R$-modules. Thus,
$
E \otimes_R F^e_{R*}R \to E \otimes_R F^e_{\widehat{R}*}\widehat{R}
$
is always an injective map of $R$-modules. It then follows that $\id_E \otimes_{\widehat{R}}\lambda^{\widehat{R}}_{c,e}$ is injective if and only if 
$\id_E \otimes_R \lambda^{R}_{c,e}$ is injective. Equivalently, this shows that
$\lambda^{R}_{c,e}$ is ${R}$-pure if and only if $\lambda^{\widehat{R}}_{c,e}$ is $\widehat{R}$-pure.
\end{proof}

\section{Frobenius Ohm-Rush (trace) and \emph{F}-intersection flatness}
\label{sec:Frobenius Ohm-Rush trace}

\subsection{Definitions and first properties}
 Both the intersection flatness condition
and the Ohm Rush condition have been explored in the theory of $F$-singularities
in the past \cite{HochsterHunekeFRegularityTestElementsBaseChange, BlickleMustataSmithDiscretenessAndRationalityOfFThresholds,KatzmanParameterTestIdealOfCMRings, KatzmanLyubeznikZhangOnDiscretenessAndRationality, SharpBigTestElements, HochsterJeffriesintflatness}. More recently, these notions were revisited and further developed in \cite{DattaEpsteinTucker}. Our goal now is to investigate this circle of ideas in the specific case of the Frobenius endomorphism
of a ring of prime characteristic.

\begin{definition}
\label{def:FORT}
     If $R$ is a ring of prime characteristic $p > 0$, we say that $R$ is 
     \begin{itemize}
     \item[$\bullet$] \emph{Frobenius Ohm-Rush trace} (abbrv. FORT) if $F_*R$ is
    an ORT $R$-module; 
    \item[$\bullet$] \emph{$F$-intersection
    flat} if $F_*R$ is an intersection flat $R$-module; 
    \item[$\bullet$] \emph{Frobenius
    Ohm-Rush} (abbrv. FOR) if $F_*R$ is an Ohm-Rush $R$-module.
    \end{itemize}
\end{definition}

\begin{lemma}
    \label{lem:FORT-FIF-FOR-iterated-Frobenius}
    Let $R$ be a ring of characteristic $p > 0$ and let $e \in \mathbb{Z}_{>0}$. We have the following:
    \begin{enumerate}[label=\textnormal{(\arabic*)}]
        \item $R$ is FORT $\implies F^e_*R$ is an ORT $R$-module.\label{lem:FORT-FIF-FOR-iterated-Frobenius.1} 
        \item $R$ is $F$-intersection flat $\implies F^e_*R$ is an intersection flat $R$-module.\label{lem:FORT-FIF-FOR-iterated-Frobenius.2}
        \item $R$ is FOR $\implies F^e_*R$ is an Ohm-Rush $R$-module.\label{lem:FORT-FIF-FOR-iterated-Frobenius.3}
    \end{enumerate}
\end{lemma}


\begin{proof}
    Using induction on $e$, \ref{lem:FORT-FIF-FOR-iterated-Frobenius.1} follows by \cite[Corollary 4.1.8]{DattaEpsteinTucker}, \ref{lem:FORT-FIF-FOR-iterated-Frobenius.2} follows by \cite[Remark 2.4.3 (d)]{DattaEpsteinTucker} and \ref{lem:FORT-FIF-FOR-iterated-Frobenius.3} follows by \cite[Corollary 3.4.9]{DattaEpsteinTucker}.
\end{proof}

We first provide alternate characterizations of the FORT and $F$-intersection flatness properties.  

\begin{theorem}
    \label{thm:FORT-F-intersection-flat}
    Let $R$ be a ring of prime characteristic $p > 0$. Then we have
    the following:
    \begin{enumerate}[label=\textnormal{(\arabic*)}]        
        \item \label{thm:FORT-F-intersection-flat.1} 
        Let $(R, \fm)$ be a Noetherian local ring. Let $\widehat{R}$ denote its $\fm$-adic completion and $F_{\widehat{R}/R}$ the relative Frobenius of the
        completion map $R \to \widehat{R}$. The following are equivalent:
         
        \begin{enumerate}[label=\textnormal{(1\alph*)}]
            \item \label{thm:FORT-F-intersection-flat.1a}  
            $R$ is $F$-intersection flat.
            
            \item \label{thm:FORT-F-intersection-flat.1b}  
            $F_*R$ is a flat $R$-module and $F_{\widehat{R}/R}$ is pure as a map of $\widehat{R}$-modules. 
            
            \item \label{thm:FORT-F-intersection-flat.1c}  
            $F_*R$ is a flat $R$-module and $F_{\widehat{R}/R}$
            is cyclically pure as a map of $\widehat{R}$-modules. 

            \item \label{thm:FORT-F-intersection-flat.1d}  
            $F_*R \otimes_R \widehat{R}$ is a flat Ohm-Rush $\widehat{R}$-module. 
            
            \item \label{thm:FORT-F-intersection-flat.1e} 
            $R$ is regular and for all cyclic $\widehat{R}$-modules $L$, $F_*R \otimes_R L$ is $\fm\widehat{R}$-adically (equivalently $\fm$-adically) separated. 
        \end{enumerate}

        \item \label{thm:FORT-F-intersection-flat.2} 
        If $R$ is an excellent regular local ring, then $R$ is 
        $F$-intersection flat.

        \item \label{thm:FORT-F-intersection-flat.3}
        Suppose $R$ is a regular local ring of Krull dimension $1$ (i.e. $R$ is a DVR). Then $R$ is $F$-intersection flat if and only if $R$ is excellent.

        \item \label{thm:FORT-F-intersection-flat.4}
        We have FORT $\implies$ $F$-intersection flat $\implies$ FOR. Moreover, all the implications are strict even for the class of regular local rings of Krull dimension $1$.

        \item \label{thm:FORT-F-intersection-flat.5}
        Suppose $R$ is local (but not necessarily Noetherian). Then we have the following:
        \begin{enumerate}[label=\textnormal{(5\alph*)}]
            \item \label{thm:FORT-F-intersection-flat.5a}
            Assume $F_*R$ is a flat $R$-module. Then $R$ is FOR if and only if for any $c \in R$ there exists a finitely generated free $R$-submodule $L$ of $F_*R$ containing $F_*c$ such that $L \hookrightarrow F_*R$ is pure as a map of $R$-modules.

            \item \label{thm:FORT-F-intersection-flat.5b}
            $R$ is $F$-intersection flat if and only if for any finitely generated submodule $M$ of $F_*R$, there exists a finitely generated free $R$-submodule $L$ of $F_*R$ containing $M$ such that $L \hookrightarrow F_*R$ is pure as $R$-modules. 

            \item \label{thm:FORT-F-intersection-flat.5c}
            $R$ is FORT if and only if for any finitely generated submodule $M$ of $F_*R$, there exists a finitely generated free $R$-submodule $L$ of $F_*R$ containing $M$ such that $L$ is a direct summand of $F_*R$.

        \end{enumerate}
        \item  \label{thm:FORT-F-intersection-flat.6} 
        Let $(R,\fm)$ be an $\fm$-adically complete Noetherian regular local ring. Then $R$ is FORT.

        \item \label{thm:FORT-F-intersection-flat.7} If $R$ is $F$-intersection flat and $x_1,\dots,x_n$ are indeterminates, then for any multiplicative set $W$ of $R[x_1,\dots,x_n]$, we have that $W^{-1}(R[x_1,\dots,x_n])$ is $F$-intersection flat.
    \end{enumerate}
\end{theorem}

Note that \ref{thm:FORT-F-intersection-flat.2} improves on \cite[Theorem 4.1]{KatzmanLyubeznikZhangOnDiscretenessAndRationality}, and \ref{thm:FORT-F-intersection-flat.6} improves on \cite[Proposition 5.3]{KatzmanParameterTestIdealOfCMRings}.

\begin{proof}
  \ref{thm:FORT-F-intersection-flat.1} 
  Recall that the relative Frobenius $F_{\widehat{R}/R} \colon F_*R \otimes_R \widehat{R} \to F_*\widehat{R}$  
  can be identified with the canonical map from $F_*R \otimes_R \widehat{R}$ to its $\fm\widehat{R}$-adic completion by \autoref{prop:rel-Frob-completion}~\ref{prop:rel-Frob-completion.3}.
  
  We first prove the equivalence of \ref{thm:FORT-F-intersection-flat.1a}  and \ref{thm:FORT-F-intersection-flat.1b}. Note that $(R, \fm)$ is $F$-intersection flat if and only if 
  $F_*R \otimes_R \widehat{R}$ is an intersection flat $\widehat{R}$-module by base change (see \cite[Theorem 4.3.1]{DattaEpsteinTucker}) and pure descent of intersection flatness (see \cite[Corollary 4.3.2]{DattaEpsteinTucker}). 
  Since $\widehat{R}$ is complete, \cite[Theorem 4.3.12]{DattaEpsteinTucker} implies that the intersection flatness of $F_*R \otimes_R \widehat{R}$ is equivalent to the $\widehat{R}$-purity of the  the canonical map
  from $F_*R \otimes_R \widehat{R}$ to the $\fm\widehat{R}$-adic completion
  of $F_*R \otimes_R \widehat{R}$
  . 
  But the latter canonical map can be identified with the relative Frobenius
  $F_{\widehat{R}/R}$. 
  
  The implication $\ref{thm:FORT-F-intersection-flat.1b} \implies \ref{thm:FORT-F-intersection-flat.1c}$ follows because purity is a stronger 
  condition than cyclic purity. For $\ref{thm:FORT-F-intersection-flat.1c} \implies \ref{thm:FORT-F-intersection-flat.1b}$, the hypothesis that
  $F_*R$ is a flat $R$-module implies that $R$ is regular \cite{KunzCharacterizationsOfRegularLocalRings}. Thus, $\widehat{R}$ is also a 
  regular ring, and $F_*\widehat{R}$ is a flat $\widehat{R}$-module. Then
  $\widehat{R}$-cyclic purity of $F_{\widehat{R}/R} \colon F_*R \otimes_R \widehat{R} \to F_*\widehat{R}$ implies $\widehat{R}$-purity by 
  \cite[Lemma 3.3.1]{DattaEpsteinTucker}. Thus \ref{thm:FORT-F-intersection-flat.1b} and \ref{thm:FORT-F-intersection-flat.1c} are equivalent.

  For $\ref{thm:FORT-F-intersection-flat.1c} \implies \ref{thm:FORT-F-intersection-flat.1d}$ note that $F_*R \otimes_R \widehat{R}$ is $\widehat R$-flat by base change. Then $F_*R \otimes_R \widehat{R}$ is an Ohm-Rush $\widehat R$-module by \cite[Theorem 4.3.12, $(6) \implies (5)$]{DattaEpsteinTucker}. 
  Conversely, if $F_*R \otimes_R \widehat{R}$ is a flat $\widehat{R}$-module, then by faithfully flat descent of flatness, $F_*R$ is a flat $R$-module. Furthermore, $F_{\widehat{R}/R}$ is a cyclically pure map of $\widehat R$-modules by \cite[Theorem 4.3.12, $(5) \implies (6)$]{DattaEpsteinTucker} 
  because $F_*R \otimes_R \widehat{R}$ is a flat Ohm-Rush $\widehat R$-module. Thus, \ref{thm:FORT-F-intersection-flat.1c} and \ref{thm:FORT-F-intersection-flat.1d} are equivalent.

  Finally, the equivalence of 
  \ref{thm:FORT-F-intersection-flat.1d} and \ref{thm:FORT-F-intersection-flat.1e}
  follows by \cite[Theorem 4.3.12, $(5) \Longleftrightarrow (10)$]{DattaEpsteinTucker}. This completes the proof of \ref{thm:FORT-F-intersection-flat.1}.

  \ref{thm:FORT-F-intersection-flat.2} 
  Since $R$ is excellent and local, the
  relative Frobenius $F_{\widehat{R}/R}$ is a faithfully flat ring homomorphism by the Radu-Andr\'e theorem; see \autoref{thm:Radu-Andre-Dumitrescu}.
  Since faithfully flat ring maps are pure, $F_{\widehat{R}/R}$ is also pure ring homomorphism.
  By restriction of scalars, it follows that $F_{\widehat{R}/R}$ is then also pure as a $\widehat{R}$-algebra homomorphism. We then get that
  $R$ is $F$-intersection flat by \ref{thm:FORT-F-intersection-flat.1}.

  \ref{thm:FORT-F-intersection-flat.3} 
  The implication $\Leftarrow$ follows by \ref{thm:FORT-F-intersection-flat.2}. So assume $(R,\fm)$ is a DVR which is $F$-intersection flat. Since $F_*R$ is a flat $R$-module by Kunz's theorem \cite{KunzCharacterizationsOfRegularLocalRings}, we get $F_{\widehat{R}/R}$ is $\widehat{R}$-pure by \ref{thm:FORT-F-intersection-flat.1} and hence an injective map. Then $R \to \widehat{R}$ has geometrically regular fibers by \autoref{prop:rel-Frob-completion}~\ref{prop:rel-Frob-completion.5}.  
  Consequently, $R$ is excellent by \cite[Prop.\ 5.5.1(ii)]{GabberWeakUniformization}
  as a Cohen-Macaulay ring is always universally catenary.

  \ref{thm:FORT-F-intersection-flat.4} 
  The implications hold due to \autoref{eq:ORTImpliesIFImpliesOR}.
  The first implication is not an equivalence because there exists
  an excellent Henselian regular local ring $(R, \fm)$ of Krull dimension $1$ such that $\Hom_R(F_*R, R) = 0$ \cite{DattaMurayamaFsolidity}. Such a ring $R$ is $F$-intersection flat by \ref{thm:FORT-F-intersection-flat.2}, but not FORT. This is because the FORT property implies that the canonical map $F_*R \to \Hom_R(\Hom_R(F_*R,R),R)$ is 
  cyclically pure (\cite[Lemma 4.1.3]{DattaEpsteinTucker}) 
  and 
  hence injective, which is impossible
  if $\Hom_R(F_*R,R) = 0$. 

  To show the second implication is strict, suppose $R$ is a DVR of characteristic $p > 0$. 
  Then any $\fm$-adically separated $R$-module $M$ is Ohm-Rush by \cite[Prop.\ 2.1]{OhmRu-content}.
    Taking $M = F_*R$, we then get that $R$ is FOR. However, we have seen in \ref{thm:FORT-F-intersection-flat.3} that $R$ is $F$-intersection flat precisely when $R$ is excellent. Hence any non-excellent DVR of prime characteristic $p > 0$ is FOR but not $F$-intersection flat. For an abundance of examples of non-excellent DVRs of prime characteristic, even in the function field of $\mathbb{P}^2$, the reader can consult \cite{DattaSmithExcellence}.

    \ref{thm:FORT-F-intersection-flat.5a} follows by 
    \cite[Corollary 3.4.32~(2)]{DattaEpsteinTucker}. 
    Since $F$-intersection flatness is equivalent to $F_*R$ being flat and Mittag-Leffler by \cite[Theorem 4.3.1]{DattaEpsteinTucker}, 
    we get \ref{thm:FORT-F-intersection-flat.5b}  
    by 
    \cite[Proposition 3.2.6~(1)]{DattaEpsteinTucker}.
    Since FORT is equivalent to $F_*R$ being flat and strictly Mittag-Leffler by \cite[Part II, Prop. 2.3.4]{rg71}, 
    \ref{thm:FORT-F-intersection-flat.5c}
    follows by 
    \cite[Proposition 3.2.6~(b)]{DattaEpsteinTucker}.

    \ref{thm:FORT-F-intersection-flat.6} The $\fm$-adic completion of $F_*R$ coincides with $F_*\widehat{R}$. Since $R$ is $\fm$-adically complete, it follows that $F_*R$ is $\fm$-adically complete. Moreover, $F_*R$ is a flat $R$-module because $R$ is regular. Thus, \ref{thm:FORT-F-intersection-flat.6} follows by 
    \cite[Corollary 4.3.14]{DattaEpsteinTucker}.

    \ref{thm:FORT-F-intersection-flat.7} Since $R$ is $F$-intersection flat, $R$ is reduced (since Frobenius of $R$ is flat and hence injective). Thus, the Frobenius map on $R[x_1,\dots,x_n]$ can be expressed as the composition $R[x_1,\dots,x_n] \hookrightarrow R^{1/p}[x_1,\dots,x_n] \hookrightarrow R^{1/p}[x_1^{1/p},\dots,x_1^{1/p}]$. Since $R^{1/p}[x_1,\dots,x_n] \hookrightarrow R^{1/p}[x_1^{1/p},\dots,x_1^{1/p}]$ is a free extension, it is ORT, and hence, intersection flat by \cite[Proposition 4.3.8]{DattaEpsteinTucker}. Moreover, since $R \hookrightarrow R^{1/p}$ is intersection flat, $R[x_1,\dots,x_n] \hookrightarrow R^{1/p}[x_1,\dots,x_n]$ is intersection flat by base change (see \cite[Theorem 4.3.1]{DattaEpsteinTucker}). Thus, the composition is intersection flat by \cite[Remark 4.2.3]{DattaEpsteinTucker}.  That is, $R[x_1,\dots,x_n]$ is $F$-intersection flat. As a consequence, $W^{-1}(R[x_1,\dots,x_n])$ is $F$-intersection flat again by base change (here we use that localization commutes with $F_*$).
\end{proof}

\begin{example}
    \label{ex:OR-not-preserved-base-change-completion}
    Let $R$ be a non-excellent DVR of characteristic $p > 0$. Consider the $R$-module $F_*R$. Then $F_*R$ is an Ohm-Rush $R$-module but $F_*R$ is not an intersection flat $R$-module, as shown in the proof of \autoref{thm:FORT-F-intersection-flat}\ref{thm:FORT-F-intersection-flat.4} above.
Thus, the flat $\widehat{R}$-module $F_*R \otimes_R \widehat{R}$ cannot be an Ohm-Rush $\widehat{R}$-module by $\ref{thm:FORT-F-intersection-flat.1a} \Longleftrightarrow \ref{thm:FORT-F-intersection-flat.1d}$ of \autoref{thm:FORT-F-intersection-flat}. In other words, this example demonstrates that the flat Ohm-Rush property is not preserved by base change along the completion of a local ring.
\end{example}

\begin{remark}
    \label{rem:base-change-Frobenius-equivalences}
    {\*}
    \begin{enumerate}
        \item One can add further equivalences to the list of statements in part \ref{thm:FORT-F-intersection-flat.1} of \autoref{thm:FORT-F-intersection-flat} using \cite{DattaEpsteinTucker}. Indeed, via \cite[Theorem 4.3.12]{DattaEpsteinTucker}, $F_*R \otimes_R \widehat{R}$ being a flat Ohm-Rush $\widehat{R}$-module is equivalent to it being flat and intersection flat (resp, Mittag-Leffler, strictly Mittag-Leffler, ORT) as a $\widehat{R}$-module.  

        \item Let $(R, \fm)$ be an excellent Henselian DVR of characteristic $p > 0$ such that $\Hom_R(F_*R,R)$ is trivial \cite{DattaMurayamaTate}. Then clearly, the flat $R$-module $F_*R$ is neither ORT nor, equivalently, is $F_*R$ strictly Mittag-Leffler as an $R$-module. However, the faithfully flat base change $F_*R \otimes_R \widehat{R}$ is Ohm-Rush, Mittag-Leffler, strictly Mittag-Leffler, ORT and intersection flat by (a) because $F_{\widehat{R}/R}$ is faithfully flat and hence $\widehat{R}$-pure by restriction of scalars. This shows that the properties of being ORT and strictly Mittag-Leffler do not satisfy pure/faithfully flat descent, unlike the intersection flatness \cite[Corollary 4.3.2]{DattaEpsteinTucker}, Ohm-Rush and Mittag-Leffler conditions \cite[Theorem 3.5.4]{DattaEpsteinTucker}.

        \item If $(R,\fm)$ is a Noetherian local ring of characteristic $p > 0$, then $R$ is $F$-intersection flat if and only if $F_*R$ is \emph{universally Ohm-Rush} in the sense that for all $R$-algebras $S$, $F_*R \otimes_R S$ is an Ohm-Rush $S$-algebra. This follows from \autoref{thm:FORT-F-intersection-flat}~\ref{thm:FORT-F-intersection-flat.1a}$\Longleftrightarrow$\ref{thm:FORT-F-intersection-flat.1d}.
    \end{enumerate}
\end{remark}

\begin{corollary}
    \label{cor:global-F-intersection-flatness}
    Let $R$ be a locally excellent (i.e. $R_\p$ is excellent for all prime ideals $\p$) Noetherian regular ring of prime characteristic $p > 0$. We have the following:
    \begin{enumerate}[label=\textnormal{(\arabic*)}]      
        \item $R$ is $F$-intersection flat if and only if for all free $R$-modules $L$ of finite rank and linear maps $\varphi \colon L \to F_*R$, the cyclically pure locus of $\varphi$ is open in $\Spec(R)$. \label{cor:global-F-intersection-flatness.1}
        \item Suppose $R$ is a domain. Then $R$ is $F$-intersection flat if and only if for all free $R$-modules $L$ of finite rank and injective linear maps $\varphi \colon L \to F_*R$, \textrm{$\{\p \in \Spec(R)\colon \varphi \otimes_R \id_{R/\p}$ is injective$\}$} is open in $\Spec(R)$.\label{cor:global-F-intersection-flatness.2}
    \end{enumerate}
\end{corollary}

\begin{proof}
    Since $R$ is regular, $F_*R$ is a flat $R$-module. Since $R$ is locally excellent, for all $\p \in \Spec(R)$, $R_\p$ is $F$-intersection flat by \autoref{thm:FORT-F-intersection-flat}. Thus, \ref{cor:global-F-intersection-flatness.1} and \ref{cor:global-F-intersection-flatness.2} both follow by \cite[Theorem 3.7.1]{DattaEpsteinTucker}.
\end{proof}

If $R$ is a FOR ring of characteristic $p > 0$, the condition that $F^e_*R$ is an Ohm-Rush $R$-module for all integers $e > 0$ is easily seen to be equivalent to the following assertion: for all ideals $\ba$ of $R$, the collection of ideals $I$ of $R$ such that $
\ba \subseteq I^{[p^e]}
$
has a smallest ideal (under inclusion). Following \cite{BlickleMustataSmithDiscretenessAndRationalityOfFThresholds}, who studied the existence of these smallest ideals in the regular $F$-finite setting, it is customary to denote this ideal by $\ba^{[1/p^e]}$. Note that $\ba^{[1/p^e]} = c_{F^e_*R}(F^e_*\ba)$, the content of $F^e_*\ba$. One can use properties of the content function to deduce standard properties of the $[1/p^e]$-operator. Here we isolate one property that follows by general results on flat Ohm-Rush modules. 

\begin{proposition}
    \label{prop:[1/p^e]-completion}
   Let $(R,\fm)$ be a regular local ring of prime characteristic $p > 0$. Let $\widehat{R}$ denote the $\fm$-adic completion of $R$. If $R$ is FOR, then for all integers $e > 0$ we have the following:
   \begin{enumerate}[label=\textnormal{(\arabic*)}] 
    \item For all ideals $\ba$ of $R$, $\ba^{[1/p^e]}\widehat{R} = (\ba\widehat{R})^{[1/p^e]}$. \label{prop:[1/p^e]-completion.1}
    \item For all ideals $\bb$ of $\widehat{R}$, $(\bb \cap R)^{[p^e]} = \bb^{[p^e]} \cap R$. \label{prop:[1/p^e]-completion.2}
   \end{enumerate}
\end{proposition}

\begin{proof}
    \ref{prop:[1/p^e]-completion.1} By \autoref{lem:FORT-FIF-FOR-iterated-Frobenius} we have that for all integers $e > 0$, $F^e_{R*}R$ is a flat Ohm-Rush $R$-module. Note that the $\fm$-adic completion of $F^e_{R*}R$ can be canonically identified with $F^e_{\widehat{R}*}\widehat{R}$. 
    
    Then applying \cite[Corollary 3.4.30]{DattaEpsteinTucker} to the $R$-algebra $F^e_*R$, we get
    \[
    \ba^{[1/p^e]}\widehat{R} = c_{F^e_{R*}R}(F^e_{R*}\ba)\widehat{R} = c_{F^e_*\widehat{R}}(F^e_{R*}\ba \cdot F^e_{\widehat{R}*}\widehat{R}) = c_{F^e_{\widehat{R}*}\widehat{R}}(F^e_{\widehat{R}*}(\ba\widehat{R})). 
    \]
    But $\widehat{R}$ is FOR from \autoref{thm:FORT-F-intersection-flat}, and so, $c_{F^e_{\widehat{R}*}\widehat{R}}(F^e_{\widehat{R}*}(\ba\widehat{R})) = (\ba\widehat{R})^{[1/p^e]}$.

    \ref{prop:[1/p^e]-completion.2} The inclusion $(\bb \cap R)^{[p^e]} \subseteq \bb^{[p^e]} \cap R$ is clear. Let $x \in \bb^{[p^e]} \cap R$. Let $\varphi \colon R \to \widehat{R}$ be the canonical map. Then $\varphi(x) \in \bb^{[p^e]}$, and so,
    $
    (xR)^{[1/p^e]}\widehat{R} \stackrel{\ref{prop:[1/p^e]-completion.1}}{=} (\varphi(x)\widehat{R})^{[1/p^e]} \subseteq \bb.
    $
    Contracting back to $R$ gives us
    $
    (xR)^{[1/p^e]} \subseteq \bb \cap R,    
    $
    and so,
    \[
    x \in xR \stackrel{\textrm{$R$ is FOR}}{\subseteq} ((xR)^{[1/p^e]})^{[p^e]} \subseteq (\bb \cap R)^{[p^e]}.
    \]
    This establishes $(\bb^{[p^e]} \cap R) \subseteq (\bb \cap R)^{[p^e]}$, which is the inclusion we wanted to show.
\end{proof}

\begin{remark}
The surprising aspect of \autoref{prop:[1/p^e]-completion} is that one does not need any niceness assumptions on the fibers of the completion map $R \to \widehat{R}$. This is in contrast with \cite[Lem.\ 6.6]{Lyubeznik-Smith-Test-Ideal}, which can also be used to deduce \autoref{prop:[1/p^e]-completion} but only in the excellent setting.
\end{remark}

\subsection{FORT rings} Of the three notions introduced in Definition \ref{def:FORT}, the FORT condition has not been previously explored in the $F$-singularity literature. We now discuss this notion in more detail, beginning with some
elementary observations.

\begin{remark}
\label{rem:FORT-properties}
{\*}
    \begin{enumerate}
        \item If $R$ is FORT, then $R$ is 
        Frobenius split. This follows by \cite[Remark 4.1.2]{DattaEpsteinTucker}.
        Thus, FORT rings are reduced.\label{rem:FORT-properties.a}
        
        \item \label{rem:FORT-properties.b}
        Unravelling the definition, we see that a reduced ring $R$ is FORT if and only if for all $x \in R$ we have
        \begin{equation}
            \label{eq:FORT-membership}
            F_*x \in \Tr_{F_*R}(F_*x)\cdot F_*R = F_*(\Tr_{F_*R}(F_*x)^{[p]}),
        \end{equation}
        where 
        $
            \Tr_{F_*R}(F_*x) = \sum_{\phi \in \Hom_R(F_*R,R)}\phi(F_*x)R.
        $
        The set membership in \autoref{eq:FORT-membership} is the same as saying $x \in \Tr_{F_*R}(F_*x)^{[p]}$, and so, $R$ is FORT precisely when for all $x \in R$,
        \[
        x \in \left(\sum_{\phi \in \Hom_R(F_*R,R)}\phi(F_*x)R\right)^{[p]}.    
        \]
        Moreover, in this case for all ideals $\ba \subseteq R$, we have
        \begin{equation*}
            \ba \subseteq \left(\sum_{\phi \in \Hom_R(F_*R,R)} \phi(F_*\ba) \right)^{[p]}.
        \end{equation*}
        In fact, since $F^e_*R$ is an ORT $R$-module for all integers $e > 0$ by \autoref{lem:FORT-FIF-FOR-iterated-Frobenius}, it follows that
        \begin{equation*}
            \ba \subseteq \left(\sum_{\phi \in \Hom_R(F^e_*R,R)} \phi(F^e_*\ba) \right)^{[p^e]}.
        \end{equation*} 

        \item \label{rem:FORT-properties.c} If $R$ is FORT, it is also FOR,  hence it follows that 
        $
        \left(\bigcap_i \ba_i\right)^{[p^e]} = \bigcap_i \ba_i^{[p^e]}.
        $ for any collection $\{\ba_i\}_i$ of ideals of $R$.

        \item \label{rem:FORT-properties.d} If $R$ is FORT and Noetherian, then $F_*R$ is a flat $R$-module by \cite[Remark 4.1.2]{DattaEpsteinTucker}. Thus, $R$ is regular by Kunz's theorem \cite{KunzCharacterizationsOfRegularLocalRings}. In light of this, we will assume that our rings are regular whenever we discuss the FORT property in a Noetherian setting.
    \end{enumerate}
\end{remark}

\begin{remark}
\label{rem:intersection-flatness-regularity}
A much weaker condition than FORT is often sufficient for flatness of Frobenius in the Noetherian setting. From \cite[Thm.\ 1]{JensenFlatness} (in the domain case), \cite[Main Thm.]{Epsteinregularitybracket} and \cite{KunzCharacterizationsOfRegularLocalRings}, it follows that if
$R$ is a Noetherian reduced ring of prime characteristic characteristic $p > 0$, then 
the following are equivalent:
\begin{enumerate}
    \item $R$ is regular. 
    \item $F_*R$ is a flat $R$-module. 
    \item For any two ideals $\ba_1, \ba_2$ of $R$, 
    $\ba_1^{[p]} \cap \ba_2^{[p]} = (\ba_1 \cap \ba_2)^{[p]}$.
\end{enumerate}
As a consequence, a Noetherian reduced FOR ring is regular. In fact, at the expense of replacing `reduced' by `domain', \cite[Thm.\ 1]{JensenFlatness} is substantially more general. Namely, it implies that if $R$ is any domain of characteristic $p > 0$ (Noetherian or not) such that $\ba_1^{[p]} \cap \ba_2^{[p]} = (\ba_1 \cap \ba_2)^{[p]}$ for any two ideals $\ba_1,\ba_2$ of $R$, then $F_*R$ is a flat $R$-module. Hence, FOR integral domains have flat Frobenius.
\end{remark}

\begin{lemma}
\label{lem:trace-equals-content}
    Suppose that $R$ is a 
    ring of prime characteristic $p > 0$ that is FORT. If $\ba \subseteq R$ is an ideal, then for all integers $e > 0$, and all ideals $\bb$ of $R$,
    \begin{equation*}
        \sum_{\phi \in \Hom_R(F^e_*R,R)} \phi(F^e_*\ba) = \bigcap_{\ba \subseteq \bb^{[p^e]}} \bb = c_{F^e_*R}(F^e_*\ba).
    \end{equation*}
    Thus, $\sum_{\phi \in \Hom_R(F^e_*R,R)} \phi(F^e_*\ba)$ is the unique smallest ideal $\bb$ of $R$ such that $\ba \subseteq \bb^{[p^e]}$.
\end{lemma}

\begin{proof}
    If $\bb \subseteq R$ is any ideal, then $\ba \subseteq \bb^{[p^e]}$ if and only if $F^e_*\ba \subseteq \bb \cdot F^e_*R$. Thus, if $\ba \subseteq \bb^{[p^e]}$, then
    \begin{equation*}
        \sum_{\phi \in \Hom_R(F^e_*R,R)} \phi(F^e_*\ba) \subseteq \sum_{\phi \in \Hom_R(F^e_*R,R)} \phi(\bb \cdot F^e_*R) \subseteq \bb.
    \end{equation*}
    The first equality is immediate by \autoref{rem:FORT-properties}(b) because $\ba \subseteq (\sum_{\phi \in \Hom_R(F^e_*R,R)} \phi(F^e_*\ba))^{[p^e]}$. Note that $c_{F^e_*R}(F^e_*\ba)$ is the intersection of all ideals $\bb$ of $R$ such that $F^e_*\ba \subseteq \bb \cdot F^e_*R = F^e_*(\bb^{[p^e]})$, which is the same as $\ba \subseteq \bb^{[p^e]}$. Thus, $\bigcap_{\ba \subseteq \bb^{[p^e]}} \bb = c_{F^e_*R}(F^e_*\ba)$. 
\end{proof}

The next result should be viewed as a generalization of the well-known fact that if $R$ is an $F$-finite regular ring, then $R$ is FORT \cite{BlickleMustataSmithDiscretenessAndRationalityOfFThresholds,KatzmanParameterTestIdealOfCMRings}.

\begin{lemma}
\label{lem:projective-regular-rings}
(cf. \cite[Proposition 5.4]{KatzmanParameterTestIdealOfCMRings})
    Suppose $R$ is a ring of prime characteristic $p > 0$.
    Then we have the following:
    \begin{enumerate}[label=\textnormal{(\arabic*)}] 
        \item If $F_*R$ is a projective $R$-module, then $R$ is FORT. \label{lem:projective-regular-rings.1}
        \item Suppose $F_*R$ is a countably generated $R$-module. Then $F_*R$ is a projective $R$-module if and only if $R$ is FORT. \label{lem:projective-regular-rings.2}
    \end{enumerate} 
\end{lemma}

\begin{proof}
\ref{lem:projective-regular-rings.1} follows by \cite[Lemma 4.1.5]{DattaEpsteinTucker}. For \ref{lem:projective-regular-rings.2}, 
by \cite[Remark 3.2.4]{DattaEpsteinTucker} we know that countably generated FORT modules are projective.
\end{proof}

\begin{proposition}
    \label{prop:polynomialmaps}
    Let $R$ be a ring of prime characteristic $p > 0$ that is FORT. Then the 
    polynomial ring $R[x_1,\dots,x_n]$ is FORT. If $R$ is Noetherian, then the power series ring $R\llbracket x_1,\dots,x_n\rrbracket $
    is FORT. 
\end{proposition}

\begin{proof}
   Frobenius on $R[x_1,\dots,x_n]$ factors as
   $
   R[x_1,\dots,x_n] \to (F_*R)[x_1,\dots,x_n] \to F_*(R[x_1,\dots,x_n]).
   $
   Here the first map is ORT 
   by base change \cite[Proposition 4.1.9]{DattaEpsteinTucker}, and the second map is ORT by \cite[Lemma 4.1.5]{DattaEpsteinTucker}
   because $F_*(R[x_1,\dots,x_n])$ is a free $(F_*R)[x_1,\dots,x_n]$-module. Therefore the composition is ORT by \cite[Lemma 4.1.7]{DattaEpsteinTucker}, which is precisely what it means for 
   $R[x_1,\dots,x_n]$ to be FORT.
   
   The proof that $R\llbracket x_1,\dots,x_n\rrbracket $ is FORT when $R$ is
   Noetherian follows by an analogous argument using
   the factorization 
   \[
   R\llbracket x_1,\dots,x_n\rrbracket  \to (F_*R)\llbracket x_1,\dots,x_n\rrbracket  \to F_*(R\llbracket x_1,\dots,x_n\rrbracket ).
   \]
   and \cite[Proposition 4.1.9]{DattaEpsteinTucker}. As in the polynomial 
   case, $F_*(R\llbracket x_1,\dots,x_n\rrbracket )$ is a free $(F_*R)\llbracket x_1,\dots,x_n\rrbracket $-module.
\end{proof}

\begin{corollary}
\label{cor:variables-FORT}
Let $p > 0$ be a prime number. The following rings are FORT:
\begin{enumerate}[label=\textnormal{(\arabic*)}] 
    \item Polynomial and power series rings in finitely many indeterminates over fields of characteristic $p$.\label{cor:variables-FORT.1}
    \item Power series rings over rings from \ref{cor:variables-FORT.1}.\label{cor:variables-FORT.2}
    \item Complete regular local rings of characteristic $p$.\label{cor:variables-FORT.3}
    \item Polynomial rings over complete regular local rings
    of characteristic $p$.\label{cor:variables-FORT.4}
\end{enumerate}
\end{corollary}

\begin{proof}
   By \autoref{prop:polynomialmaps} and Cohen's structure theorem, it
   suffices to show that a field $k$ of characteristic $p$ is FORT. But this is true because $F_*k$ is a free $k$-module. Note that proof here that complete regular ring of characteristic $p > 0$ are FORT is very different in flavor from the argument in \autoref{thm:FORT-F-intersection-flat}, which is less direct.
\end{proof}

 We now show some permanence properties for FORT
rings. Recall that a ring homomorphism $A \to B$ is \emph{weakly \'etale}
or \emph{absolutely flat} if both $A \to B$ and 
\begin{align*}
    B \otimes_A B &\to B\\
    b \otimes c &\mapsto bc
\end{align*}
are flat ring maps.

\begin{proposition}
    \label{prop:ascent-FORT-weakly-etale}
    Let $R$ be a FORT (resp. $F$-intersection flat) ring of prime characteristic $p > 0$ and suppose
    $R \to S$ is a ring map. Then $S$ is FORT (resp. $F$-intersection flat) in each of the following
    cases:
    \begin{enumerate}[label=\textnormal{(\arabic*)}]
        \item $R \to S$ is weakly \'etale. \label{prop:ascent-FORT-weakly-etale.1}
        \item $R \to S$ is smooth. \label{prop:ascent-FORT-weakly-etale.2}
    \end{enumerate}
    Thus, Henselizations and strict Henselizations of FORT (resp. $F$-intersection flat) local rings are FORT (resp. $F$-intersection flat).
\end{proposition}

\begin{proof}
    \ref{prop:ascent-FORT-weakly-etale.1} If $R \to S$ is weakly \'etale, then the relative Frobenius 
    $
    F_{S/R} \colon F_*R \otimes_R S \to F_*S
    $
    is an isomorphism by \cite[\href{https://stacks.math.columbia.edu/tag/0F6W}{Tag 0F6W}]{stacks-project}. Thus,  
    $
    F_S \colon S \to F_*S
    $
    can be identified
    as the base change of 
    $
    F_R \colon R \to F_*R
    $
    along $R \to S$. If $R$ is FORT (resp. is $F$-intersection flat), then $F_*R$ is an ORT $R$-module (resp. an intersection flat $R$-module) and we know by
    \cite[Proposition 4.1.9]{DattaEpsteinTucker} that the ORT property (resp. by \cite[Theorem 4.3.1]{DattaEpsteinTucker} for the intersection flatness property)
    is preserved under arbitrary base change. Thus, $F_*S \cong F_*R \otimes_R S$ is an ORT (resp. intersection flat) $S$-module, that is, $S$ is FORT (resp. $F$-intersection flat).
    
    \ref{prop:ascent-FORT-weakly-etale.2} Since $R \to S$ is smooth, the relative Frobenius $F_{S/R}$ is finite and  syntomic \cite[\href{https://stacks.math.columbia.edu/tag/0FW2}{Tag 0FW2}]{stacks-project} (note that this part of the argument in \cite{stacks-project} does not need $R \to S$ to be smooth of a fixed relative dimension, and in any case, one can always work on a distinguished affine open cover of $S$ where the relative dimension is fixed on the individual subsets of the cover). In particular, a syntomic map by definition is flat and of finite presentation as a ring homomorphism \cite[\href{https://stacks.math.columbia.edu/tag/00SL}{Tag 00SL}]{stacks-project}. Thus, $F_{S/R} \colon F_*R \otimes_R S \to F_*S$ is a flat, finite and finitely presented ring map. Then $F_*S$ is finitely presented as a $F_*R \otimes_R S$-module \cite[\href{https://stacks.math.columbia.edu/tag/0564}{Tag 0564}]{stacks-project}, and consequently, $F_*S$ is a projective $F_*R \otimes_R S$-module (this assertion is also stated without proof in \cite[1. Notations et rappels]{DeligneIllusieAlgebraicdeRham}). This implies that $F_*S$ is an ORT $F_*R \otimes_R S$-module by 
    \cite[Lemma 4.1.5]{DattaEpsteinTucker}
    since projective modules are always ORT, and hence $F_*S$ is also an intersection flat $F_*R \otimes_R S$-module since ORT modules are intersection flat 
    \cite[Proposition 4.3.8]{DattaEpsteinTucker}.
    Furthermore, since by assumption $F_*R$ is an ORT (resp. intersection flat) $R$-module, we get that $F_*R \otimes_R S$ is an ORT (resp. intersection flat) $S$-module by base change again. But a composition of ORT (resp. intersection flat) ring maps is ORT 
    \cite[Lemma 4.1.7]{DattaEpsteinTucker}
    (resp. is intersection flat \cite[Prop.\ 5.7(a)]{HochsterJeffriesintflatness}), and so, $F_*S$ is an ORT (resp. intersection flat) $S$-module.

    If $(R, \fm)$ is a local ring, and $R^h$ (resp. $R^{sh}$) is the
    Henselization (resp. strict Henselization) of $R$, then the 
    canonical maps $R \to R^h$ and $R^h \to R^{sh}$ are both ind-\'etale (filtered colimit of \'etale maps), and hence, both weakly 
    \'etale by \cite[\href{https://stacks.math.columbia.edu/tag/097N}{Tag 097N}]{stacks-project}. Thus, if $R$ is FORT (resp. $F$-intersection flat), then $R^h$ and $R^{sh}$ are
    FORT (resp. $F$-intersection flat)  by $(1)$.
\end{proof}

\begin{remark}
    \label{rem:FORT-FIF-regular-ascent}
    The reader may now naturally wonder if the FORT and $F$-intersection flatness properties ascend under a regular map (i.e. a flat map of Noetherian rings with geometrically regular fibers). We first observe that no such general result is possible. Indeed, let $R$ be a regular $\mathbb{F}_p$-algebra. Then the canonical map $\mathbb{F}_p \to R$ is a regular map because $\mathbb{F}_p$ is perfect and a regular ring over a perfect field is always geometrically regular by \cite[\href{https://stacks.math.columbia.edu/tag/0382}{Tag 0382}, \href{https://stacks.math.columbia.edu/tag/0381}{Tag 0381}]{stacks-project}. However, if $R$ is a non-excellent DVR, then $R$ is not $F$-intersection flat and hence also not FORT by \autoref{thm:FORT-F-intersection-flat}. The key point in this example is that the relative Frobenius of $\mathbb{F}_p \to R$ is just the Frobenius on $R$. Thus, in general if one hopes to ascend the FORT and $F$-intersection flatness properties, one needs strong assumptions on the relative Frobenius. For instance, here is a positive result:

\begin{quote}
    \emph{Let $R \to S$ be a regular map of Noetherian rings of prime characteristic $p > 0$. If $R$ is FORT (resp. $F$-intersection flat) and the relative Frobenius $F_{S/R}$ is FORT (resp. $F$-intersection flat), then $S$ is FORT (resp. $F$-intersection flat).}
\end{quote}

    \noindent This assertion follows because $F_R \otimes_R \id_S$ is ORT (resp. intersection flat) by base change and the composition of ORT (resp. intersection flat) ring maps is ORT (resp. intersection flat).
\end{remark}

\begin{remark}
\label{rem:FORT}
        The property of being FORT is not 
        local, that is, if $R_\p$ is FORT for every
        prime ideal $\p$ of $R$, then it is not true that $R$ is FORT, even when $R$ is a principal ideal domain (PID). 
        Heitmann has constructed an example of a PID $R$ defined over a countable 
        algebraically closed field $K$ of characteristic $p > 0$ such that $R$ has
        countably many non-zero prime ideals $f_iR$, $i \in \mathbb{Z}_{>0}$. Moreover,
        for each $i \in \mathbb{Z}_{>0}$, 
        $R_{(f_i)} \cong K[x,y_{i}]_{(g_i)},$    
        for suitably chosen 
        indeterminates $x, y_1, y_2, \dots$ and a height $1$ prime $(g_i)$ of $K[x,y_i]$ \cite{Heitmann-Nagata-not-local}. Thus $R$ is locally $F$-finite and hence is locally FORT. However, Heitmann shows that $R$ is not excellent because
        it fails to be a Nagata ring. In particular, since $\Frac(R)$ is $F$-finite,
        it follows by \cite[Thm.\ 3.2]{DattaSmithExcellence} that $\Hom_R(F_*R,R) = 0$, that is, $R$ is not FORT.
\end{remark}

\begin{corollary}
    \label{cor:FORT-localization}
    Let $R$ be a FORT ring of prime characteristic $p > 0$. Then for any multiplicative set $S$ of $R$, $S^{-1}R$ is FORT. Moreover, for any ideal $\ba$ of $R$ and for all $e \in \mathbb{Z}_{>0}$,
    \[
    \left(\sum_{\phi \in \Hom_R(F^e_*R,R)} \phi(F^e_*\ba)\right)(S^{-1}R) = \sum_{\varphi \in \Hom_{S^{-1}R}(F^e_*(S^{-1}R), S^{-1}R)} \varphi(F^e_*(\ba S^{-1}R)).    
    \]
\end{corollary}

\begin{proof}
    That $S^{-1}R$ is FORT follows by \autoref{prop:ascent-FORT-weakly-etale}(1) because $R \to S^{-1}R$ is weakly \'etale. Moreover, for all $e \in \mathbb{Z}_{. 0}$, $S^{-1}F^e_*R$ is canonically isomorphic to $F^e_*(S^{-1}R)$ (the relative Frobenius $F^e_{S^{-1}R/R}$ is an isomorphism because $R \to S^{-1}R$ is weakly \'etale) and under this isomorphism, $S^{-1}F^e_*\ba$ is identified with $F^e_*(\ba S^{-1}R)$. Thus,
    \begin{align*}
        \left(\sum_{\phi \in \Hom_R(F^e_*R,R)} \phi(F^e_*\ba)\right)(S^{-1}R) &= c_{F^e_*R}(F^e_*\ba)S^{-1}R\\
        &= c_{S^{-1}F^e_*R}(S^{-1}(F^e_*\ba))\\
        &= c_{F^e_*(S^{-1}R)}(F^e_*(\ba S^{-1}R))\\
        &= \sum_{\varphi \in \Hom_{S^{-1}R}(F^e_*(S^{-1}R), S^{-1}R)} \varphi(F^e_*(\ba S^{-1}R)).
    \end{align*}
    Here the first and fourth equalities follow by \autoref{lem:trace-equals-content} because $R$ and $S^{-1}R$ are FORT, the second equality follows because the content function behaves well with respect to localization for Ohm-Rush modules by 
    \cite[Proposition 3.4.14]{DattaEpsteinTucker} 
    (noting $F^e_*R$ is Ohm-Rush because it is ORT), and the third equality follows by the identifications mentioned above via the relative Frobenius $F^e_{S^{-1}R/R}$.
\end{proof}

\begin{proposition}
    \label{prop:FORT-pure-split-locus}
Let $R$ be a FORT ring of prime characteristic $p > 0$. Then for all finitely presented $R$-modules $P$, for all $e \in \mathbb{Z}_{>0}$ and for all $R$-linear maps 
\[
f \colon P \to F^e_*R,
\]
we have the following:
\begin{enumerate}[label=\textnormal{(\arabic*)}]
    \item $f$ is pure if and only if $f$ splits. \label{prop:FORT-pure-split-locus.1}
    \item $\Pure(f) = \{\p \in \Spec(R) \colon f_\p$ splits as $R_\p$-modules$\}$, In other words, the pure locus of $f$ coincides with its split locus. \label{prop:FORT-pure-split-locus.2}
    \item $\Pure(f)$ is open in $\Spec(R)$.\label{prop:FORT-pure-split-locus.3}
\end{enumerate}
\end{proposition}

\begin{proof}
    Since $R$ is FORT, for all $e \in \mathbb{Z}_{>0}$ and for all prime ideals $\p$ of $R$, $F^e_*R$ (resp. $F^e_*R_\p$) is an ORT $R$-module (resp. $R_\p$-module) by \autoref{lem:FORT-FIF-FOR-iterated-Frobenius} and \autoref{cor:FORT-localization}. 
    Then \ref{prop:FORT-pure-split-locus.1} and \ref{prop:FORT-pure-split-locus.2} follow by 
    \cite[Corollary 4.3.9]{DattaEpsteinTucker}. Furthermore, $\Pure(f)$ is open in $\Spec(R)$ by 
    \cite[Corollary 3.6.3]{DattaEpsteinTucker}, proving \ref{prop:FORT-pure-split-locus.3}.
\end{proof}

The next result talks about how the FORT property behaves with respect to completions. In comparison with an analogous result for the Ohm-Rush property \cite[Corollary 3.4.30]{DattaEpsteinTucker}, we are even able to make an assertion about arbitrary ideal-adic completions.

\begin{proposition}
    \label{prop:FORT-completions}
    Let $R$ be a Noetherian ring of prime characteristic $p > 0$ that is FORT. For any ideal $I$ of $R$, the $I$-adic completion $\widehat{R}^I$ is FORT. Moreover, for all ideals $\ba$ of $R$ and $e \in \mathbb{Z}_{> 0}$, we have 
        \[
        \left(\sum_{\phi \in \Hom_R(F^e_*R,R)} \phi(F^e_*\ba)\right)\widehat{R}^I = \sum_{\varphi \in \Hom_{\widehat{R}^I}(F^e_*\widehat{R}^I,\widehat{R}^I)} \varphi(F^e_*(\ba\widehat{R}^I)).
        \]
\end{proposition}

\begin{proof}
    By \cite[Proposition 4.1.9]{DattaEpsteinTucker}, the $I$-adic completion of Frobenius on $R$ is ORT, that is,
    $
    \widehat{F_R}^I \colon \widehat{R}^I \to \widehat{F_*R}^I    
    $
    is ORT. Since $I$ is finitely generated, the collection of ideals $\{(I^n)^{[p^e]} \colon n \in \mathbb{Z}_{>0}\}$ is cofinal with $\{I^n \colon n \in \mathbb{Z}_{>0}\}$ because $(I^n)^{[p^e]} = (I^{[p^e]})^n$ and $I$, $I^{[p^e]}$ have the same radical. This immediately implies that
    $
    \widehat{F_*R}^I \cong F_*\widehat{R^I}.    
    $
    Under this identification, the map $\widehat{F_R}^I$ can also be identified with the Frobenius $F_{\widehat{R}^I}$ on $\widehat{R}^I$. Thus, $\widehat{R}^I$ is FORT.

    Since $R$ and $\widehat{R}^I$ are FORT, by \autoref{lem:trace-equals-content} we have that 
    $\sum_{\phi \in \Hom_R(F^e_*R,R)} \phi(F^e_*\ba)$ 
    is the smallest ideal $\bb$ of $R$ such that $\ba \subseteq \bb^{[p^e]}$, and similarly,  
    $\sum_{\varphi \in \Hom_{\widehat{R}^I}(F^e_*\widehat{R}^I,\widehat{R}^I)} \varphi(F^e_*(\ba\widehat{R}^I))$ 
    is the smallest ideal $\bc$ of $\widehat{R}^I$ such that $\ba\widehat{R}^I \subseteq \bc^{[p^e]}$. 

    If $\ba \subseteq \bb^{[p^e]}$, then $\ba \widehat{R}^I \subseteq (\bb\widehat{R}^I)^{[p^e]} = \bb^{[p^e]}\widehat{R}^I$. Thus, 
    \[
        \sum_{\varphi \in \Hom_{\widehat{R}^I}(F^e_*\widehat{R}^I,\widehat{R}^I)} \varphi(F^e_*(\ba\widehat{R}^I)) \subseteq  \left(\sum_{\phi \in \Hom_R(F^e_*R,R)} \phi(F^e_*\ba)\right)\widehat{R}^I.
    \]
    Now note that every $\phi \in \Hom_R(F^e_*R,R)$ induces, after taking the $I$-adic completion, a $\widehat{R}^I$-linear map 
    $
    \widehat{\phi}^I \colon F^e_*\widehat{R}^I \to \widehat{R}^I    
    $
    such that the following diagram with the canonical vertical maps commutes:
    \[
        \xymatrix{F^e_*R \ar[r]^{\phi} \ar[d] &  R \ar[d] \\  F^e_*\widehat{R}^I \ar[r]^{\widehat{\phi}^I} &  \widehat{R}^I.}  
    \]
    Thus, $\phi(F^e_*\ba)\widehat{R} \subseteq \widehat{\phi}^I(F^e_*(\ba\widehat{R}^I))$, and so,
    \[
        \left(\sum_{\phi \in \Hom_R(F^e_*R,R)} \phi(F^e_*\ba)\right)\widehat{R}^I \subseteq \sum_{\phi \in \Hom_R(F^e_*R,R)} \widehat{\phi}^I(F^e_*(\ba\widehat{R}^I)) \subseteq \sum_{\varphi \in \Hom_{\widehat{R}^I}(F^e_*\widehat{R}^I,\widehat{R}^I)} \varphi(F^e_*(\ba\widehat{R}^I)).
    \]
    This proves the desired equality.
\end{proof}

\subsection{Quotients of FORT rings} 
Let $S$ be a ring of prime characteristic $p > 0$ that is FORT. Then we have seen in \autoref{lem:trace-equals-content} that for an ideal $\ba$ of  and for $e \in \mathbb{Z}_{>0}$, the content $c_{F^e_*S}(F^e_*\ba)$ can be completely characterized in terms of maps $F^e_*S \to S$. 
Now suppose $R = S/I$ is a quotient of the FORT ring $S$. Then for any $\phi \in \Hom_S(F^e_*S,S)$ and for all $c \in (I^{[p^e]}\colon_S I)$, the composition
$
    F^e_*S \xrightarrow{F^e_*c \cdot} F^e_*S \xrightarrow{\phi} S    
$
is an $S$ linear map that sends $F^e_*I$ into $I$. In other words, $\phi \circ (F^e_*c \cdot)$ induces an $R$-linear map $F^e_*R \to R$. 

Note that $\phi \circ (F^e_*c \cdot)$ is just the action of $F^e_*c$ on $\phi$ via the $F^e_*S$-module structure of $\Hom_S(F^e_*S,S)$. Thus, we see that there is a natural map from the $F^e_*S$-submodule $F^e_*(I^{[p^e]}\colon_S I) \cdot \Hom_S(F^e_*S, S)$ to $\Hom_R(F^e_*R,R)$ upon taking quotients.

\begin{lemma}
    \label{lem:FORT-quotient-single-element}
    Let $S$ be a ring of prime characteristic $p > 0$ that is FORT. Suppose $R = S/I$ is a quotient of $S$ and $x \in R$ with a lift $\tilde{x}$ to $S$ (i.e. $x = \tilde{x} + I$). Then we have the following:
    \begin{enumerate}[label=\textnormal{(\arabic*)}]
        \item \label{lem:FORT-quotient-single-element.1} If $\tilde{\tilde x}$ is another lift of $x$ to $S$ then for all $e \in \mathbb{Z}_{> 0}$,
        \[
            c_{F^e_*S}(F^e_*\tilde{x}(I^{[p^e]}:_S I))^{[p^e]} + I = c_{F^e_*S}(F^e_*\tilde{\tilde{x}}(I^{[p^e]}:_S I))^{[p^e]} + I.
        \]
         
        \item \label{lem:FORT-quotient-single-element.2} If $\tilde{x} \in c_{F^e_*S}(F^e_*\tilde{x}(I^{[p^e]}:_S I))^{[p^e]} + I$, then $x \in \Tr_{F^e_*R}(F^e_*x)^{[p^e]}$.
    \end{enumerate}
\end{lemma}

\begin{proof}
    \ref{lem:FORT-quotient-single-element.1} Since $\tilde{x} = \tilde{\tilde{x}} + i$ for some $i \in I$, for all $e \in \mathbb{Z}_{> 0}$ we have
    \[
    \tilde{x}(I^{[p^e]}:_S I) \subseteq \tilde{\tilde{x}}(I^{[p^e]}:_S I) + i(I^{[p^e]}:_SI) \subseteq \tilde{\tilde{x}}(I^{[p^e]}:_S I) + I^{[p^e]}.    
    \]
    Similarly,
    \[
        \tilde{\tilde{x}}(I^{[p^e]}:_S I)  \subseteq \tilde{x}(I^{[p^e]}:_S I) + (-i)(I^{[p^e]}:_S I) \subseteq \tilde{x}(I^{[p^e]}:_S I) + I^{[p^e]}.   
    \]
    Thus,
    \begin{equation}
        \label{eq:inclusion-1}
    c_{F^e_*S}(F^e_*\tilde{x}(I^{[p^e]}:_S I)) \subseteq c_{F^e_*S}(F^e_*\tilde{\tilde{x}}(I^{[p^e]}:_S I)) + c_{F^e_*S}(F^e_*I^{[p^e]})    
    \end{equation}
    and
    \begin{equation}
        \label{eq:inclusion-2}
        c_{F^e_*S}(F^e_*\tilde{\tilde{x}}(I^{[p^e]}:_S I)) \subseteq   c_{F^e_*S}(F^e_*\tilde{x}(I^{[p^e]}:_S I)) + c_{F^e_*S}(F^e_*I^{[p^e]}),  
    \end{equation}
    where we are using additivity of content over submodules \cite[Corollary 3.4.5]{DattaEpsteinTucker} because $F^e_*S$ is ORT (by \autoref{lem:FORT-FIF-FOR-iterated-Frobenius}) and hence Ohm-Rush.
    
    But $c_{F^e_*S}(F^e_*I^{[p^e]}) = I$. Indeed, it is clear that $c_{F^e_*R}(F^e_*I^{[p^e]}) \subseteq I$. Let $J$ be an ideal of $R$ such that $F^e_*I^{[p^e]} \subseteq JF^e_*R = F_*J^{[p^e]}$. Then by flatness of Frobenius we get $I \subseteq J$. Since this holds for all ideals $J$, we then have $I \subseteq c_{F^e_*R}(F^e_*I^{[p^e]})$. 
    Using $c_{F^e_*S}(F^e_*I^{[p^e]}) = I$ we now get
    \[
        c_{F^e_*S}(F^e_*\tilde{x}(I^{[p^e]}:_S I)) + I \stackrel{\autoref{eq:inclusion-1}}{\subseteq} c_{F^e_*S}(F^e_*\tilde{\tilde{x}}(I^{[p^e]}:_S I)) + I  \stackrel{\autoref{eq:inclusion-2}}{\subseteq}  c_{F^e_*S}(F^e_*\tilde{x}(I^{[p^e]}:_S I)) + I,
    \]
    that is,
    $
        c_{F^e_*S}(F^e_*\tilde{x}(I^{[p^e]}:_S I)) + I =  c_{F^e_*S}(F^e_*\tilde{\tilde{x}}(I^{[p^e]}:_S I)) + I.
    $
    Raising everything to $[p^e]$-powers and then adding $I$ now gives \ref{lem:FORT-quotient-single-element.1}.

   \ref{lem:FORT-quotient-single-element.2} By \autoref{lem:trace-equals-content}, 
   \begin{align*}
   c_{F^e_*S}(F^e_*\tilde{x}(I^{[p^e]}:_S I)) &= \sum_{\phi \in \Hom_S(F^e_*S,S)}\phi(F^e_*\tilde{x}(I^{[p^e]}:_S I))\\ 
   &= \sum_{\phi \in \Hom_S(F^e_*S,S)}\phi(F^e_*(I^{[p^e]}:_S I)F^e_*\tilde{x})\\
   &= \{\phi \circ (F^e_*c \cdot)(F^e_*\tilde{x}) \colon c \in (I^{[p^e]}:_S I)\}.
   \end{align*}
   Let $\overline{\phi_c}$ be the $R$-linear map $F^e_*R \to R$ induced by $\phi  \circ (F^e_*c \cdot) \colon F^e_*S \to S$. If $\pi \colon S \twoheadrightarrow R$ is the canonical projection, we get a commutative diagram
   \[
        \xymatrix@R+1pc@C+1pc{F^e_*S \ar[r]^{\phi \circ (F^e_*c \cdot)} \ar[d]_{F^e_*\pi} &  S \ar[d]^{\pi} \\  F^e_*{R} \ar[r]_{\overline{\phi_c}} & {R}.}  
\]
   Then in $R = S/I$,
   $
    \phi \circ (F^e_*c \cdot)(F^e_*\tilde{x}) + I = \pi \circ (\phi \circ (F^e_*c \cdot))(F^e_*\tilde{x}) = \overline{\phi_c} \circ F^e_*\pi (F^e_*\tilde{x}) = \overline{\phi_c}(F^e_*x).
   $
   The last equality follows because $\tilde{x} + I = x$ by hypothesis.
   This shows that the expanded ideal $c_{F^e_*S}(F^e_*\tilde{x}(I^{[p^e]}:_S I)) R$ is contained in 
   \[
    \Tr_{F^e_*R}(F^e_*x) = \im(\Hom_R(F^e_*R,R) \xrightarrow{\ev @ F^e_*x} R).    
   \]
   Thus, if $\tilde{x} \in c_{F^e_*S}(F^e_*\tilde{x}(I^{[p^e]}:_S I))^{[p^e]} + I$, then 
   $
   x = \tilde{x} + I \in (c_{F^e_*S}(F^e_*\tilde{x}(I^{[p^e]}:_S I))^{[p^e]} + I)R = (c_{F^e_*S}(F^e_*\tilde{x}(I^{[p^e]}:_S I))R)^{[p^e]} \subseteq \Tr_{F^e_*R}(F^e_*x)^{[p^e]},
   $
   completing the proof.
\end{proof}

We now come to the main result of this subsection.

\begin{theorem}
    \label{thm:regular-quotient-FORT}
    Let $S$ be a Noetherian regular ring of prime characteristic $p > 0$ that is FORT. Let $R = S/I$ be a quotient of $S$. If $R$ is regular, then $R$ is FORT.
\end{theorem}
\vspace{-2ex}
\begin{proof}
    Let $x \in R$. We have to show that $F_*x \in \Tr_{F_*R}(F_*x)F_*R = F_*(\Tr_{F_*R}(F_*x)^{[p]})$, or equivalently, that $x \in \Tr_{F_*R}(F_*x)^{[p]}$. Let $\tilde{x}$ be a lift of $x$ to $S$. By \autoref{lem:FORT-quotient-single-element}\ref{lem:FORT-quotient-single-element.2}, it suffices to show that
    \begin{equation}
        \label{eq:content-membership}
    \tilde{x} \in c_{F_*S}(F_*\tilde{x}(I^{[p]}:_S I))^{[p]} + I.    
    \end{equation}
    Note that \autoref{eq:content-membership} holds if and only if for all prime ideals $\p$ of $S$,
    \begin{equation}
        \label{eq:content-membership-completion}
    \tilde{x} \in (c_{F_*S}(F_*\tilde{x}(I^{[p]}:_S I))^{[p]} + I)\widehat{S_\p},
    \end{equation}
    where $\widehat{S_\p}$ denotes the $\p S_\p$-adic completion of $S_\p$. Here by abuse of notation, we are using $\tilde{x}$ to also denote the image of $\tilde{x}$ in $\widehat{S_\p}$. If $I \nsubseteq \p$, then $I\widehat{S_\p} = \widehat{S_\p}$ and \autoref{eq:content-membership-completion} is automatically satisfied. Thus, we may assume $I \subseteq \p$.
    In addition, we have 
    \begin{align*}
        c_{F_*S}(F_*\tilde{x}(I^{[p]}:_S I))^{[p]}\widehat{S_\p} &= (c_{F_*S}(F_*\tilde{x}(I^{[p]}:_S I))\widehat{S_\p})^{[p]}\\
        &=\left(\left(\sum_{\phi \in \Hom_S(F_*S,S)} \phi(F_*\tilde{x}(I^{[p]}:_S I))\right)\widehat{S_\p}\right)^{[p]}\\
        &=\left(\left(\sum_{\varphi \in \Hom_{S_\p}(F_*{S_\p},S_\p)} \varphi(F_*(\tilde{x}(I^{[p]}:_S I)S_\p))\right)\widehat{S_\p}\right)^{[p]}\\
        &=\left(\sum_{\psi \in \Hom_{\widehat{S_\p}}(F^e_*\widehat{S_\p}, \widehat{S_\p})}\psi(F^e_*(\tilde{x}(I{\widehat{S_\p}}^{[p]}:_{\widehat{S_\p}} I\widehat{S_\p})))\right)^{[p]}.\\
        &= c_{F_*\widehat{S_\p}}(F^e_*\tilde{x}(I\widehat{S_\p}^{[p]}\colon_{\widehat{S_\p}} I\widehat{S_\p}))^{[p]}
    \end{align*}    
    The first equality follows because taking $[p]$-powers of ideals commutes with expansions, the second equality follows by \autoref{lem:trace-equals-content} because $S$ is ORT, the third equality follows by \autoref{cor:FORT-localization}, the fourth equality follows by \autoref{prop:FORT-completions} because $S_\p$ is FORT by \autoref{cor:FORT-localization}, and the final equality follows by \autoref{lem:trace-equals-content} because $\widehat{S_\p}$ is FORT by \autoref{prop:FORT-completions}.

    Hence we have shown that
    \[
        \textrm{$\tilde{x} \in c_{F_*S}(F_*\tilde{x}(I^{[p]}:_SI))^{[p]} + I \Longleftrightarrow$ for all $\p \in \mathbf{V}(I),  \tilde{x} \in  c_{F_*\widehat{S_\p}}(F_*\tilde{x}(I\widehat{S_\p}^{[p]}\colon_{\widehat{S_\p}} I\widehat{S_\p}))^{[p]} + I\widehat{S_\p}$.}
    \]
    Replacing $S$ by $\widehat{S_\p}$ and $R$ by $\widehat{S_\p}/I\widehat{S_\p}$, we may assume that $S$ is complete regular local and $R$ is a regular quotient of $S$. Then $I$ must be generated by part of a regular system of parameters of $S$. Let $\kappa$ be the residue field of $S$. By Cohen's structure theorem we may further assume that $S = \kappa\llbracket x_1,\dots,x_d\rrbracket $ and $I = (x_1,\dots,x_c)$ for some $0 \leq c \leq d$. 
    
    Since $x \in R = S/I$, any lift $\tilde{x}$ of $x$ is of the form $\tilde{\tilde{x}} + i$, for some $i \in I$ and some 
    $
    \tilde{\tilde{x}} \in \kappa\llbracket x_{c+1},\dots,x_d\rrbracket  \subseteq \kappa\llbracket x_1,\dots,x_d\rrbracket $ 
    that is also a lift of $x$. Then by \autoref{lem:FORT-quotient-single-element}\ref{lem:FORT-quotient-single-element.1}, we have
    \begin{equation}
        \label{eq:equality-content-mod-I}
    c_{F_*S}(F_*\tilde{x}(I^{[p]}\colon_S I))^{[p]} + I  
    =c_{F_*S}(F_*\tilde{\tilde{x}}(I^{[p]}\colon_S I))^{[p]} + I.    
    \end{equation}

    Since $i \in I$, the upshot is that to show
    $
    \tilde{x} = \tilde{\tilde{x}} + i \in c_{F_*S}(F_*\tilde{x}(I^{[p]}\colon_S I))^{[p]} + I
    $
    it suffices to show
    $
    \tilde{\tilde{x}} \in c_{F_*S}(F_*(\tilde{\tilde{x}}(I^{[p]}\colon_S I)))^{[p]} + I.
    $
    The reason for comparing different lifts of $x \in R$ to $S$ is because in our situation, the lift $\tilde{x}$ comes globally (before localizing and completing) while we can choose the more special lift $\tilde{\tilde{x}}$ using Cohen's structure theorem after localization and completion.

    We will show something stronger. Namely, we know that 
    $
    x_1^{p-1}\cdots x_c^{p-1} \in (I^{[p]}:_SI).    
    $
    Thus, it is enough to show that 
    $
     \tilde{\tilde{x}} \in c_{F_*S}(F_*\tilde{\tilde{x}}x_1^{p-1}\cdots x_c^{p-1})^{[p]} + I
    $
    because $c_{F_*S}(F_*\tilde{\tilde{x}}x_1^{p-1}\cdots x_c^{p-1}) \subseteq c_{F_*S}(F_*\tilde{\tilde{x}}(I^{[p]}:I))$.

    Since $A \coloneqq \kappa\llbracket x_{c+1},\dots,x_n\rrbracket $ is FORT, we have 
    $
    F_{A*}\tilde{\tilde{x}} \in \Tr_{F_{A*}A}(F_{A*}\tilde{\tilde{x}}) \cdot F_{A*}A = F_{A*}\Tr_{F_{A*}A}(F_{A*}\tilde{\tilde{x}})^{[p]}.
    $ 
    Hence, there exists $\phi \colon F_{A*}A \to A$ such that 
    $
    F_{A*}\tilde{\tilde{x}} = F_{A*}\phi(F_{A*}\tilde{\tilde{x}})^p,    
    $
    or equivalently, $\tilde{\tilde{x}} = \phi(F_{A*}\tilde{\tilde{x}})^p$. Consider the $S = A\llbracket x_1,\dots,x_c\rrbracket $-linear map
    $
    \phi\llbracket x_1,\dots,x_c\rrbracket  \colon (F_{A*}A)\llbracket x_1,\dots,x_c\rrbracket  \to S
    $
    that extends $\phi$.
    We again have
    $
        \tilde{\tilde{x}} = \phi(F_{A*}\tilde{\tilde{x}})^p = \phi\llbracket x_1,\dots,x_c\rrbracket (F_{A*}\tilde{\tilde{x}})^p .   
    $
    Now observe that $F_{S*}S$ is a free $(F_{A*}A)\llbracket x_1,\dots,x_c\rrbracket $-module with basis
    $
    \{F_{S*}x^{\alpha_1}_1 \cdots x^{\alpha_c}_c \colon \alpha_i \in \mathbb{Z}, 0 \leq \alpha_i \leq p-1\}.
    $
    Let $\pi \colon F_{S_*}S \to (F_{A*}A)\llbracket x_1,\dots,x_c\rrbracket $ be any $(F_{A*}A)\llbracket x_1,\dots,x_c\rrbracket $-linear map that sends 
    $
    F_{S*}x^{p-1}_1 \cdots x^{p-1}_c \mapsto 1.
    $
    Then
    $
    \pi(F_{S*}\tilde{\tilde{x}}x^{p-1}_1 \cdots x^{p-1}_c) = \pi(F_{A*}\tilde{\tilde{x}} \cdot F_{S*}x^{p-1}_1 \cdots x^{p-1}_c) = F_{A*}\tilde{\tilde{x}}    
    $
    by linearity. Thus, the composition
    \[
    F_{S*}S \xlongrightarrow {\pi} (F_{A*}A)\llbracket x_1,\dots,x_c\rrbracket  \xlongrightarrow{\phi\llbracket x_1,\dots,x_c\rrbracket } S    
    \]
    is an $S$-linear map that sends $F_{S*}\tilde{\tilde{x}}x^{p-1}_1 \cdots x^{p-1}_c \mapsto \phi(F_{A*}\tilde{\tilde{x}})$. Consequently,
    \begin{align*}
    \tilde{\tilde{x}} = (\phi\llbracket x_1,\dots, x_c\rrbracket  \circ \pi(F_{S*}\tilde{\tilde{x}}x^{p-1}_1 \cdots x^{p-1}_c))^p &\in \Tr_{F_{S*}S}(F_{S*}\tilde{\tilde{x}}x^{p-1}_1 \cdots x^{p-1}_c)^{[p]}\\ 
    &= c_{F_{S*}S}(F_{S*}\tilde{\tilde{x}}x^{p-1}_1 \cdots x^{p-1}_c)^{[p]}\\
    &\subseteq  c_{F_{S*}S}(F_{S*}\tilde{\tilde{x}}x^{p-1}_1 \cdots x^{p-1}_c)^{[p]} + I,
    \end{align*}
as desired. Here the equality between trace and content again follows because $F_*S$ is an ORT $S$-module; see \cite[Remark 4.1.2]{DattaEpsteinTucker}. 
\end{proof}

\subsection{Descent of \emph{F}-intersection flatness and Frobenius Ohm-Rush}
\label{subsec:Descent-FIF}

Let $\varphi \colon R \to S$ be a homomorphism of rings of prime characteristic $p > 0$ such that $S$ is $F$-intersection flat. It would be desirable to have conditions on $\varphi$ that imply that $R$ is also $F$-intersection flat. We note that $\varphi$ has to satisfy a stronger condition than just faithful flatness. Indeed, if $(R,\fm)$ is a non-excellent DVR of characteristic $p > 0$, then we have seen that $R$ is not $F$-intersection flat (\autoref{thm:FORT-F-intersection-flat}). However, $\widehat{R}$ is always $F$-intersection flat and $R \to \widehat{R}$ is faithfully flat. The condition we want will be given in terms of an appropriate purity assumption on the relative Frobenius
$
F_{\varphi} \colon F_*R \otimes_R S \to F_*S
$
of $\varphi$. 

Purity of the relative Frobenius as a ring map has been studied before by Hashimoto \cite{HashimotoF-pure-homomorphisms}. Recall, that a homomorphism $\varphi \colon R \to S$ of Noetherian rings of prime characteristic $p > 0$ is called \emph{F-pure} (this is the relative version of an $F$-pure ring) in \cite[(2.3)]{HashimotoF-pure-homomorphisms} if the relative Frobenius $F_\varphi$ is a pure ring map. By the Radu-Andr{\'e} theorem, any regular homomorphism of Noetherian rings is $F$-pure because the relative Frobenius is then faithfully flat and faithful flatness implies purity. However, the notion of an $F$-pure homomorphism is a lot weaker than the notion of a regular homomorphism. Indeed, for any Noetherian $F$-pure ring $R$, the canonical map $\mathbb{F}_p \to R$ is an $F$-pure homomorphism which is very far from being regular in general (otherwise all Noetherian $F$-pure rings would be regular!). At the same time, in some cases being $F$-pure is equivalent to being regular. For example, for a Noetherian local ring $(R,\fm)$, the canonical map $R \to \widehat{R}$ being $F$-pure is equivalent to $R \to \widehat{R}$ being regular by \autoref{cor:rel-Frob-completions}. 

Weaker than $F_{\varphi}$ being a pure ring map is the condition that $F_{\varphi}$ is pure as a map of $S$-modules. It turns out that this is the condition we will need for descent of intersection flatness.

\begin{theorem}
    \label{thm:descent-F-intersection-flat}
    Let $\varphi \colon R \to S$ be a pure homomorphism of rings of prime characteristic $p > 0$. Let
    \[
    F_{\varphi} \colon F_*R \otimes_R S \to F_*S
    \]
   be the relative Frobenius of $\varphi$. We have the following:
   \begin{enumerate}[label=\textnormal{(\arabic*)}]
    \item If $F_\varphi$ is a pure map of $S$-modules and $S$ is $F$-intersection flat, then $R$ is $F$-intersection flat.\label{thm:descent-F-intersection-flat.1}
    \item If $R$ and $S$ are Noetherian and $\varphi$ is an $F$-pure homomorphism (for example, if $\varphi$ is regular) and $S$ is $F$-intersection flat, then $R$ is $F$-intersection flat.\label{thm:descent-F-intersection-flat.2}
    \item Let $(A, \fm)$ be a Noetherian local ring of prime characteristic $p > 0$ such that $A$ is a $G$-ring. If $R$ is a regular ring that is essentially of finite type over $A$, then $R$ is $F$-intersection flat.\label{thm:descent-F-intersection-flat.3}  
   \end{enumerate}
\end{theorem}

\begin{proof}
    We will need the diagram for the relative Frobenius
        \[
        \begin{tikzcd}[column sep=4em]
          R \rar{F_R}\dar[swap]{\varphi} & F_{*}R
          \arrow[bend left=30]{ddr}{F_{*}\varphi}
          \dar{\id_{F_{*}R} \otimes_R \varphi }\\
          S \rar{F_R \otimes_R \id_S} \arrow[bend right=12,end
          anchor=west]{drr}[swap]{F_S} & F_{*}R \otimes_R S
          \arrow[dashed]{dr}[description]{F_{\varphi}}\\
          & & F_{*}S,
        \end{tikzcd}
        \]

        \ref{thm:descent-F-intersection-flat.1} Recall first that pure submodules of flat modules are flat \cite[Lemma 2.2.6]{DattaEpsteinTucker}, and so, $F_*R \otimes_R S$ is also a flat $S$-module.
        Since $S$ is $F$-intersection flat, we get that $F_*R \otimes_R S$ is also $F$-intersection flat by \cite[Remark 3.2.4, Theorem 4.3.1]{DattaEpsteinTucker}. Consequently, $F_*R$ is an intersection flat $R$-module by pure descent of intersection flatness \cite[Corollary 4.3.2]{DattaEpsteinTucker}.

        \ref{thm:descent-F-intersection-flat.2} follows from \ref{thm:descent-F-intersection-flat.1} because $\varphi$ being an $F$-pure homomorphism means that $F_\varphi$ is a pure ring map, and so, by restriction of scalars also a pure map of $S$-modules. 

        \ref{thm:descent-F-intersection-flat.3} Since $(A,\fm)$ is a $G$-ring, $A \rightarrow \widehat{A}$ is a regular map. Then for any essentially of finite type $A$-algebra $R$, the induced map $R \to R \otimes_A \widehat{A}$ is faithfully flat and regular as well (\cite[\href{https://stacks.math.columbia.edu/tag/07C1}{Tag 07C1}]{stacks-project} shows this when $R$ is of finite type, but a further localization of a regular ring map is clearly also regular because the fibers after localization remain the same as before one localizes). Note that $R \otimes_A \widehat{A}$ is a regular ring because $R$ is and the property of being regular ascends along regular maps \cite[\href{https://stacks.math.columbia.edu/tag/033A}{Tag 033A}]{stacks-project}. Thus, $R \otimes_A \widehat{A}$, which is essentially of finite type over $\widehat{A}$, is FORT by \autoref{thm:regular-quotient-FORT} (as well as \autoref{cor:variables-FORT} and \autoref{cor:FORT-localization}). In particular, $R \otimes_A \widehat{A}$ is $F$-intersection flat by  \autoref{thm:FORT-F-intersection-flat}. Since $R \to R \otimes_A \widehat{A}$ is faithfully flat and regular, by \ref{thm:descent-F-intersection-flat.2} we now get that $R$ is $F$-intersection flat.
\end{proof}

\begin{remark}
    \label{rem:cyclic-purity-purity-rel-Frob}
    When $F_*S$ is a flat $S$-module (as is the case when $S$ is $F$-intersection flat), the $S$-purity of the relative Frobenius $F_*R \otimes_R S \to F_*S$ is equivalent to the $S$-cyclic purity of the relative Frobenius by \cite[\href{https://stacks.math.columbia.edu/tag/0AS5}{Tag 0AS5}]{stacks-project}. Thus, we could have stated \autoref{thm:descent-F-intersection-flat} by requiring the relative Frobenius to be $S$-cyclically pure instead of $S$-pure.
\end{remark}

We obtain an analogous descent result for the Frobenius Ohm-Rush property.

\begin{theorem}
    \label{thm:descent-FOR}
    Let $\varphi \colon R \to S$ be a cyclically pure homomorphism of rings of characteristic $p$ such that $F_*R$ is a flat $R$-module. Suppose the relative Frobenius $F_\varphi$ is $S$-cyclically pure. If $S$ is FOR, then $R$ is FOR.
\end{theorem}

Before proving the theorem, we note that we do not require $S$ to have flat Frobenius in the generality in which the result is stated. As a consequence, $S$-cyclic purity of $F_\varphi$ is possibly weaker than $S$-purity (see \autoref{rem:cyclic-purity-purity-rel-Frob}). However, if $S$ is Noetherian or if $S$ is a domain, then flatness of Frobenius will follow from the other assumptions. If $S$ is a domain then $F_*S$ is a torsion-free $S$-module. As a consequence, if $S$ is FOR, then $F_*S$ is $S$-flat by \cite[Thm.\ 1]{JensenFlatness} (see \autoref{rem:intersection-flatness-regularity}). If $S$ is Noetherian, but not necessarily a domain, then the $S$-cyclic purity of $F_\varphi$ implies $F_\varphi$ is injective. Thus, $F_S = F_\varphi \circ (F_R \otimes_R \varphi)$ is injective as well since $F_R \otimes_R \varphi$ is a faithfully flat ring map. Thus, $S$ is reduced. But a reduced Noetherian FOR ring is regular by \cite[Main Thm.]{Epsteinregularitybracket}, and so, $F_*S$ is $S$-flat by \cite{KunzCharacterizationsOfRegularLocalRings}.

\begin{proof}[Proof of \autoref{thm:descent-FOR}]
$S$-cyclic purity of $F_\varphi \colon F_*R \otimes_R S \to F_*S$ implies that $F_*R \otimes_R S$ is an Ohm-Rush $S$-module since $F_*S$ is an Ohm-Rush $S$-module; see \cite[Proposition 3.4.14]{DattaEpsteinTucker}.
Then by descent \cite[Theorem 3.5.4]{DattaEpsteinTucker} we conclude that $F_*R$ is an Ohm-Rush $R$-module (we need $F_*R$ to be $R$-flat to apply cyclically pure descent of the Ohm-Rush property). In other words, $R$ is FOR.
\end{proof}

One can also use the purity of relative Frobenius to deduce an ascent statement for the content function of Frobenii.

\begin{proposition}
    \label{prop:ascent-content-charp}
    Let $\varphi \colon R \to S$ be a homomorphism of rings of prime characteristic $p > 0$ such that $F_*R$ is a flat $R$-module. Let $e > 0$ be an integer and suppose that the relative Frobenius $F^e_\varphi$ is a cyclically pure map of $S$-modules (for example, if $R, S$ are Noetherian and $\varphi$ is a regular map). Let $x \in R$ and $y \coloneqq \varphi(x)$. If $F^e_{R*}x \in c_{F^e_{R*}R}(F^e_{R*}x)F^e_{R*}R$, then $F^e_{S*}y \in c_{F^e_{S*}S}(F^e_{S*}y)F^e_{S*}S$ and $c_{F^e_{S*}S}(F^e_{S*}y) = c_{F^e_{R*}R}(F^e_{R*}x)S$.
\end{proposition}

\begin{proof}
    This is a straightforward application of \cite[Proposition 3.4.16]{DattaEpsteinTucker} upon taking $N = F^e_{R*}R$ (which is $R$-flat because Frobenius on $R$ is flat), $M = F^e_{S*}S$ and $f \colon N \to M = F^e_*\varphi$. Note that the relative Frobenius $F^e_{\varphi}$ is then precisely the $S$-linear map $\tilde{f} \colon N \otimes_R S \to M$ induced by $f$.
\end{proof}

\begin{remark}
    As far as we can tell, just assuming $S$-cyclic purity of $F_\varphi \colon F_*R \otimes_R S \to F_*S$ does not guarantee the $S$-cyclic purity of $F^e_\varphi$ for all integers $e > 0$. In addition, without assuming $R$ is Frobenius Ohm-Rush, it is not clear if $F_{R*}x \in c_{F_{R*}R}(F_{R*}x)F_{R*}R$ implies $F^e_{R*}x \in c_{F^e_{R*}R}(F^e_{R*}x)F^e_{R*}R$ for all integers $e > 0$.
\end{remark}

\section{On the openness of pure loci of \texorpdfstring{$p^e$}--linear maps}
\label{sec:openness-of-pure-loci}

\subsection{Openness of \emph{F}-pure locus} Our next goal is to show that for a large class of regular rings $S$ of prime characteristic $p > 0$ (including non-excellent ones) and a quotient $R = S/I$ of $S$, the pure locus of maps of the form $R \to F^e_*R$ are open for all integers $e > 0$ (\autoref{thm:opennessforquotientsofregular}). We will deduce the statement for all quotients of $S$ by first proving the statement when $R = S$ itself (\autoref{thm:purelocusofregisopen}). First, we need some preparatory lemmas.

\begin{lemma}
    \label{lem:filtration}
    Suppose $S$ is a commutative ring, $x_1, \ldots, x_c$ is a permutable regular sequence in $S$, and $b \in \mathbb{Z}_{>0}$. Consider the set $\sA = \{ (\alpha_1, \ldots, \alpha_c) \in \mathbb{Z}^{\oplus c} \mid 0 \leq \alpha_i < b \}$ with the lexicographic ordering $\leq_{\mathrm{lex}}$, and set
    \begin{equation*}
        I_\alpha := (x_1^b, \ldots, x_c^b) + \left( x_1^{a_1}\cdots x_c^{a_c} \mid \alpha \leq_{\mathrm{lex}} (a_1, \ldots, a_c) \in \sA \right)
    \end{equation*}
    For each $\alpha \leq_{\mathrm{lex}} \alpha'$ with $\alpha, \alpha' \in \sA$, it follows that
    \begin{equation*}
        I_\infty := (x_1^b, \ldots, x_c^b) \subset \cdots \subset I_{\alpha'} \subset I_\alpha \subset  \cdots \subset I_{(0,\ldots,0)} = S
    \end{equation*}
    and the successive quotients in this filtration are all isomorphic to $S / (x_1, \ldots, x_c)$.
\end{lemma}

\begin{proof}
    Suppose $\alpha = (\alpha_1, \ldots, \alpha_c) \in \sA$ with immediate successor $\alpha' \in \sA \cup \{ \infty\}$. Since $I_{\alpha} = (I_{\alpha'}, x^\alpha)$ with $x^\alpha = x_1^{\alpha_1}\cdots x_c^{\alpha_c}$, it suffices to check that $(I_{\alpha'}:_S x^\alpha) = (x_1, \ldots, x_c)$. As $x_1, \ldots, x_c$ are a permutable regular sequence, by \cite{EagonHochster} (\textit{cf.} \cite[Proposition 7.4]{HochsterHunekeTCandBrianconSkoda}) we need only verify this statement for $S = \mathbb{Z}[x_1, \ldots, x_c]$. Furthermore, this result can be checked after reduction modulo $p \gg 0$.
    
    Thus, we may assume $S = \mathbb{K}[x_1, \ldots, x_c]$ for a field $\mathbb{K}$, and are left with a straightforward calculation of a colon of monomial ideals. As each of the monomials $x_1^{b_1}\cdots x_c^{b_c}$ generating $I_{\alpha'}$ necessarily have $b_i > \alpha_i$ for some $1 \leq i \leq c$, it follows that $(I_{\alpha'}:_S x^\alpha) \subseteq (x_1, \ldots, x_c)$. For the reverse containment, consider each $x_i$ with $1 \leq i \leq c$. If $\alpha_i = b - 1$, then $x_i x^\alpha \in (x_i^b) \subseteq I_{\alpha'}$. Else, if $\alpha_i < b-1$, then $(\alpha_1, \ldots, \alpha_{i-1}, \alpha_i + 1, \alpha_{i+1}, \ldots, \alpha_c) \in \sA$ and $\alpha <_{\mathrm{lex}} (\alpha_1, \ldots, \alpha_{i-1}, \alpha_i + 1, \alpha_{i+1}, \ldots, \alpha_c)$. Since $\alpha'$ is the immediate successor of $\alpha$, we have $\alpha' \leq_{\mathrm{lex}} (\alpha_1, \ldots, \alpha_{i-1}, \alpha_i + 1, \alpha_{i+1}, \ldots, \alpha_c)$ and it follows $x_i x^\alpha \in I_{\alpha'}$ as desired.
\end{proof}

\begin{lemma}
    \label{lem:extnofFedder}
    Suppose that $(S,\fram,K)$ is a regular local ring of characteristic $p > 0$, and $R = S/I$ for an ideal $I$ of $S$. Let $\tilde{r} \in S$ have image $r = \tilde{r} + I \in R$, and assume $e \in \mathbb{Z}_{>0}$. Then, the $R$-module homomorphism
    \begin{equation*}
        R \xrightarrow{1 \mapsto F^e_*r} F^e_*R
    \end{equation*}
    is pure if an only if $(I^{[p^e]}:_S I) \tilde{r} \not\subseteq \fram^{[p^e]}$.
    \end{lemma}
    
    \begin{proof}
       Using \autoref{lem:F-purity-completions}, we may immediately reduce to the case where $R$ and $S$ are complete, and next reduce to the case where the residue field $K$ is algebraically closed. Picking a coefficient field and identifying $S$ with $K\llbracket x_1, \ldots, x_n \rrbracket$, let $\overline{S} = \overline{K}\llbracket x_1, \ldots, x_n \rrbracket$ and $\overline{R} = \overline{S}/IS$. The map $R \to \overline{R}$ is flat with closed fiber a field $\overline{R}/\fram \overline{R} = \overline{K}$ a field, whence the injective hull of the residue field of  $R$ base changes to the injective hull of the residue field of $\overline{R}$, or symbolically $\overline{R}\otimes_R E_R(K) = E_{\overline{R}}(\overline{K})$ (more generally, see \cite[Lemma 7.10]{HochsterHunekeFRegularityTestElementsBaseChange}). Additionally, if $\delta \in E_R(K)$ generates the socle, then $1 \otimes \delta \in E_{\overline{R}}(\overline{K})$ generates the socle. Consider the diagram
    \begin{equation}
    \label{eq:forR}
    \begin{tikzcd}
        R && {F^e_*R} \\
        {\overline{R}} && {F^e\overline{R}}
        \arrow["{1 \mapsto F^e_*r}", from=1-1, to=1-3]
        \arrow["{1 \mapsto F^e_*r}"', from=2-1, to=2-3]
        \arrow[from=1-1, to=2-1]
        \arrow[from=1-3, to=2-3]
    \end{tikzcd}.
    \end{equation}
    Tensoring over $R$ with $E_R(K)$ gives
    \[\begin{tikzcd}
        & {E_R(K)} && {E_R(K) \otimes_R F^e_*R} \\
        {\overline{R} \otimes_R E_R(K) } & {E_{\overline{R}}(\overline{K})} && {E_{\overline{R}}(\overline{K}) \otimes_{\overline{R}}F^e\overline{R}} & {E_R(K) \otimes_R F^e_*\overline{R}}
        \arrow["{1 \mapsto F^e_*r}", from=1-2, to=1-4]
        \arrow["{1 \mapsto F^e_*r}"', from=2-2, to=2-4]
        \arrow[from=1-2, to=2-2]
        \arrow[from=1-4, to=2-4]
        \arrow["{=}"{description}, draw=none, from=2-1, to=2-2]
        \arrow["{=}"{description}, draw=none, from=2-4, to=2-5]
    \end{tikzcd}\]
    where the downward arrows are injective. Since $\delta$ maps to zero on the top row if and only if $1 \otimes \delta$ maps to zero along the bottom row, it follows that the top row of \eqref{eq:forR} is pure as a map of $R$-modules if and only if the bottom row of \eqref{eq:forR} is pure as a map of $\overline{R}$-modules.
    
    Thus, we are reduced to the case where $R$ and $S$ are complete with algebraically closed residue field, and so in particular $F$-finite. Here, the statement is well known; see \cite[Lemma 2.2]{Glassbrenner}.
    \end{proof}

    The next result, which holds over any commutative ring $S$ of prime characteristic with flat Frobenius, gives an ideal-theoretic criterion for the purity of the maps $\lambda_{x,e}$. In particular, we get an ideal-theoretic purity criterion for the maps $\lambda_{x,e}$ when $S$ is a regular ring.

    \begin{proposition}
        \label{prop:regular-bracket-powers}
        Let $S$ be a ring of prime characteristic $p > 0$ such that $F_*S$ is a flat $S$-module. Let $\p$ be a prime ideal of $S$, $x \in S$ and 
        \begin{align*}
            S &\xrightarrow{\lambda_{x,e}} F^e_*S\\ 
            1 &\mapsto F^e_*x
            \end{align*}
         be the unique $S$-linear map for any integer $e > 0$.  Then we have the following:
        \begin{enumerate}[label=\textnormal{(\arabic*)}]
            \item For all integers $e > 0$, $x \in \p^{[p^e]}$ if and only if $x/1 \in \p^{[p^e]}S_\p$. \label{prop:regular-bracket-powers.1}
            \item $\lambda_{x,e}$ is pure if and only if for all maximal ideals $\fm$ of $S$, $x \notin \fm^{[p^e]}$.  \label{prop:regular-bracket-powers.2}
        \end{enumerate}
    \end{proposition}

    \begin{proof}
        For all integers $e > 0$, $F^e_*S$ is a flat $S$-module. 

        \ref{prop:regular-bracket-powers.1} The non-trivial implication is to show that if $x/1 \in \p^{[p^e]}S_\p$, then $x \in \p^{[p^e]}$. By assumption, there exists $t \notin \p$ such that $tx \in \p^{[p^e]}$. Then $t^{p^e}x \in \p^{[p^e]}$ as well. Thus,
        $
        x \in (\p^{[p^e]} \colon_S \hspace{1mm} t^{p^e}) = (\p \colon_S \hspace{1mm} t)^{[p^e]} = \p^{[p^e]},    
        $
        where in the first equality we are using that $F^e_*S$ is a flat $S$-module (see \cite[3H]{MatsumuraCommutativeAlgebra}). The final equality follows because $\p$ is prime and $t \notin \p$.

        \ref{prop:regular-bracket-powers.2} The implication $\implies$ is clear because pure maps are cyclically pure. For $\impliedby$, we apply \cite[Corollary 3.4.37]{DattaEpsteinTucker} to the map $\lambda_{x, e}$ to get the desired result.
    \end{proof}

    We can show the openness of loci of $p^e$-linear maps for a large class of prime characteristic regular rings.

\begin{theorem}[c.f. {\cite[Cor. 7.15]{HochsterYaoGenericLocalDuality}, \cite[Cor.\ 2.18]{HochsterYaoSFRsmallCM}, \cite[Thm. A.2.4]{LyuUniformBounds}}]
\label{thm:purelocusofregisopen}
    Suppose that $S$ is a regular ring of characteristic $p > 0$ so that the regular locus of $S/\bp$ contains a non-empty open subset of $\Spec(S/\p)$ for all $\bp \in \Spec(S)$ (i.e. $S/\p$ is J-0 for all prime ideals $\p$ of $S$). For any $x \in S$ and $e \in \mathbb{Z}_{>0}$, the pure locus of 
    \begin{align*}
    S &\xrightarrow{\lambda_{x,e}} F^e_*S\\ 
    1 &\mapsto F^e_*x
    \end{align*}
    is open.
\end{theorem}

\begin{proof}
    Let $U \subseteq \Spec(S)$ be the pure locus of $\iota$. Since purity passes to localizations, we have that $\bp \subseteq \bq$ and $\bq \in U$ implies $\bp \in U$ for $\bp,\bq \in \Spec(S)$. Thus, by Nagata's criterion \cite[\href{https://stacks.math.columbia.edu/tag/0541}{Tag 0541}]{stacks-project} (see also \cite[Lem.\ 22.B]{MatsumuraCommutativeAlgebra}), it suffices to show for all $\bp \in U$ that $U$ contains a non-empty open subset of $\Spec(S/\bp)$. 
    
    By assumption, the regular locus of $\Spec(S/\bp)$ contains a non-empty open set. Thus, inverting an element of $S \setminus \bp$ and replacing $S$ with this localization, we may assume that $S/\bp$ is regular. 
    Furthermore, since $S_\p$ is regular, after inverting an element of $S \setminus \bp$ and replacing $S$ with this localization we may further assume that $\bp = (x_1, \ldots, x_c)$ for a permutable regular sequence $x_1, \ldots, x_c$ in $S$. Indeed, $S_\p$ is regular so $\p S_\p$ is generated by regular system of parameters $x_1,\dots,x_c$ of $S_\p$ which we may assume are all in the image of $S \to S_\p$. Now the property of being a regular sequence spreads to an open neighborhood $D(f)$ of $\p$ for some $f \notin \p$ \cite[\href{https://stacks.math.columbia.edu/tag/061L}{Tag 061L}]{stacks-project}. So spreading out the property of being a regular sequence for all permutations of $x_1,...,x_c$ (which remain regular sequences over $S_\p$) and taking the intersection of these open neighborhoods gives us the desired result.

    Since $(\lambda_{x,e})_{\bp}$ is pure, we have that $x \not\in \bp^{[p^e]}$. Consider now the filtration of $\bp^{[p^e]}$ given in \autoref{lem:filtration}, with the ideal $I_{(a_1, \ldots, a_c)}$ of $S$ generated by $x_1^{p^e},\ldots, x_c^{p^e}$ and all monomials in $x_1, \ldots, x_c$ with exponent vectors greater than or equal to $(a_1, \ldots, a_c)$ in the lexicographic order and with entries at most $p^e$ for $(a_1, \ldots, a_c) \in \sA = \{ (\alpha_1, \ldots, \alpha_c) \in \mathbb{Z}^{\oplus c} \mid 0 \leq \alpha_i < p^e \}$. If $\alpha \in \sA$ is smallest so that $x \in I_\alpha$, it follows that there is some $u \in S \setminus \bp$ so that $x - u x^\alpha \in I_{\alpha'}$. Replacing $S$ by $S[u^{-1}]$ and $x$ by $u^{-1}x$, we may assume $x - x^\alpha \in I_{\alpha'}$ and argue that $\Spec(S/\bp)$ is contained in $U$. Suppose $\bq \in \Spec(S)$ with $\bp \subseteq \bq$ and $\dim(S_\bq) = d$. Lifting a regular system of parameters of $(S/\bp)_\bq$ to $S_\bq$, we can find $x_{c+1}, \ldots, x_d \in S_\bq$ so that $x_1, \ldots, x_c, x_{c+1}, \ldots x_d$ are a regular system of parameters of $S_\bq$. Again, $x_1, \ldots, x_d$ are a permutable regular sequence, and we may consider the filtration of $\bq^{[p^e]}S_\bq$ given in \autoref{lem:filtration}, with the ideal $J_{(b_1, \ldots, b_d)}$ of $S_\bq$ generated by $x_1^{p^e},\ldots, x_d^{p^e}$ and all monomials in $x_1, \ldots, x_d$ with exponent vectors greater than or equal to $(b_1, \ldots, b_d)$ in the lexicographic order and with entries at most $p^e$ for $(b_1, \ldots, b_d) \in \{ (\beta_1, \ldots, \beta_d) \in \mathbb{Z}^{\oplus d} \mid 0 \leq \beta_i < p^e \}$.
    Since the function
    \begin{eqnarray*}
        \sA \xrightarrow{(a_1, \ldots, a_c) \mapsto (a_1, \ldots, a_c,0, \ldots, 0)} \sB
    \end{eqnarray*}
    is order preserving, if $\beta = (\alpha_1, \ldots, \alpha_c, 0, \ldots, 0) \in \sB$ with immediate successor $\beta' \in \sB \cup \{\infty\}$, it follows that $I_\alpha S_\bq \subseteq J_\beta$ and $I_{\alpha'}S_\bq \subseteq J_{\beta'}$. In particular, we see that the image of $x$ in $S_\bq$ lies in $J_\beta \setminus J_{\beta'}$, whence the image of $x$ is not in $\bq^{[p^e]}S_\bq$ and so $(\lambda_{x,e})_\bq$ is pure by \autoref{prop:regular-bracket-powers}~\ref{prop:regular-bracket-powers.2} as desired.
\end{proof}

\begin{remark}
    Recall that by \autoref{prop:J0-J1-equivalence}, the assumption of $S/\p$ being J-0 for all prime ideals $\p$ of $S$ is equivalent to $S/\p$ being J-1 (i.e. $S/\p$ having open regular locus) for all prime ideals $\p$ of $S$. 
\end{remark}

\begin{remark}
\label{rem:EnescuYaoSemicont}
In the notation of \autoref{thm:purelocusofregisopen} and when $x = 1$, i.e. for the iterates of Frobenius themselves, the above argument is similar to \cite[Section 3]{EnescuYaoLowerSemicont}. Indeed, tracing through the proof of \cite[Theorem 3.4]{EnescuYaoLowerSemicont}, the stronger condition of excellent and locally equidimensional (rather than requiring $S/\p$ to be J-0 for all prime ideals $\p$ of $S$) is essentially only used to control the normalizing factor on the splitting numbers. Since these factors do not affect positivity, one can also deduce an openness result from the lower semicontinuity of the first splitting number. 
\end{remark}

As an immediate consequence, we are able to recover a previous result of Epstein and Shapiro. An application of \cite[Thm.\ 3.4]{EpsteinShapiroOhmRushII} to the Frobenius map shows that Dedekind domains of prime characteristic are always Frobenius Ohm-Rush, regardless of whether they are excellent. Alternatively, the result also follows directly from \autoref{thm:purelocusofregisopen}.

\begin{corollary}
    \label{cor:Dedekind-domains-FOR}
    Let $S$ be a Dedekind domain of prime characteristic $p > 0$. Then $S$ is Frobenius Ohm-Rush.
\end{corollary}

\begin{proof}
    Since a prime ideal $\p$ of $S$ is either the zero ideal or a maximal ideal, it follows that $S/\p$ is always regular, and hence has open regular locus. Therefore for any $c \in S$, the map
    \begin{align*}
        S &\rightarrow F_*S\\
        1 &\mapsto F_*c
    \end{align*}
has open pure locus by \autoref{thm:purelocusofregisopen}. Furthermore, since $S_\p$ is a field or a DVR, $(F_*S)_\p = F_*(S_\p)$ is always an Ohm-Rush $S_\p$-module, that is, $F_*S$ is a locally Ohm-Rush $S$-module (the fact that $F_*(S_\p)$ is Ohm-Rush over $S_\p$ when $S_\p$ is a DVR follows by \cite[Prop.\ 2.1]{OhmRu-content}).
Since Dedekind domains are Pr{\"u}fer and $F_*S$ is torsion-free over $S$, we now get that $F_*S$ is Ohm-Rush by \cite[Proposition 3.7.2]{DattaEpsteinTucker}.
\end{proof}

Furthermore, we are able to show that excellent regular rings of prime characteristic satisfy the Frobenius Ohm-Rush property up to radical ideals. 

\begin{theorem}
    \label{thm:FOR-upto-radical}
    Let $S$ be a regular ring of prime characteristic $p > 0$ such that $S/\p$ is J-0 for all prime ideals $\p$ of $S$. If $S$ is locally FOR (for example, if $S$ is locally excellent), then for all $x \in S$ and for all $e \in \mathbb{Z}_{>0}$, there exists a smallest radical ideal $I$ of $S$ such that $x \in I^{[p^e]}$. Equivalently, for any collection of radical ideals $\{I_\alpha\}_{\alpha \in A}$ of $S$, 
    $
    \bigcap_{\alpha \in A} I_\alpha^{[p^e]} = \left(\bigcap_{\alpha \in A} I_\alpha\right)^{[p^e]}.    
    $    
\end{theorem}

\begin{proof}
    By assumption and \autoref{lem:FORT-FIF-FOR-iterated-Frobenius}, for all $e \in \mathbb{Z}_{> 0}$, $F^e_*S$ is a flat and locally Ohm-Rush $S$-module. Thus, the Theorem follows by \autoref{thm:purelocusofregisopen} and \cite[Proposition 3.7.3]{DattaEpsteinTucker}. The latter needs $F^e_*S$ to be locally Ohm-Rush over $S$. Note that locally excellent regular rings of prime characteristic are locally FOR by \autoref{thm:FORT-F-intersection-flat}.
\end{proof}

\begin{theorem}
\label{thm:opennessforquotientsofregular}
    Suppose that $S$ is a regular ring of prime characteristic $p > 0$. Let $R = S/I$ for some ideal $I \subseteq S$. For any $r \in R$ and $e \in \mathbb{Z}_{>0}$, the pure locus of 
    \begin{align*}
    R &\xrightarrow{\lambda_{r,e}} F^e_*R\\
    1 &\mapsto F^e_*r
    \end{align*} 
    is open in $\Spec(R)$ in each of the following cases:
    \begin{enumerate}[label=\textnormal{(\arabic*)}]
        \item $S/\bp$ is J-0 for all $\bp \in \Spec(S)$. \label{thm:opennessforquotientsofregular.1}
        \item $S$ is FOR. \label{thm:opennessforquotientsofregular.2}
    \end{enumerate}
\end{theorem}

\begin{proof}
\ref{thm:opennessforquotientsofregular.1} Let $\tilde{r} \in S$ be a lift of $R$ in $S$, so that $\tilde{r} + I = r$. 
For each $x \in (I^{[p^e]}:_S I)\tilde{r}$, let $U_x \subseteq \Spec(S)$ denote the pure locus of $S \xrightarrow{1 \mapsto F^e_*x} F^e_*S$ as a map of $S$-modules. By \autoref{thm:purelocusofregisopen}, $U = \bigcup_{x \in (I^{[p^e]}:_S I)\tilde{r}} U_x$ is open in $\Spec(S)$.

For each $\bq \in \Spec(S)$ with $I \subseteq \bq$, it follows from \autoref{lem:extnofFedder} that $(\lambda_{r,e})_\bq$ is pure if and only if $(I^{[p^e]}:_S I)\tilde{r} \not\subseteq \bq^{[p^e]}$ (here we are implicitly using \autoref{prop:regular-bracket-powers}~\ref{prop:regular-bracket-powers.1} because Frobenius is flat on S). Equivalently, we have that $\bq$ is in the pure locus of $\lambda_{r,e}$ if and only if there exists some $x \in (I^{[p^e]}:_S I)\tilde{r}$ with $\bq \in U_x$ by \autoref{prop:regular-bracket-powers}. Thus, identifying $\Spec(R)$ as a closed subset of $\Spec(S)$ with the subspace topology, $\Spec(R) \cap U$ is the pure locus of $\lambda_{r,e}$ and is thus open in $\Spec(R)$.

\ref{thm:opennessforquotientsofregular.2} If $S$ is FOR, then for any $x \in S$, the pure locus of $S \xrightarrow{1 \mapsto F^e_*x} F^e_*S$ is 
$
\Spec(S) \setminus \mathbf{V}((x)^{[1/p^e]}),   
$
by \cite[Lemma 3.6.4]{DattaEpsteinTucker}. Recall that $(x)^{[1/p^e]}$ denotes the smallest ideal $J$ of $S$ such that $x \in J^{[p^e]}$. In other words,
$
(x)^{[1/p^e]} = c_{F^e_*S}(F^e_*x).    
$
With this fact, one can now repeat the proof of \ref{thm:opennessforquotientsofregular.1} verbatim to obtain~\ref{thm:opennessforquotientsofregular.2}.
\end{proof}

\begin{remark}
\label{rem:openness-results-timeline}
    In \cite{HochsterYaoGenericLocalDuality}, Hochster and Yao have proved \autoref{thm:opennessforquotientsofregular} in the setting where $S$ is an excellent Cohen-Macaulay ring and $R$ is a quotient of $S$ that is $S_2$ (a version of this result was shared by Hochster-Yao in private communication with us a few years ago). In our setup, we do not need to impose any restrictions on $R$, although in the excellent case our $S$ is more restrictive because it is regular (as opposed to being CM). Very recently, in \cite{LyuUniformBounds}, Lyu builds on \cite{HochsterYaoGenericLocalDuality} to remove their assumptions that $R$ is $S_2$ or a quotient of an excellent Cohen-Macaulay ring (see \cite[A.2.4]{LyuUniformBounds}). Thus, Lyu's result recovers \autoref{thm:opennessforquotientsofregular}~\ref{thm:opennessforquotientsofregular.1}. However, our proof of \autoref{thm:opennessforquotientsofregular} is simpler than the techniques of \cite{HochsterYaoGenericLocalDuality,LyuUniformBounds} and may be of independent interest. Additionally, \autoref{thm:opennessforquotientsofregular} appeared initially in the first version of \cite{DattaEpsteinTucker} in 2023, which we subsequently split into two papers the second of which is the current one.
\end{remark}

\begin{theorem}
    \label{thm:FPureLocusOpenGeneral}
    Suppose that $S$ is a regular ring of characteristic $p > 0$. If $R = S/I$ for some ideal $I \subseteq S$, then the locus in $\Spec(R)$ where $R$ is $F$-pure is open in each of the following cases:
    \begin{enumerate}[label=\textnormal{(\arabic*)}]
        \item $S/\p$ is J-0 for all prime ideals $\p$ of $S$. \label{thm:FPureLocusOpenGeneral.1}
        \item $S$ is FOR. \label{thm:FPureLocusOpenGeneral.2}
    \end{enumerate}
\end{theorem}

\begin{proof}
    That the $F$-pure locus of $R$ is open is immediate from \autoref{thm:opennessforquotientsofregular}. 
\end{proof}

We next state a consequence of the results established thus far, which was shown in \cite[Corollary 3.5]{Murayama:TheGammaConstructionAndAsymptotic} using the $\Gamma$-construction. The point of reproving Murayama's result is to illustrate how the techniques of this paper allow one to avoid using the $\Gamma$-construction. Results in a similar vein will also appear in forthcoming work \cite{DESTPhantom}. See also \autoref{subsec:Kevin-hates-Gamma}.

\begin{corollary}\cite[Corollary 3.5]{Murayama:TheGammaConstructionAndAsymptotic}
    \label{cor:F-pure-locus-eft-G-ring}
    Let $R$ be essentially of finite type over a Noetherian local $G$-ring $(A, \fm)$
    of prime characteristic $p > 0$. Then the $F$-pure locus of $R$ is open.
\end{corollary}

\begin{proof}
    Since $A \to \widehat{A}$ is a regular map, the base change map 
    $R \to R \otimes_A \widehat{A}$ is regular as well by \cite[Prop. 6.8.3(iii)]{EGA_IV_II}.
    Let $\p$ be a prime ideal of $R$ and $\fP$ be a prime ideal of 
    $R \otimes_A \widehat{A}$ that lies over $\p$ (such a prime exists because 
    $R \to R \otimes_A \widehat{A}$ is faithfully flat). Then $R_\p \to (R \otimes_A \widehat{A})_{\fP}$ is a faithfully flat regular map. 
    We now claim that $R_\p$ is $F$-pure if and only if 
    $(R \otimes_A \widehat{A})_\fP$ is $F$-pure. Here the `if' implication 
    follows by faithfully flat descent of $F$-purity. For the `only if' implication we use the Radu-Andr\'e theorem (see \autoref{thm:Radu-Andre-Dumitrescu}) which implies that the relative Frobenius 
    \[
    F_{(R \otimes_A \widehat{A})_{\fP}/R_\p} \colon F_*R \otimes_R (R \otimes_A \widehat{A})_\fP  \to F_*(R \otimes_A \widehat{A})_\fP
    \]
    is faithfully flat hence pure. Then the Frobenius on $(R \otimes_A \widehat{A})_\fP$ 
    is pure because it is the composition of the pure maps
    \[
    (R \otimes_A \widehat{A})_\fP \xrightarrow{F_R \otimes_R \id_{(R \otimes_A \widehat{A})_\fP}} F_*R \otimes_R (R \otimes_A \widehat{A})_\fP \xrightarrow{F_{(R \otimes_A \widehat{A})_{\fP}/R_\p} }
    F_*(R \otimes_A \widehat{A})_\fP.
    \]
    Therefore, if $W \subset \Spec(R)$ is the $F$-pure locus of $R$ then the preimage $Z$ of $W$ in $\Spec(R \otimes_A \widehat{A})$ is the $F$-pure locus of $R \otimes_A \widehat{A}$. Now $Z$ is open by  \autoref{thm:FPureLocusOpenGeneral} because $R \otimes_A \widehat{A}$ is the homomorphic image of a FOR regular ring by \autoref{thm:descent-F-intersection-flat}  (this is where we avoid the $\Gamma$-construction). Then $W$ is also open by \cite[Cor.\ 2.3.12]{EGA_IV_II}.
\end{proof}

\subsection{The variants of strong \emph{F}-regularity}
\label{subsec:variant-strong-F-regular}
Hochster and Huneke introduced strong \emph{F}-regularity \cite{HH-strong-F-regularity} as a version of the notion of weak \emph{F}-regularity that is stable under localization. Their definition was originally made in the $F$-finite Noetherian setting. There are some natural ways to generalize this definition outside the $F$-finite setting that we now recall.

\begin{definition}
    \label{def:variants-strongly-F-regular}
    Let $R$ be a ring of prime characteristic $p > 0$. For an element $c \in R$, let
    $
    \lambda_{c,e} \colon R \to F^e_*R
    $
    be the unique map that sends $1 \mapsto F^e_*c$.
    \begin{enumerate}
        \item[$\bullet$] We say that $R$ is \emph{split $F$-regular} if for any nonzerodivisor $c \in R$, there exists an integer $e > 0$ such that $\lambda_{c,e}$ admits a left-inverse in $\Mod_R$. 
        \item[$\bullet$] We say that $R$ is \emph{$F$-pure regular} if for any nonzerodivisor $c \in R$, there exists an integer $e > 0$ such that $\lambda_{c,e}$ is $R$-pure. 
        \item[$\bullet$] We say $R$ is \emph{strongly $F$-regular} if for all $R$-modules $M$ and for all submodules $N$ of $M$, $N^*_M = N$, that is, $N$ is tightly closed in $M$. 
    \end{enumerate}
\end{definition}

\begin{remark}
    \label{rem:variants-strongly-F-regular}
    \mbox{}
\begin{enumerate}
    \item What we call $F$-pure regular was called \emph{very strongly $F$-regular} in \cite[Def.\ 3.4]{HashimotoF-pure-homomorphisms}. The $F$-pure regular terminology appears to have been first used in \cite{DattaSmithValuations}, and we find this terminology to be more descriptive. 
    \item The definition of strong $F$-regularity was suggested by Hochster in unpublished course notes on tight closure theory \cite[pg.\ 166]{HochsterFoundations}. 
    \item It is clear that split $F$-regular implies $F$-pure regular. Moreover, note that $F$-pure regularity localizes. That is, if $R$ is $F$-pure regular, then for any prime ideal $\p$ of $R$, $R_\p$ is also $F$-pure regular. 
    \item Hashimoto showed that strong $F$-regularity is equivalent to being locally $F$-pure regular \cite[Lem.\ 3.6]{HashimotoF-pure-homomorphisms}. Thus, $F$-pure regular implies strongly $F$-regular. 
    \item It was not clear for a long time if strong $F$-regularity is equivalent to $F$-pure regularity for excellent rings of characteristic $p$ in general. A few cases partial cases of interest are proved in \cite{HashimotoF-pure-homomorphisms,DattaMurayamaFsolidity}. However, this problem has been recently settled by work of Hochster and Yao \cite{HochsterYaoGenericLocalDuality} for all excellent ring and more generally in \cite{LyuUniformBounds}. 
    \item The implications split $F$-regularity $\implies$ $F$-pure regular $\implies$ strongly $F$-regular are strict, even for the class of regular rings. For example, any regular local ring is $F$-pure regular (see for instance \cite[Thm.\ 6.2.1]{DattaSmithValuations}). However, there exist excellent Henselian DVRs $R$ of characteristic $p > 0$ such that $\Hom_R(F^e_*R,R) = 0$ for all integers $e > 0$ \cite{DattaMurayamaTate}. Such a ring can clearly never be split $F$-regular. It is also well-known that all regular rings of prime characteristic are strongly $F$-regular. However, there exist non-excellent regular rings that are not $F$-pure regular \cite[Sec.\ 6]{HochsterYaoSFRsmallCM}. In fact, an example can already be constructed via \cite[Rem.\ (2)]{EisenbudHochster} because F-pure regularity of regular rings is intimately related to the uniform Artin-Rees property. See \autoref{subsec:Artin-Rees}.
\end{enumerate}
\end{remark}

We now explore the variants of strong $F$-regularity for FOR and FORT regular rings and their quotients. We also show that the equivalence between $F$-pure regularity and strong $F$-regularity holds for quotients of all J-2 regular rings. Note that such rings need not be excellent.

\begin{theorem}
    \label{thm:strong-F-regularity-FOR-FORT}
    Let $S$ be a regular ring of prime characteristic $p > 0$. Let $R \coloneqq S/I$ be a quotient of $S$. Then we have the following:
    \begin{enumerate}[label=\textnormal{(\arabic*)}]
        \item Suppose $S$ is FOR. Then $R$ is $F$-pure regular if and only if $R$ is strongly $F$-regular. \label{thm:strong-F-regularity-FOR-FORT.1}
        \item Suppose for all prime ideals $\p$ of $S$, $S/\p$ is J-0. Then $R$ is $F$-pure regular if and only if $R$ is strongly $F$-regular. \label{thm:strong-F-regularity-FOR-FORT.2}
          
        \item \label{thm:strong-F-regularity-FOR-FORT.3} Suppose $S$ is FORT. Then the following are equivalent:
        \begin{enumerate}[label=\textnormal{(\alph*)}]
            \item $R$ is split $F$-regular.\label{thm:strong-F-regularity-FOR-FORT.3a}
            \item $R$ is $F$-pure regular.\label{thm:strong-F-regularity-FOR-FORT.3b}
            \item $R$ is strongly $F$-regular.\label{thm:strong-F-regularity-FOR-FORT.3c}   
        \end{enumerate}  
    \end{enumerate}
\end{theorem}

\begin{proof}
    We will use throughout the fact that $R$ being strongly $F$-regular is equivalent to all local rings of $R$ being $F$-pure regular \cite[Lem.\ 3.6]{HashimotoF-pure-homomorphisms}. The same result then implies that any $F$-pure regular ring is also strongly $F$-regular. 

    By the above discussion, it remains to show in both \ref{thm:strong-F-regularity-FOR-FORT.1} and \ref{thm:strong-F-regularity-FOR-FORT.2} that if $R_{\mathfrak q}$ is $F$-pure regular for all prime ideals $\mathfrak q$ of $R$, then $R$ is $F$-pure regular. But this follows by a spreading out argument afforded by \autoref{thm:opennessforquotientsofregular} and the quasicompactness of $\Spec(R)$.

    \ref{thm:strong-F-regularity-FOR-FORT.3} We already know that \ref{thm:strong-F-regularity-FOR-FORT.3a} $\implies$ \ref{thm:strong-F-regularity-FOR-FORT.3b} $\implies$ \ref{thm:strong-F-regularity-FOR-FORT.3c} holds unconditionally for any Noetherian ring. Since FORT rings are FOR, we also have \ref{thm:strong-F-regularity-FOR-FORT.3c} $\implies$ \ref{thm:strong-F-regularity-FOR-FORT.3b} by \ref{thm:strong-F-regularity-FOR-FORT.1}. Thus, it remains to show \ref{thm:strong-F-regularity-FOR-FORT.3b} $\implies$ \ref{thm:strong-F-regularity-FOR-FORT.3a}. It suffices to show that if $r \in R$ such that 
    \begin{align*}
    \lambda_{r,e} \colon R &\to F^e_*R\\
    1 &\mapsto F^e_*r
    \end{align*}
    is pure as a map of $R$-modules, then $\lambda_{r,e}$ admits a left-inverse (i.e. it splits). 
    But if $\lambda_{r,e}$ is $R$-pure, by \autoref{lem:extnofFedder} for all prime ideals $\bq$ of $S$ such that $I \subseteq \bq$, $(I^{[p^e]}\colon I)\tilde{r} \nsubseteq \bq^{[p^e]}$. But this means that for all such $\bq$,
    $
    \Tr_{F^e_*S}(F^e_*(I^{[p^e]}\colon I)\tilde{r}) \nsubseteq \bq   
    $
    by \autoref{lem:trace-equals-content}. Thus, 
    $
        \Tr_{F^e_*S}(F^e_*(I^{[p^e]}\colon I)\tilde{r}) + I = S.   
    $
    Choose $c \in (I^{[p^e]}\colon I)$ and $\varphi \colon F^e_*S \to S$ such that 
    $
    \varphi(F^e_*c\tilde{r}) + I = 1 + I    
    $
    in $R$. Since $c \in (I^{[p^e]}:I)$, we have seen that $\varphi(F^e_*c \cdot)$ induces an $R$-linear map $\overline{\varphi} \colon F^e_*R \to R$ such that the following diagram commutes:
    \[
        \xymatrix@R+1pc@C+1pc{F^e_*S \ar[r]^{\varphi(F^e_*c \cdot)} \ar[d]_{F^e_*\pi} &  S \ar[d]^{\pi} \\  F^e_*{R} \ar[r]_{\overline{\varphi}} & {R}.}  
\]   
Here the vertical maps are induced by the canonical projection $\pi \colon S \to S/I =: R$. 
Then
$
\overline{\varphi}(F^e_*r) = \overline{\varphi} \circ F^e_*\pi(F^e_*\tilde{r}) = \pi \circ \varphi(F^e_*c\tilde{r}) = \varphi(F^e_*c\tilde{r}) + I = 1 + I.
$
Thus, $\overline{\varphi}$ is a left-inverse of $\lambda_{r,e}$.
\end{proof}

\subsection{An application of uniform Artin-Rees to \texorpdfstring{$F$}{F}-pure regularity}
\label{subsec:Artin-Rees}
One can use the ideal-theoretic purity criterion for maps $\lambda_{c,e} \colon S \to F^e_*S$ of commutative rings $S$ with flat Frobenius (\autoref{prop:regular-bracket-powers}) to give a different proof of the $F$-pure regularity of regular rings $S$ of prime characteristic $p > 0$ with J-0 (equivalently, J-1 by \autoref{prop:J0-J1-equivalence}) quotients $S/\p$ for all prime ideals $\p$. The proof relies on Duncan and O'Carroll's uniform Artin-Rees theorem, which we now recall.

\begin{citedthm}[{\cite[Thm.\ on p.\ 203]{DuncanO'CarrollArtin-Rees}}]
    \label{thm:uniform-Artin-Rees}
    Let $R$ be a Noetherian ring such that for all prime ideals $\p \in \Spec(R)$,
    $R/\p$ is J-0. 
    Let $M$ be a finitely generated $R$-module and
    $N$ a submodule of $M$. There exists a integer $k > 0$ depending only on $M$ and $N$
    such that for all maximal ideals $\fm \in \Spec(R)$ and for all integers $n \geq 0$, 
    $
    N \cap \fm^{n+k}M  = \fm^n(N \cap \fm^kM).
    $
    \end{citedthm}
    
    In other words, the bound $k$ in the usual Artin--Rees lemma can be chosen independently
    of the choice of the maximal ideal $\fm$. We note that the uniform Artin-Rees property has intimate connections with the theory of test elements in tight closure \cite{HunekeUniformArtinRees}.
    
    \begin{remark}
        Duncan and O'Carroll originally stated their theorem when $R$ is J-2 in
      \cite{DuncanO'CarrollArtin-Rees}. In \cite[Rem.\ on p.\ 49]{DuncanO'CarrollZariskiRegularity}, they noted that the
      weaker hypothesis that $R/\p$ is J-1 for all prime ideals $\p$ of $R$ suffices in order for the conclusion of \autoref{thm:uniform-Artin-Rees} to hold. However, as observed, the J-1 assertion on the prime cyclic quotients of $R$ is equivalent to $R/\p$ being J-0 for all prime ideals $\p$ by \autoref{prop:J0-J1-equivalence}.
    \end{remark}

    \begin{corollary}
        \label{cor:Artin-Rees-elements}
        Let $R$ be a Noetherian ring of prime characteristic $p > 0$ such that $R/\p$ is J-0 for all prime ideals $\p$ of $R$. Let $f \in R$ be a nonzerodivisor that
        is not a unit. Then there exists an integer $e > 0$ (that only depends on $f$) such that for
        all proper ideals $I$ of $R$, $f \notin I^{[p^e]}$.
        \end{corollary}
        
        \begin{proof}
        Choose $k$ as in Theorem \ref{thm:uniform-Artin-Rees} for $M = R$ and $N = (f)$.
        We claim that for $s = k + 1$, $f$ is not contained in $I^s$, for any proper
        ideal $I$ of $R$. Indeed, assume otherwise. Then without loss of generality,
        we may assume that $I = \fm$, for a maximal ideal $\fm$ and $f \in \fm^s$. Then by
        Theorem \ref{thm:uniform-Artin-Rees},
        $
        (f) = (f) \cap \fm^s = (f) \cap \fm^{k+1} = \fm((f) \cap \fm^k) = \fm(f).
        $
        Thus, there exists $x \in \fm$ such that $f = xf$. Then $(1-x)f = 0$, and since
        $f$ is a nonzerodivisor, this implies $1 - x = 0$. However, this is impossible
        because $x$ is not a unit. Now choose $e \gg 0$ such that $p^e > s$.
        Then for all proper ideals $I$, we have
        $
        I^{[p^e]} \subseteq I^s,
        $
        and so, $f \notin I^{[p^e]}$. \qedhere
        \end{proof}

        We can now easily deduce the main result on $F$-pure regularity.

        \begin{proposition}
            \label{prop:uniform-Artin-Rees-F-pure-regular}
            Let $S$ be a regular ring of prime characteristic $p > 0$ such that $S/\p$ is J-0 for all prime ideals $\p$ of $S$. Then $S$ is $F$-pure regular.
        \end{proposition}

        \begin{proof}
            Pick $x \in S$ a non-zero divisor and choose $e > 0$ such that for all proper ideals $I$ of $S$, $x \notin I^{[p^e]}$ by \autoref{cor:Artin-Rees-elements}. Then the map
            \begin{align*}
                \lambda_{x,e} \colon S &\to F^e_*S\\
                1 &\mapsto F^e_*x
            \end{align*} 
            is $S$-pure by \autoref{prop:regular-bracket-powers}~\ref{prop:regular-bracket-powers.2}. Thus, $S$ is $F$-pure regular by definition.
        \end{proof}

        \begin{remark}
            The argument in the proof of \autoref{prop:uniform-Artin-Rees-F-pure-regular} is based on unpublished work of the first author and Takumi Murayama. The argument gives an alternate proof of \cite[Cor.\ 2.18]{HochsterYaoSFRsmallCM} and of the special case of \autoref{thm:strong-F-regularity-FOR-FORT}~\ref{thm:strong-F-regularity-FOR-FORT.2} when $R = S$.
        \end{remark}

    \subsection{Bypassing the \texorpdfstring{$\Gamma$}{Gamma}-construction}
    \label{subsec:Kevin-hates-Gamma}
The FOR and FORT properties can be used to give different proofs of results that have relied on the $\Gamma$-construction \cite{HochsterHunekeFRegularityTestElementsBaseChange}. One instance of this has already been illustrated in \autoref{cor:F-pure-locus-eft-G-ring}. For another example, we show that \autoref{thm:strong-F-regularity-FOR-FORT} immediately implies the following result which was proved by the first author and Murayama by reducing to the $F$-finite setting.

\begin{corollary}\cite[Thm.\ 3.1.1]{DattaMurayamaFsolidity}
    \label{cor:deducing-F-regularity-Gamma-alternative}
    Let $(A,\fm, \kappa)$ be a Noetherian complete local of prime characteristic $p > 0$ and let $R$ be an essentially of finite type $A$-algebra. Then the following are equivalent:
        \begin{enumerate}[label=\textnormal{(\arabic*)}]
            \item $R$ is split $F$-regular.\label{cor:deducing-F-regularity-Gamma-alternative.1}
            \item $R$ is $F$-pure regular.\label{cor:deducing-F-regularity-Gamma-alternative.2}
            \item $R$ is strongly $F$-regular.\label{cor:deducing-F-regularity-Gamma-alternative.3}  
        \end{enumerate}
\end{corollary}

\begin{proof}
    $R$ is a homomorphic image of the localization of a polynomial ring over a power series ring over $\kappa$. The latter ring is FORT by \autoref{cor:variables-FORT} and \autoref{cor:FORT-localization}, and so, we get the desired result by \autoref{thm:strong-F-regularity-FOR-FORT}~\ref{thm:strong-F-regularity-FOR-FORT.3}.
\end{proof}

The $\Gamma$-construction allows one to reduce questions about the Frobenius endomorphism for rings that are essentially of finite type over a local $G$-ring of prime characteristic to the $F$-finite setting \cite{HochsterHunekeFRegularityTestElementsBaseChange,HashimotoF-pure-homomorphisms,Murayama:TheGammaConstructionAndAsymptotic,DattaMurayamaFsolidity}. The advantage of working with the FOR and FORT properties is that one can completely bypass the $\Gamma$-construction and work directly with rings that are quotients of FOR and FORT regular rings. This circle of ideas will be pursued further, especially in connection with the theory of test ideals and $F$-compatible ideals, in forthcoming work \cite{DESTPhantom}. We will show that one can build a robust theory of test ideals for quotients of FOR regular rings. 

\section{Tate algebras over non-Archimedean fields}

\subsection{Background on non-Archimedean fields, Banach spaces and Tate algebras}
\label{subsec.BackgroundOnTate}
We begin by introducing the notion of Tate algebras over non-Archimedean fields. A reference for many of the basic notions introduced in this section is \cite{Boschrigid}. Recall that a \emph{non-Archimedean norm} on a field $k$ is a function
$
|\cdot| \colon k \to \mathbb{R}_{\geq 0}    
$
that satisfies the following properties:
\begin{enumerate}
    \item[$\bullet$] $|x| = 0$ if and only if $x = 0$,
    \item[$\bullet$] $|xy| = |x||y|$, and
    \item[$\bullet$] $|x + y| \leq \max\{|x|,|y|\}$.
\end{enumerate}
In other words, $|\cdot|$ is a non-Archimedean multiplicative valuation of $k$ of rank $1$.
A field $k$ equipped with a non-Archimedean norm $|\cdot|$ is called a \emph{real-valued field} and is denoted $(k,|\cdot|)$.

The valuation ring of $k$ is the subring $k^\circ \coloneqq \{x \in k \colon |x| \leq 1\}$. This is a local ring with maximal ideal $k^{\circ \circ} \coloneqq \{x \in k \colon |x| < 1\}$. Note that $k^\circ$ has Krull dimension $1$. 

The following well-known Lemma will be useful.

\begin{lemma}\cite[Prop.\ 2.1/2]{Boschrigid}
    \label{lem:norm-becomes-max}
    Let $(k,|\cdot|)$ be a real-valued field. If $x, y \in k$ such that $|x| \neq |y|$, then $|x+y| = \max\{|x|,|y|\}$.
\end{lemma}

\begin{definition}
    \label{def:NA-field}
    A \emph{non-Archimedean} (abbrev. NA) field is a real-valued field $(k,|\cdot|)$ such that $k$ is complete with respect to the metric $|x - y|$ that is induced by $|\cdot|$ and such that $|k^\times| \neq {1}$, that is, $k$ is non-trivially valued.
\end{definition}

All real-valued fields will be NA in what follows. 

\begin{definition}
    \label{def:normed-space}
    Let $(k,|\cdot|)$ be a NA field. A \emph{normed space} $(E,||\cdot||)$ over $k$ is a $k$-vector space along with a norm $||\cdot|| \colon E \to \mathbb{R}_{\geq 0}$ that satisfies the following properties:
    \begin{enumerate}
        \item[$\bullet$] $||x|| = 0$ if and only if $x = 0$,
        \item[$\bullet$] $||x+y|| \leq \max\{||x||,||y||\}$, and
        \item[$\bullet$] if $c \in k$ and $x \in E$, then $||cx|| = |c|\cdot||x||$.
    \end{enumerate}
    If $E$ is complete with respect to the metric induced by $||\cdot||$, then $E$ is called a \emph{Banach space over $k$} or a \emph{$k$-Banach space}.
\end{definition}

    \begin{remarks}
        \label{rem:finite-dim-Banach-spaces}
        Let $(k,|\cdot|)$ be a NA field.
        \begin{enumerate}
            \item \label{rem:unique-norm-finite-dim} Let $E$ be a finite dimensional $k$-vector space. Then $E$ can be given the structure of a $k$-Banach space as follows: if we fix a basis $\{x_1,\dots,x_n\}$ of $E$ and express any $x \in E$ in terms of the basis as $x = \sum_{i=1}^n a_ix_i$, then one can define
            \[
            ||x|| \coloneqq \max\{|a_i| \colon i = 1,\dots,n\}.     
            \]
            Even though this norm depends on the choice of a basis of $E$, one can show that every norm on $E$ that gives $E$ the structure of a $k$-Banach space is equivalent to the one just defined \cite[Appendix A, Thm.\ 1]{Boschrigid}.

            \item \label{rem:unique-norm-algebraic-ext} Let $\ell$ be an algebraic extension of $k$. Expressing $\ell$ as a filtered colimit of its finite subextensions, one can show that there exists a unique (not just equivalent) norm on $\ell$ that extends the norm on $k$ \cite[Appendix A, Thm.\ 3]{Boschrigid}. However, if $[\ell \colon k] = \infty$, then $\ell$ need not be complete with respect to this norm, that is, $\ell$ need not be a $k$-Banach algebra.
        \end{enumerate}
    \end{remarks}

    \begin{definition}
        \label{def:Banach-algebra}
        A \emph{Banach $k$-algebra} $(A,||\cdot||)$ is a $k$-algebra $A$ such that $A$ is a $k$-Banach space and such that the norm $||\cdot||$ on $A$ satisfies the following additional property: for all $x,y \in A$,
    \[
    ||xy|| \leq ||x||\cdot||y||.    
    \]
    This last property insures that the multiplication operation on $A$ is continuous. The norm $||\cdot||$ is \emph{multiplicative} if equality holds in the above inequality.
    \end{definition}

All $k$-Banach algebras in this paper have multiplicative norms. The main example of a $k$-Banach algebra for us is a Tate algebra.

\begin{definition}
    \label{def:Tate-algebra}
    Let $(k,|\cdot|)$ be a NA field. For every positive integer $n > 0$, the \emph{Tate algebra} in $n$ indeterminates over $k$, denoted, $T_n(k)$, is the $k$-subalgebra of the formal power series ring $k\llbracket X_1,\dots,X_n\rrbracket $ consisting of those power series
    $
    \sum_{\nu \in \mathbb{Z}_{\geq 0}^n} a_\nu X^{\nu}
    $
    (written in multi-index notation) such that $|a_\nu| \to 0$ as $\nu_1 + \dots + \nu_n \to \infty$. An element of $T_n(k)$ is called a \emph{restricted power series}.  A homomorphic image of $T_n(k)$ is called an \emph{affinoid algebra}.
\end{definition}

The Tate algebra becomes a $k$-Banach algebra equipped with the \emph{Gauss norm} \cite[Prop.\ 2.2/3]{Boschrigid}, which is defined as follows: for all $\sum_{\nu \in \mathbb{Z}_{\geq 0}^n} a_\nu X^{\nu} \in T_n(k)$, 
$
\bigg|\bigg|\sum_{\nu \in \mathbb{Z}_{\geq 0}^n} a_\nu X^{\nu}\bigg|\bigg|   \coloneqq \max\{|a_\nu| \colon \nu \in \mathbb{Z}_{\geq 0}^n\}.
$
The Gauss norm on $T_n(k)$ is multiplicative \cite[pp. 13-14]{Boschrigid}. We will always consider $T_n(k)$ as a $k$-Banach algebra with respect to the Gauss norm.

Remarkably, $T_n(k)$ shares many of the properties of the polynomial ring $k[X_1,\dots,X_n]$. We summarize some of these properties below for the reader's convenience.

\begin{theorem}
    \label{thm:Tate-algebras-properties}
    Let $(k,|\cdot|)$ be a NA field, and let $n$ be a positive integer. Then the Tate algebra $T_n(k)$ satisfies the following properties:
    \begin{enumerate}[label=\textnormal{(\alph*)}]
        \item \label{thm:Tate-Noetherian} $T_n(k)$ is Noetherian.
        \item \label{thm:Tate-UFD} $T_n(k)$ is a unique factorization domain.
        \item \label{thm:Tate-Jacobson} $T_n(k)$ is Jacobson, that is, every radical ideal is the intersection of the maximal ideals containing it.
        \item \label{thm:Tate-dimension} All maximal ideals of $T_n(k)$ have height $n$ and are generated by $n$ elements. In particular, $T_n(k)$ has Krull dimension $n$.
        \item \label{thm:Tate-regular} $T_n(k)$ is regular.
        \item \label{thm:Tate-Euclidean-dim1} $T_1(k)$ is a Euclidean domain with associated Euclidean function $T_1(k)\setminus\{0\} \to \mathbb{Z}_{\geq 0}$ given by mapping a restricted power series $f = \sum_{i=0}^\infty a_iX^i$ to the largest index $N$ such that $|a_N| = ||f||$.
        \item \label{thm:Tate-maximal-ideal-residue} If $\mathfrak m$ is a maximal ideal of $T_n(k)$, then $T_n(k)/\mathfrak{m}$ is a finite extension of $k$.
        \item \label{thm:Tate-excellent} $T_n(k)$ is excellent.
        \item \label{thm:Tate-ideals-closed} Every ideal of $T_n(k)$ is closed in the topology on $T_n(k)$ that is induced by the Gauss norm.
    \end{enumerate}
\end{theorem}

\begin{proof}[Indication of proof]
    \ref{thm:Tate-Noetherian}--\ref{thm:Tate-dimension} are proved in \cite[Props. 2.2/14--17]{Boschrigid}, while \ref{thm:Tate-regular} follows from \ref{thm:Tate-dimension}. Property \ref{thm:Tate-Euclidean-dim1} is proved in \cite[Cor.\ 2.2/10]{Boschrigid} while \ref{thm:Tate-maximal-ideal-residue} follows by \cite[Cor.\ 2.2/12]{Boschrigid}. The most difficult property to prove is \ref{thm:Tate-excellent}, which is shown in \cite[Thm.\ 3.3]{Kiehl-Tate}. Finally, \ref{thm:Tate-ideals-closed} follows by \cite[Cor.\ 2.3/8]{Boschrigid}.
\end{proof}

\subsection{Continuous maps and the Hahn-Banach extension property} Let $(k,|\cdot|)$ be a NA field and $(E,||\cdot||)$ be a non-trivial normed space over $k$. We will use 
$
\Hom^{\cont}_k(E,k)$
to denote the submodule of the $k$-dual space of $E$ consisting of \emph{continuous} functions $E \to k$. In general, it is not always true that $\Hom^{\cont}_k(E,k) \neq 0$. For example, for an arbitrary NA field $k$ of characteristic $p > 0$, this may fail to hold even for $E = k^{1/p}$ \cite[Section 5]{DattaMurayamaTate}. The existence of such pathological NA fields in prime characteristic was used by Datta and Murayama to show that Tate algebras over such fields have no non-zero $T_n(k)$-linear maps $F_*T_n(k) \to T_n(k)$, thereby giving a negative answer to a folklore conjecture in prime characteristic commutative algebra about the existence of Frobenius splittings for excellent $F$-pure rings.

In this section, we recall a class of NA fields, called \emph{spherically complete} fields, for which there always exist non-zero continuous linear functions $E \to k$. First, we recall some basic facts about continuous maps of normed spaces over NA fields. We will say that a subset $S$ of $E$ is \emph{bounded} if there exists $a \in \mathbb{R}_{\geq 0}$ such that $S$ is contained in the closed ball $B_a(0)$ of radius $a$ centered at $0 \in E$. A sequence $(x_n)_{n \in \mathbb{Z}_{\geq 0}}$ in $E$ is bounded if it is bounded as set. Similarly,  $(x_n)_{n \in \mathbb{Z}_{\geq 0}}$ is \emph{null} if $||x_n|| \to 0$ and $n \to \infty$. With these definitions, we have the following characterization of continuous maps of normed spaces over NA fields. Recall that by our convention, all NA fields are non-trivially valued.

\begin{lemma}\cite[Lem.\ 2.11]{DattaMurayamaTate}
    \label{lem:continuous-maps-normed-spaces}
    Let $(k,|\cdot|)$ be a non-Archimedean field and $(E,{||\cdot||}_E)$, $(F,{||\cdot||}_F)$
    be normed spaces over $k$. Then, for a $k$-linear map $f\colon E \rightarrow F$, the 
    following are equivalent:
    \begin{enumerate}[label=\textnormal{(\alph*)}]
      \item\label{lem:continuouscont} $f$ is continuous.
      \item\label{lem:continuousnull} $f$ maps null sequences to null sequences.
      \item\label{lem:null-to-bounded} $f$ maps null sequences to bounded sequences.
      \item\label{lem:bounded} $f$ maps bounded sets to bounded sets.
      \item\label{lem:one-bounded} There exists $a, b \in \mathbb{R}_{> 0}$ such that 
      $f(B_a(0)) \subseteq B_b(0)$.
      \item\label{lem:continuousconstb} There exists $B \in \mathbb{R}_{> 0}$ such that for
        all $x \in E$, we have 
        $||f(x)||_F \leq B\cdot ||x||_E$.
    \end{enumerate}
    \end{lemma}

    \autoref{lem:continuous-maps-normed-spaces} shows that for all continuous linear maps $f \colon (E,||\cdot||_E) \to (F,||\cdot||_F)$,
    \[
    ||f|| \coloneqq \sup_{x \neq 0}\bigg\{\frac{||f(x)||_F}{||x||_E}\bigg\}    
    \]
    is finite, that is, all continuous maps of normed spaces over NA fields are \emph{bounded continuous}. We call $||f||$ the \emph{operator norm} or \emph{Lipschitz norm} of $f$. Under this norm, $\Hom_k^{\cont}(E,F)$ also becomes a normed space. 

    We next introduce the Hahn-Banach extension property for normed spaces over NA fields. This is the NA analog of the Hahn-Banach extension property that holds over $\mathbb{R}$ or $\mathbb{C}$.

    \begin{definition}
        \label{def:Hahn-Banach-extension}
        Let $(k,|\cdot|)$ be a NA field and $(E,||\cdot||)$ be a normed space over $k$. Then we say $E$ satisfies the \emph{Hahn-Banach extension property} if for every subspace $D$ of $E$ and every continuous linear functional $f \colon D \to k$, there exists a continuous linear functional $\tilde{f} \colon E \to k$ such that 
        \begin{enumerate}
            \item[$\bullet$] $\tilde{f}|_D = f$, and
            \item[$\bullet$] $||\tilde{f}|| = ||f||$.
        \end{enumerate}
    \end{definition}

Our goal now is to introduce a class of NA fields such that normed spaces over such fields always satisfy the Hahn-Banach extension property.

\begin{definition}
    \label{def:spherically-complete}
    Let $(k,|\cdot|)$ be a real-valued field. We say $k$ is \emph{spherically complete} if, for every decreasing sequence of closed non-empty balls
    \[
    D_1 \supseteq D_2 \supseteq D_3 \supseteq \dots,    
    \]
    the intersection $\bigcap_{i=1}^\infty D_i$ is also non-empty.
\end{definition}

The point here is that we are not assuming that the balls $D_i$ all have a common center. The above condition is equivalent to following: if $\{D_i\}_{i\in I}$ is a collection of closed balls such that $D_i \cap D_j \neq \emptyset$ for $i \neq j$, then $\bigcap_{i \in I} D_i \neq \emptyset$ \cite[Lem.\ 2.3]{vanRooij}.

\begin{remarks}
    \label{rem:spherically-complete-facts}
    In this remark, we assume $(k,|\cdot|_k)$ is a real-valued field such that $|k^\times|_k$ is not the trivial group.
    \begin{enumerate}
        \item If $k$ is spherically complete, then $k$ complete; see \cite[Pg.\ 5]{Perez-Garcia-Schikhof}. That is, by our convention, a non-trivially valued spherically complete field is automatically non-Archimedean.
        \item If $k$ is a NA field such that $|k^\times|_k \cong \mathbb{Z}$, then $k$ is spherically complete \cite[Thm.\ 1.2.13]{Perez-Garcia-Schikhof}. In other words, a NA field whose corresponding valuation ring is a complete discrete valuation ring is spherically complete.
        \item There exist NA fields that are not spherically complete. For example, $\mathbb{C}_p$, the completion of the algebraic closure of the $p$-adic numbers $\mathbb{Q}_p$, is not spherically complete \cite[Thm.\ 1.2.12]{Perez-Garcia-Schikhof}.
        \item Recall that an extension of real-valued fields $(k,|\cdot|_k) \hookrightarrow (\ell, |\cdot|_\ell)$ is \emph{immediate} if $|k^\times|_k = |\ell^\times|_\ell$ and if the induced map on residue fields $k^\circ/k^{\circ\circ} \hookrightarrow \ell^\circ/\ell^{\circ\circ}$ is an isomorphism. If $k$ has no proper immediate extensions, then $k$ is automatically complete because the completion is otherwise a proper immediate extension. Moreover, one can show that such a field is spherically complete \cite[Thm.\ 4.47]{vanRooij} (see also \cite{KaplanskyMaximalValuation}). An arbitrary real-valued field $(k,|\cdot|_k)$ has a maximal immediate extension \cite{KrullMaximalImmediateExtension} (see also \cite[Thm.\ 4.49]{vanRooij}). Such a maximal immediate extension is called a \emph{spherical completion} of $(k,|\cdot|_k)$. Any two spherical completions of $(k,|\cdot|_k)$ are isomorphic as normed spaces over $k$ \cite[Thm.\ 4.43]{vanRooij}. However, there may \emph{not} exist field isomorphism between spherical completions that fix $k$ \cite[Sec.\ 5]{KaplanskyMaximalValuation} (see also \cite[Thm\ 4.59]{vanRooij}).
    \end{enumerate}
\end{remarks}

The main reason why spherically complete fields are important for  us is that normed spaces over such fields satisfy the Hahn-Banach extension property.
\begin{theorem}
    \label{thm:spherically-complete-Hahn-Banach}
    Let $(k,|\cdot|)$ be a NA field.
    \begin{enumerate}[label=\textnormal{(\alph*)}]
        \item \label{thm:spherically-complete-Hahn-Banach:a} If $k$ is spherically complete, then every normed space over $k$ satisfies the Hahn-Banach extension property.
        \item \label{thm:spherically-complete-Hahn-Banach:b} Suppose there exists an infinite dimensional Banach space over $k$ that satisfies the Hahn-Banach extension property. Then $k$ is spherically complete.
    \end{enumerate}
\end{theorem}

\begin{proof}
    \ref{thm:spherically-complete-Hahn-Banach:a} follows by \cite[Thm.\ 4.8]{vanRooij}, and \ref{thm:spherically-complete-Hahn-Banach:b} follows by \cite[Thm.\ 4.54]{vanRooij}.
\end{proof}

As a consequence, one can deduce the following:

\begin{corollary}
    \label{cor:extending-arbitrary-functionals-finite-dim}
    Let $(k,|\cdot|)$ be a spherically complete NA field and $(E,||\cdot||)$ be a normed space over $k$. If $F$ is a finite dimensional subspace of $E$, then \emph{any} linear functional 
    $
    F \to k    
    $
    extends to a continuous linear functional $E \to k$.
\end{corollary}

\begin{proof}
    Since $F$ is finite dimensional, $F$ is a Banach space and all linear functionals $F \to k$ are continuous by \cite[Thm.\ 3.15, part ii.]{vanRooij}. Since $E$ satisfies the Hahn-Banach extension property by \autoref{thm:spherically-complete-Hahn-Banach}~\ref{thm:spherically-complete-Hahn-Banach:a}, we then see that any functional $F \to k$ extends to a continuous functional $E \to k$.
\end{proof}

\begin{remark}
    \label{rem:things-dont-work-not-spherically-complete}
    Let $(k,|\cdot|)$ be a NA field. If $k$ is not spherically complete, then there always exists a two dimensional normed space $(E,||\cdot||)$ over $k$ and a subspace $D$ of $E$ with a continuous functional $f \colon D \to k$ such that $f$ does \emph{not} admit an extension $\tilde{f} \colon E \to k$ with $||f|| = ||\tilde{f}||$. See \cite[Ex.\ 4.2.9]{Perez-Garcia-Schikhof}. Said differently, if every two-dimensional normed space over a NA field satisfies the Hahn-Banach extension property, then the field is spherically complete.
\end{remark}

In the Hahn-Banach extension property, we required the extension of the functional to have the same (Lipschitz) norm as the functional. It is natural to ask what happens if we relax this latter requirement. For example, can we extend continuous linear functionals over fields that are not spherically complete if we do not require the extension to be norm preserving? In order to answer this question we define a modification of the Hahn-Banach property. The notion is inspired by \autoref{thm:+-epsilon-Hahn-Banach} and was introduced in \cite[Definition 2.12]{DattaMurayamaTate}.

\begin{definition}
    \label{def:1+t-Hahn-Banach}
    Let $(k,|\cdot|)$ be a NA field and $(E,||\cdot||)$ be a normed space over $k$. We say that $E$ satisfies the \emph{$(1+\epsilon)$-Hahn-Banach extension property} if for every subspace $D$ of $E$, every continuous linear functional $f \colon D \to k$ and every real number $\epsilon > 0$, there exists a continuous linear functional $\tilde{f}_\epsilon \colon E \to k$ such that 
    \begin{enumerate}
        \item[$\bullet$] $\tilde{f_\epsilon}|_D = f$, and
        \item[$\bullet$] $||\tilde{f_\epsilon}|| \leq (1+\epsilon)||f||$.
    \end{enumerate}
\end{definition}

With this relaxation, we have the following extension theorem for continuous linear functionals over an arbitrary NA field.

\begin{theorem}
    \label{thm:+-epsilon-Hahn-Banach}
    Suppose $(E,||\cdot||)$ is a normed space over a NA field $(k,|\cdot|)$ that is of \emph{countable type}, that is, $E$ has a dense subspace with a countable $k$-linear basis. Then $E$ satisfies the $(1+\epsilon)$-Hahn-Banach extension property.
\end{theorem}

 \begin{proof}
    This follows by \cite[Cor.\ 4.2.5]{Perez-Garcia-Schikhof} (see also \cite[Thm.\ 3.16, vi.]{vanRooij}). 
 \end{proof}

\subsection{Orthogonal bases}
Another reason why spherically complete fields are nice has to do with the concept of (norm) orthogonality (see also a generalization in \autoref{def:t-orthogonality}). We will perhaps use a non-standard, but equivalent, definition of this notion.

\begin{definition}\cite[Thm.\ 2.2.3, Def.\ 2.2.4, Def.\ 2.2.6]{Perez-Garcia-Schikhof}
    \label{def:orthogonality}
    Let $(E,||\cdot||)$ be a normed space over a NA field $(k,|\cdot|)$. For $x,y \in E$, we say $x$ and $y$ are \emph{orthogonal}, denoted $x\perp y$, if for all $a, b \in k$,
    $
    ||ax + by|| = \max\{||ax||,||by||\}.    
    $
    Two subsets $C, D \subseteq E$ are \emph{orthogonal} if for all $x \in C, y \in D$, we have $x \perp y$. This is denoted $C \perp D$. A subset $X \subseteq E$ is an \emph{orthogonal system} if $0 \notin X$ and if for all $x \in X$, we have $\{x\} \perp \Span(X\setminus\{x\})$.
\end{definition}

We now make a few observations about this concept.

\begin{remarks}
    \label{rem:orthogonality}
    Let $(E,||\cdot||)$ be a normed space over a NA field $(k,|\cdot|)$ and $x, y \in E$. We have the following:
    \begin{enumerate}
        \item \label{rem:orthogonality.1} If $x \perp y$, then $y \perp x$. That is, the notion of orthogonality is symmetric. Moreover, $0 \perp x$ for all $x \in E$ and $x \perp x$ precisely when $x = 0$. Indeed, for the last assertion, one has 
        $
        0 = ||x + (-x)|| = \max\{||x||,||-x||\} = ||x||.    
        $
        If $x \perp y$, then for all $a, b \in k$, one has $ax \perp by$.

        \item \label{rem:orthogonality.2} By induction on $n$ it follows that a finite set $\{x_1,\dots,x_n\}$ is an orthogonal system if and only if for all $a_1,\dots,a_n \in k$, one has 
        $
        ||a_1x_1 + \dots + a_nx_n|| = \max\{||a_1x_1||,\dots,||a_nx_n||\}.    
        $
        Consequently, if the $x_i$'s are all non-zero, the set $\{x_1,\dots,x_n\}$ is $k$-linearly independent.

        \item \label{rem:orthogonality.3} $X \subset E$ is an orthogonal system if and only if every finite subset of $X$ is an orthogonal system. By \autoref{rem:orthogonality.2}, this implies that orthogonal systems are always linearly independent. By Zorn's Lemma, every orthogonal system is contained in a maximal orthogonal system. Moreover, one can show that maximal orthogonal systems in a Banach space over $k$ have the same cardinality. See \cite[Thm.\ 5.4]{vanRooij}.

        \item \label{rem:orthogonality.4} If $\{e_i\}_{i \in I} \subset E$ is an orthogonal system and if for all $i$ we choose $t_i \in k^\times$, then $\{t_ie_i\}_{i \in I}$ is also an orthogonal system.

        \item \label{rem:orthogonality.5} The notion of an orthogonal system, as defined, is not equivalent to saying that any two distinct elements of the set are pairwise orthogonal. For instance, consider $k^2$ with the norm 
        $
        ||(a,b)|| = \max\{|a|,|b|\}.    
        $
        Then one can check that the elements $(1,0),(0,1),(1,1)$ are pairwise orthogonal. However, the set $X \coloneqq \{(1,0),(0,1),(1,1)\}$ is not an orthogonal system because $(1,1) \in \Span(X\setminus\{(1,1)\})$, and a nonzero vector cannot be orthogonal to itself by \autoref{rem:orthogonality.1}.
    \end{enumerate}
\end{remarks}

The previous remarks raise the natural question of when a normed space over a field has a $k$-vector space basis that is also an orthogonal system.

\begin{theorem}
    \label{thm:finite-dimensional-orthogonal-basis}
    Let $(E,||\cdot||)$ be a finite dimensional normed space over a NA $(k,|\cdot|)$ that is spherically complete. Then $E$ has a $k$-vector space basis that is an orthogonal system, that is, there exists a $k$-basis $\{x_1,\dots,x_n\}$ of $E$ such that for all $a_1,\dots,a_n \in k$, 
    $
    ||a_1x_1 + \dots + a_nx_n|| = \max\{||a_1x_1||,\dots,||a_nx_n||\}.    
    $
\end{theorem}

\begin{proof}
    This follows by \cite[2.4.4/2]{BGR}. In their terminology, saying that a normed space is $k$-Cartesian means that every finite dimensional subspace has a $k$-vector space basis that is an orthogonal system \cite[2.4.3/2, 2.4.1/1]{BGR}.
\end{proof}

    There is a variant of \autoref{thm:finite-dimensional-orthogonal-basis} that holds for finite dimensional normed spaces over an \emph{arbitrary} NA field $(k,|\cdot|)$. We first introduce the relevant definition.
    
    \begin{definition}\cite[Pg.\ 27]{Perez-Garcia-Schikhof}
        \label{def:t-orthogonality}
        Let $(E,||\cdot||)$ be a normed space over a NA field $(k,|\cdot|)$. Fix a $t \in (0,1]$. A subset $X \subset E \setminus \{0\}$ is a \emph{$t$-orthogonal system} if for all $n \in \mathbb{Z}_{> 0}$, for all distinct $x_1,\dots,x_n \in X$ (if they exist) and for all $a_1,\dots,a_n \in k$, 
    $
    ||a_1x_1+\dots+a_nx_n|| \geq t\cdot\max\{||a_1x_1||,\dots,||a_nx_n||\}.    
    $
    \end{definition}
    
    Since normed spaces satisfy the NA-triangle inequality, a $1$-orthogonal system is precisely an orthogonal system in our earlier terminology. Like the orthogonality property, $t$-orthogonality is preserved under scaling elements of $X$ by elements of $k^\times$ \cite[Rem.\ 2.2.17]{Perez-Garcia-Schikhof}.  Furthermore, $t$-orthogonal systems are also $k$-linearly independent. 
    
    The relevant result about $t$-orthogonal systems is: 
    \begin{theorem}\cite[Thm.\ 3.15,~iii.]{vanRooij}
        \label{thm:t-orthogonal-basis}
        Let $(k,|\cdot|)$ be a NA field, $(E,||\cdot||)$ be a finite dimensional normed space over $k$, and $t \in (0,1)$. Then $E$ has a $k$-vector space basis that is a $t$-orthogonal system, that is, there exists a $k$-basis $\{x_1,\dots,x_n\}$ of $E$ such that for all $a_1,\dots,a_n \in k$,
        $
        ||a_1x_1+ \dots + a_nx_n|| \geq t\cdot\max\{||a_1x_1||,\dots,||a_nx_n||\}.
        $
    \end{theorem}

\subsection{Tate algebras and the Ohm-Rush trace property} Recall that the Ohm-Rush trace property encapsulates when an $R$-module $M$ has `sufficiently many' $R$-linear maps $M \to R$ (see \cite[Section 4.1]{DattaEpsteinTucker}). We now exhibit certain classes of NA extensions for which the corresponding extensions of Tate algebras are Ohm-Rush trace.

\begin{theorem}
    \label{thm:ORT-spherically-complete-countable-type}
    Let $(k,|\cdot|_k) \hookrightarrow (\ell, |\cdot|_\ell)$ be an extension of NA fields such that one of the following conditions is satisfied:
    \begin{itemize}
        \item $k$ is spherically complete, or
        \item $\ell$ is of countable type over $k$, that is, $\ell$ has a dense subspace over $k$ with a countable basis. 
    \end{itemize}
     Then for all integers $n > 0$, the extension of Tate algebras
    $
    T_n(k) \hookrightarrow T_n(\ell)    
    $
    is Ohm-Rush trace. Moreover, $T_n(k) \hookrightarrow T_n(\ell)$ is a split faithfully flat extension.
\end{theorem}

\begin{remark}
    \label{rem:faithful-flatness-extensions-Tate}
    Note that for any extension of NA fields $(k,|\cdot|_k) \hookrightarrow (\ell, |\cdot|_\ell)$, the induced extension of Tate algebras $T_n(k) \hookrightarrow T_n(\ell)$ is known to be faithfully flat. For instance, this follows by \cite[Lem.\ 2.1.2]{BerkovichEtale} and the fact that $T_n(\ell) = T_n(k) \widehat{\otimes}_k \ell$ \cite[Appendix B, Prop.\ 5]{Boschrigid}. However, the point of the final assertion of \autoref{thm:ORT-spherically-complete-countable-type} is that faithful flatness in the theorem's setup is also a consequence of the Ohm-Rush trace property and the fact that the extension is split.
\end{remark}

We will utilize the following lemma in the proof of \autoref{thm:ORT-spherically-complete-countable-type}.

\begin{lemma}
    \label{lem:uniform-scaling}
    Let $(k,|\cdot|)$ be a real-valued field such that $|k^\times|$ is not the trivial group. Let $(E,||\cdot||)$ be a normed space over $k$. Fix any $M \in |k^\times|$ such that $M > 1$. Then for any $x \in E \setminus \{0\}$, there exists $c \in k^{\times}$ such that $1 \leq ||cx|| < M$.
\end{lemma}

\begin{proof}[\textbf{Proof of \autoref{lem:uniform-scaling}}]
    The existence of an $M \in |k^\times|$ such that $M > 1$ follows because $|k^\times|$ is not the trivial group. Let $\alpha \in k$ such that $|\alpha| = M$. Since $||x||\neq 0$ and
    $
    \displaystyle \mathbb{R}_{> 0} = \bigsqcup_{e \in \mathbb{Z}} [M^e, M^{e+1}),    
    $
    there exists a unique $e \in \mathbb{Z}$ such that $M^e \leq ||x|| < M^{e+1}$. Then taking $c \coloneqq \alpha^{-e}$ we get the desired result.
\end{proof}

\begin{proof}[\textbf{Proof of \autoref{thm:ORT-spherically-complete-countable-type}}] 
    If we can show that $T_n(k) \hookrightarrow T_n(\ell)$ is Ohm-Rush trace, then flatness will follow because Ohm-Rush trace modules are flat; see \cite[Rem.\ 5.1.2(g)]{DattaEpsteinTucker} or \cite[Sec.\ 7, Pg.\ 66]{OhmRu-content}. For faithfulness, it suffices to show that $T_n(k) \hookrightarrow T_n(\ell)$ is a split extension. Using \autoref{thm:spherically-complete-Hahn-Banach}~\ref{thm:spherically-complete-Hahn-Banach:a} or \autoref{thm:+-epsilon-Hahn-Banach} one observes that $\id_k \colon k \to k$ extends to a continuous functional $\phi \colon \ell \to k$. Note that $\phi(1) = 1$. Since $\phi$ is continuous, it maps null sequences to null sequences by \autoref{lem:continuous-maps-normed-spaces}~\ref{lem:continuousnull}, and so, it induces a $T_n(k)$-linear map 
    \begin{align*}
    \overline{\phi} \colon T_n(\ell) &\to T_n(k)\\
    \sum_{\nu \in \mathbb{Z}^n_{\geq 0}} c_\nu X^\nu &\mapsto \sum_{\nu \in \mathbb{Z}^n_{\geq 0}} \phi(c_\nu)X^\nu.
    \end{align*}
    By construction, $\overline{\phi}$ sends $1 \mapsto 1$, that is, $\overline{\phi}$ is a left-inverse of $T_n(k) \hookrightarrow T_n(\ell)$.
    
    We will prove both cases simultaneously. We fix at the outset an $M \in |k^\times|$ such that $M > 1$ (note $k$ is non-trivially valued by definition of a NA field).

    Let $f \in T_n(\ell)$ and let
    \[
    \Tr(f) \coloneqq \im\left(\Hom_{T_n(k)}(T_n(\ell), T_n(k)) \xrightarrow{\ev @ f} T_n(k)\right).    
    \]
    We want to show that $f \in \Tr(f)T_n(\ell)$. 
    
    Recall that ideals in Tate algebras are closed (\autoref{thm:Tate-algebras-properties}~\ref{thm:Tate-ideals-closed}).  Thus, we would be done if we can show that there exists a constant $C > 0$ such that for all real numbers $\epsilon > 0$, there exists a $g_\epsilon \in \Tr(f)T_n(\ell)$ that satisfies 
    \begin{equation}
        \label{eq:bound-to-establish}
        ||f - g_\epsilon|| < C\epsilon.
    \end{equation}

    Fix $\epsilon > 0$. Suppose $f = \sum_{\nu \in \mathbb{Z}_{\geq 0}^n} a_\nu X^{\nu}$, and consider the polynomial 
    $
    f_\epsilon \coloneqq \sum_{|a_\nu|_\ell \geq \epsilon} a_\nu X^\nu.
    $
    By definition of the Gauss norm on Tate algebras, 
    $
    ||f - f_{\epsilon}|| < \epsilon.
    $ 
    Let $\ell_\epsilon$ be the $k$-subspace of $\ell$ spanned by the (finitely many) coefficients of $f_\epsilon$, that is, 
    $
    \ell_\epsilon = \Span_k\{a_\nu \colon |a_\nu|_\ell \geq \epsilon\}.    
    $
    Then $\ell_\epsilon$ is a finite dimensional $k$-vector space. Using \autoref{thm:t-orthogonal-basis}, fix a $k$-vector space basis 
    $
    \{x_{\epsilon, 1},\dots,x_{\epsilon,n_{\epsilon}}\}    
    $
    of $\ell_\epsilon$ that is a $1/2$-orthogonal system\footnote{Instead of $1/2$ one can choose any $t \in (0,1)$. When $k$ is spherically complete, one can take $t=1$ by \autoref{thm:finite-dimensional-orthogonal-basis}.}. Recall that this means that for all $b_1,\dots,b_{n_\epsilon} \in k$, 
    \begin{equation}
        \label{eq:bound-0}
    |b_1x_{\epsilon, 1} + \dots + b_{n_\epsilon}x_{\epsilon,n_\epsilon}|_\ell \geq \frac{1}{2}\max\{|b_1x_{\epsilon, 1}|_\ell,\dots,|b_{n_\epsilon}x_{\epsilon,n_\epsilon}|_\ell\}.    
    \end{equation}
     
    Since scaling by elements of $k^\times$ preserves $1/2$-orthogonality (see \autoref{def:t-orthogonality}), we may assume using \autoref{lem:uniform-scaling} that for all $i = 1, \dots, n_\epsilon$, 
    \begin{equation}
        \label{eq:bound-1}
        1 \leq |x_{\epsilon,i}|_\ell < M.    
    \end{equation}
    Note that $M$ is independent of $\epsilon$. Let $\{x^*_{\epsilon, 1},\dots,x^*_{\epsilon,n_{\epsilon}}\}$ denote the dual basis. We claim that
    \begin{equation}
        \label{eq:bound-1.5}
        ||x^*_{\epsilon, i}|| \leq 2.    
    \end{equation}
    Indeed, suppose $y = b_1x_{\epsilon,1} + \dots + b_{n_\epsilon}x_{\epsilon,n_\epsilon} \in \ell_\epsilon$, where the $b_i \in k$. Then 
    \[
       \frac{|x^*_{\epsilon,i}(y)|_k}{|y|_\ell} \stackrel{\autoref{eq:bound-0}}{\leq} 2\cdot\frac{|b_i|_k}{\max_{1 \leq j \leq n_\epsilon}\{|b_j|_k|x_{\epsilon,j}|_\ell\}} \stackrel{\autoref{eq:bound-1}}{\leq}2\cdot\frac{|b_i|_k}{\max_{1\leq j \leq n_\epsilon}\{|b_j|_k\}} \leq 2,
    \]
    proving our claim. 
    
  Now, if $k$ is spherically complete or if $\ell$ is of countable type over $k$, one uses \autoref{thm:spherically-complete-Hahn-Banach}~\ref{thm:spherically-complete-Hahn-Banach:a} or \autoref{thm:+-epsilon-Hahn-Banach} respectively to see that each $x^*_{\epsilon, i} \colon \ell_\epsilon \to k$ extends to a continuous functional $\widetilde{x^*_{\epsilon, i}} \colon \ell \to k$ with the property that 
  \[
  ||\widetilde{x^*_{\epsilon, i}}|| \leq 2||x^*_{\epsilon, i}|| \stackrel{\autoref{eq:bound-1.5}}{\leq} 4.
  \]
  In other words, for all $y \in \ell$, 
   \begin{equation}
    \label{eq:bound-2}
    |\widetilde{x^*_{\epsilon, i}}(y)|_k \leq 4|y|_\ell.
   \end{equation}
    Since continuous maps send null sequences to null sequences (\autoref{lem:continuous-maps-normed-spaces}), each $\widetilde{x^*_{\epsilon,i}} \colon \ell_\epsilon \to k$ induces a $T_n(k)$-linear map 
    \begin{align*}
        \varphi_{\epsilon,i} \colon T_n(\ell) &\to T_n(k)\\
        \sum_{\nu \in \mathbb{Z}_{\geq 0}^n} c_\nu X^{\nu} &\to \sum_{\nu \in \mathbb{Z}_{\geq 0}^n} \widetilde{x^*_{\epsilon,i}}(c_\nu) X^{\nu}.
    \end{align*}
    Using \autoref{eq:bound-2} and the Gauss norm on Tate algebras, we further get that for all $g \in T_n(\ell)$,
    \begin{equation}
        \label{eq:bound-3}
        ||\varphi_{\epsilon,i}(g)|| \leq 4||g||.
    \end{equation}
    Define
    $
    g_\epsilon \coloneqq \sum_{i=1}^{n_\epsilon} \varphi_{\epsilon,i}(f)x_{\epsilon,i}.    
    $
    By construction, $g_{\epsilon} \in \Tr(f)T_n(\ell)$. In addition, since we have
    \begin{align*}
    \varphi_{\epsilon,i}(f)  
    &= \varphi_{\epsilon,i}\left(\sum_{|a_\nu|_\ell \geq \epsilon}a_\nu X^\nu\right) + \varphi_{\epsilon,i}\left(\sum_{|a_\nu|_\ell < \epsilon}a_\nu X^\nu\right)\\
    &= \sum_{|a_\nu|_\ell\geq\epsilon}\widetilde{x^*_{\epsilon,i}}(a_\nu)X^\nu + \varphi_{\epsilon,i}\left(\sum_{|a_\nu|_\ell < \epsilon}a_\nu X^\nu\right)\\
    &\stackrel{\widetilde{x^*_{\epsilon,i}}|_{\ell_\epsilon} = x^*_{\epsilon,i}}{=} \sum_{|a_\nu|_\ell\geq\epsilon}x^*_{\epsilon,i}(a_\nu)X^\nu + \varphi_{\epsilon,i}\left(\sum_{|a_\nu|_\ell < \epsilon}a_\nu X^\nu\right),
    \end{align*}
    we then get
    \begin{align*}
        g_\epsilon &= \sum_{i=1}^{n_\epsilon}x_{\epsilon,i}\sum_{|a_\nu|_\ell\geq\epsilon}x^*_{\epsilon,i}(a_\nu)X^\nu + \sum_{i=1}^{n_\epsilon}x_{\epsilon,i}\varphi_{\epsilon,i}\left(\sum_{|a_\nu|_\ell < \epsilon}a_\nu X^\nu\right)\\
        &= \sum_{|a_\nu|_\ell\geq\epsilon}\left(\sum_{i=1}^{n_\epsilon} x^*_{\epsilon,i}(a_\nu)x_{\epsilon,i}\right)X^\nu + \sum_{i=1}^{n_\epsilon}x_{\epsilon,i}\varphi_{\epsilon,i}\left(\sum_{|a_\nu|_\ell < \epsilon}a_\nu X^\nu\right)\\
        &= \sum_{|a_\nu|_\ell\geq\epsilon} a_\nu X^\nu + \sum_{i=1}^{n_\epsilon}x_{\epsilon,i}\varphi_{\epsilon,i}\left(\sum_{|a_\nu|_\ell < \epsilon}a_\nu X^\nu\right)\\
        &= f_\epsilon + \sum_{i=1}^{n_\epsilon}x_{\epsilon,i}\varphi_{\epsilon,i}\left(\sum_{|a_\nu|_\ell < \epsilon}a_\nu X^\nu\right).
    \end{align*}
    Hence,
    \begin{align*}
    ||g_\epsilon - f_\epsilon|| &\leq \max_{1 \leq i \leq n_\epsilon}\left\{|x_{\epsilon,i}|_\ell\left|\left| \varphi_{\epsilon,i}\bigg(\sum_{|a_\nu|_\ell < \epsilon}a_\nu X^\nu\bigg)\right|\right|\right\} \\
    &\stackrel{\autoref{eq:bound-3}}{\leq} \max_{1 \leq i \leq n_\epsilon}\left\{|x_{\epsilon,i}|_\ell\cdot4\cdot\left|\left| \sum_{|a_\nu|_\ell < \epsilon}a_\nu X^\nu\right|\right|\right\}\\
    &\stackrel{\autoref{eq:bound-1}}{<} 4M\epsilon.
    \end{align*}
    Then 
    $
    ||f - g_\epsilon|| \leq \max\{||f-f_\epsilon||, ||g_\epsilon-f_\epsilon||\} < \max\{\epsilon,4M\epsilon\} = 4M\epsilon.    
    $
    Here we are using the fact that $M > 1$. Taking $C = 4M$, this establishes \autoref{eq:bound-to-establish}, thereby completing the proof of the Theorem.
\end{proof}

\subsection{Tate algebras and intersection flatness} Let $(k,|\cdot|_k)$ be a NA field. Recall that if $\ell$ is an algebraic field extension of $k$, then there is a unique extension of $|\cdot|_k$ to $\ell$ that makes $\ell$ into a normed space over $k$ \cite[Appendix A, Cor.\ 2, Thm.\ 3]{Boschrigid}. The completion, $\widehat{\ell}$, of $\ell$ with respect to this extended norm then becomes a NA field with a canonical norm that extends the norm on $k$. In particular, one can take $\ell = \overline{k}$, an algebraic closure of $k$. Note that $\overline{k}$ may not be complete with respect to its unique norm extension. Thus, in order to get a NA field, one takes $\widehat{\overline{k}}$. A remarkable result, due to Krasner, is that $\widehat{\overline{k}}$ is also algebraically closed \cite[Appendix A, Lem.\ 6]{Boschrigid}. Thus, if $(k,|\cdot|_k)\hookrightarrow (\ell,|\cdot|_\ell)$ is any algebraic extension of NA fields, then one can always find a $k$-embedding $\ell \hookrightarrow \widehat{\overline{k}}$, which has to be norm preserving because the restriction of the norm on $\widehat{\overline{k}}$ yields a norm on $\ell$ extending $|\cdot|_k$, and so, must coincide with $|\cdot|_\ell$ by uniqueness. As a consequence, there is a faithfully flat $T_n(k)$-algebra homomorphism $T_n(\ell) \hookrightarrow T_n(\widehat{\overline{k}})$ (see \autoref{rem:faithful-flatness-extensions-Tate}). In particular, by restriction of scalars, $T_n(\ell) \hookrightarrow T_n(\widehat{\overline{k}})$ is a pure map of $T_n(k)$-modules. Thus, if we can show that $T_n(\widehat{\overline{k}})$ satisfies good properties as a $T_n(k)$-algebra, one can often use pure descent to show that $T_n(\ell)$ also satisfies similar properties. Our main result then is the following:

\begin{theorem}
    \label{thm:IF-completion-algebraic-closure}
    Let $(k,|\cdot|_k)$ be a NA field. Let $\overline{k}$ be an algebraic closure of $k$. Then, $T_n(\widehat{\overline{k}})$ is an intersection flat $T_n(k)$-algebra.
\end{theorem}

Using descent of intersection flatness along pure maps \cite[Corollary 4.3.2]{DattaEpsteinTucker}, the following is an immediate consequence of \autoref{thm:IF-completion-algebraic-closure} and the above discussion.

\begin{corollary}
    \label{cor:IF-algebraic-extensions}
Let $(k,|\cdot|_k) \hookrightarrow (\ell, |\cdot|_\ell)$ be an algebraic extension of NA fields. Then $T_n(\ell)$ is an intersection flat $T_n(k)$-algebra.
\end{corollary}

\begin{proof}[\textbf{Proof of \autoref{thm:IF-completion-algebraic-closure}}]
    Define $\Sigma$ to be the collection of intermediate field extensions $k \subseteq F \subseteq \overline{k}$ such that $F$ has a countable basis over $k$. Note that $\Sigma$ is closed under compositum of field extensions. Indeed, if $E, F \in \Sigma$, then the compositum $E.F$ is a quotient of $E \otimes_k F$, and the latter has a countable basis over $k$ because both $E, F$ do. Thus, $\Sigma$ is filtered under inclusion. 
    
    Now take $\Sigma'$ to be the collection of topological closures in $\widehat{\overline{k}}$ of the elements of $\Sigma$. Note that if $F \in \Sigma$, then its topological closure, $F^{\cl}$, in $\widehat{\overline{k}}$ is the completion of $F$ with respect to the unique extension of the norm on $k$ to $F$ (this norm also coincides with the restriction to $F$ of the norm on $\widehat{\overline{k}}$). Indeed, by the universal property of completions of normed fields, the completion $\widehat{F}$ of $F$ embeds isometrically in $\widehat{\overline{k}}$, so we may regard it as a subfield of $\widehat{\overline{k}}$. But $\widehat{F}$ must be closed in $\widehat{\overline{k}}$ since it is a complete subspace of a complete space. Since $F$ is dense in $\widehat{F}$, it then follows that we must have $F^{\cl} = \widehat{F}$. Thus, the elements of $\Sigma'$ are intermediate NA subfields of $\widehat{\overline{k}}/k$ that are of countable type over $k$ by construction. Moreover, $\Sigma'$ is filtered under inclusion because $\Sigma$ is. 
    
    We then get a system of Tate algebras $\{T_n(\ell) \colon \ell \in \Sigma'\}$ filtered under inclusion, where for $\ell \subseteq \ell'$, $T_n(\ell) \hookrightarrow T_n(\ell')$ is faithfully flat by \autoref{rem:faithful-flatness-extensions-Tate} (or one can use \autoref{thm:ORT-spherically-complete-countable-type} because $\ell'$ is of countable type over $\ell$). Thus, by restriction of scalars, $T_n(\ell) \hookrightarrow T_n(\ell')$ is a pure map of $T_n(k)$-modules. By \autoref{thm:ORT-spherically-complete-countable-type}, each $T_n(\ell)$ is an Ohm-Rush trace $T_n(k)$-algebra. Thus, for all $\ell \in \Sigma'$, $T_n(\ell)$ is an intersection flat $T_n(k)$-algebra by \cite[Proposition 4.3.8]{DattaEpsteinTucker}. By \cite[Corollary 4.3.2]{DattaEpsteinTucker}, it then follows that 
    $
    \colim_{\ell \in \Sigma'} T_n(\ell)    
    $
    is an intersection flat $T_n(k)$-algebra. 

    To finish the proof, it is enough to show that $\colim_{\ell \in \Sigma'} T_n(\ell) = T_n(\widehat{\overline{k}})$. Indeed, each $T_n(\ell)$ is a subring of $T_n(\widehat{\overline{k}})$ by construction. Since the system $\{T_n(\ell) \colon \ell \in \Sigma'\}$ is filtered under inclusion, we then have
    $
        \colim_{\ell \in \Sigma'} T_n(\ell)  = \bigcup_{\ell \in \Sigma'} T_n(\ell) \subseteq T_n(\widehat{\overline{k}}).    
    $
    Now let 
    $
    \sum_{\nu \in \mathbb{Z}^n_{\geq 0}} c_\nu X^\nu \in T_n(\widehat{\overline{k}}).    
   $
    Since $\overline{k}$ is dense in $\widehat{\overline{k}}$, for each $\nu \in \mathbb{Z}^n_{\geq 0}$, one can choose a sequence $(a_{\nu,n})_{n \in \mathbb{Z}_{\geq 0}}$ of elements in $\overline{k}$ such that 
    \[
    \lim_{n \mapsto \infty} a_{\nu,n} = c_\nu.    
    \]
    Then $S \coloneqq \bigcup_{\nu \in \mathbb{Z}^n_{\geq 0}} \{a_{\nu,n} \colon n \in \mathbb{Z}_{\geq 0}\}$ is a countable subset of $\overline{k}$, and so, $k(S) \in \Sigma$. Then $k(S)^{\cl} \in \Sigma'$, and by construction, $c_\nu \in k(S)^{\cl}$ for all $\nu$. This shows that $\sum_{\nu \in \mathbb{Z}^n_{\geq 0}} c_\nu X^\nu \in T_n(k(S)^{\cl})$, proving that $\bigcup_{\ell \in \Sigma'} T_n(\ell) = T_n(\widehat{\overline{k}})$.
\end{proof}

\begin{corollary}
    \label{cor:Tate-char-p-FIF-FORT}
    Let $(k,|\cdot|_k)$ be a NA field of characteristic $p > 0$. Consider the Tate algebra $T_n(k)$. Then we have the following:
    \begin{enumerate}[label=\textnormal{(\arabic*)}]
        \item $T_n(k)$ is $F$-intersection flat.\label{cor:Tate-char-p-FIF-FORT.1}
        \item If $k$ is spherically complete or if $k^{1/p}$ has a dense subspace over $k$ with a countable basis, then $T_n(k)$ is FORT.\label{cor:Tate-char-p-FIF-FORT.2}
    \end{enumerate}
\end{corollary}

\begin{proof}
    For an arbitrary NA field $k$ of characteristic $p$, the Frobenius map of $T_n(k)$ can be identified with the composition $T_k(k) \hookrightarrow T_k(k^{1/p}) \hookrightarrow (T_n(k))^{1/p}$. The extension $T_k(k^{1/p}) \hookrightarrow (T_n(k))^{1/p}$ is free with basis $\{X_1^{\alpha_1/p}\cdots X_n^{\alpha_n/p}\colon 0 \leq \alpha_i \leq p-1\}$ (to see this, one can adapt the argument in \cite[Lemma 3.3.3]{DattaMurayamaFsolidity}), and is hence ORT, and thus, also intersection flat. Since the ORT and intersection flatness properties are preserved under composition \cite[Lemma 4.1.7, Remark 4.2.3 (d)]{DattaEpsteinTucker}, \ref{cor:Tate-char-p-FIF-FORT.1} follows by \autoref{cor:IF-algebraic-extensions} and \ref{cor:Tate-char-p-FIF-FORT.2} follows by \autoref{thm:ORT-spherically-complete-countable-type}
\end{proof}

Furthermore, as a consequence of \cite{SharpBigTestElements} we obtain the following new class of rings that admit test elements.

\begin{corollary}
    \label{cor:affinoid-rings-test-elements}
    Let $(k,|\cdot|_k)$ be a NA field of characteristic $p > 0$. Let $R$ be essentially of finite type over $T_n(k)$ (for example, an affinoid algebra) and regular in codimension $0$. Then $R$ has big test elements. In fact, if $c \in R$ is not contained in any minimal primes of $R$ and $R_c$ is regular, then some power of $c$ is a big test element.
\end{corollary}

\begin{proof}
    By assumption, $R$ is a homomorphic image of a localization of a polynomial ring over $T_n(k)$. The latter ring is excellent and $F$-intersection flat by \autoref{cor:Tate-char-p-FIF-FORT} \ref{cor:Tate-char-p-FIF-FORT.1} and \autoref{thm:FORT-F-intersection-flat} \ref{thm:FORT-F-intersection-flat.7}. Then the result is an immediate application of \cite[Theorem 10.2]{SharpBigTestElements}.
\end{proof}

\begin{corollary}
    \label{cor:ideal-adic-completion-affinoid}
    Let $(k,|\cdot|_k)$ be a NA field of characteristic $p > 0$. Let $y_1,\dots,y_m$ be indeterminates over $T_n(k)$. Then $T_n(k) {\llbracket}y_1,\dots,y_m{\rrbracket}$ is $F$-intersection flat. Thus, ideal adic completions of affinoid $k$-algebras regular in codimension $0$ have big test elements.
\end{corollary}

\begin{proof}
    For the first assertion, Frobenius on $T_n(k){\llbracket}y_1,\dots,y_m{\rrbracket}$ factorizes as 
    \[T_n(k){\llbracket}y_1,\dots,y_m{\rrbracket} \to T_n(k^{1/p}){\llbracket}y_1,\dots,y_m{\rrbracket} \to T_n(k^{1/p}){\llbracket}y_1^{1/p},\dots,y_m^{1/p}{\rrbracket},\] 
    where the second map is a free extension. Thus, it suffices to show that the first map $T_n(k){\llbracket}y_1,\dots,y_m{\rrbracket} \to T_n(k^{1/p}){\llbracket}y_1,\dots,y_m{\rrbracket}$ is intersection flat. As in the proof of \autoref{thm:IF-completion-algebraic-closure}, any $f \in T_n(k^{1/p}){\llbracket}y_1,\dots,y_m{\rrbracket}$ is in $T_n(\ell){\llbracket}y_1,\dots,y_m{\rrbracket}$, where $k \subset \ell \subset k^{1/p}$ is an intermediate NA field of countable type over $k$. Indeed, as before, one can take $\ell$ to just be the topological closure in $k^{1/p}$ of the field obtained by adjoining to $k$ all the coefficients in $k^{1/p}$ of all the (countably many) terms of $f$. The upshot is that $T_n(k^{1/p}){\llbracket}y_1,\dots,y_m{\rrbracket}$ is a filtered union of the $T_n(k){\llbracket}y_1,\dots,y_m{\rrbracket}$-algebras $T_n(\ell){\llbracket}y_1,\dots,y_m{\rrbracket}$, where $k \subset \ell \subset k^{1/p}$ is an intermediate NA field of countable type over $k$. Since $T_n(k) \to T_n(\ell)$ is Ohm-Rush trace by \autoref{thm:ORT-spherically-complete-countable-type}, so is $T_n(k){\llbracket}y_1,\dots,y_m{\rrbracket} \to T_n(\ell){\llbracket}y_1,\dots,y_m{\rrbracket}$ by \cite[Proposition 4.1.9 (3)]{DattaEpsteinTucker}. Moreover, if $k \subset \ell_1 \subset \ell_2 \subset k^{1/p}$ are two such intermediate extensions, then $T_n(\ell_1) \hookrightarrow T_n(\ell_2)$ is split faithfully flat by \autoref{thm:ORT-spherically-complete-countable-type}, and so, $T_n(\ell_1){\llbracket}y_1,\dots,y_m{\rrbracket} \hookrightarrow T_n(\ell_2){\llbracket}y_1,\dots,y_m{\rrbracket}$ is also split faithfully flat. In other words, the transition maps of the above filtered system are all split faithfully flat. Then $T_n(k^{1/p}){\llbracket}y_1,\dots,y_m{\rrbracket}$ is intersection flat over $T_n(k){\llbracket}y_1,\dots,y_m{\rrbracket}$ by \cite[Corollary 4.3.2 (3)]{DattaEpsteinTucker}, proving the first assertion.

    Now suppose $A$ is an affinoid $k$-algebra regular in codimension $0$ (i.e. $A$ satisfies $R_0$) and $I$ is an ideal of $A$. Since $A$ is excellent, $A \to \widehat{A}^I$ is a regular map by \cite[\href{https://stacks.math.columbia.edu/tag/0AH2}{Tag 0AH2}]{stacks-project}. Hence $\widehat{A}^I$ is also $R_0$ by \cite[\href{https://stacks.math.columbia.edu/tag/033A}{Tag 033A}]{stacks-project}. Choose $n > 0$ such that we have a surjection $T_n(k) \twoheadrightarrow A$. If $I$ is generated by $m$ elements, we then get a surjection $T_n(k){\llbracket}y_1,\dots,y_m{\rrbracket} \twoheadrightarrow \widehat{A}^I$, and so, $\widehat{A}^I$ has big test elements by \cite[Theorem 10.2]{SharpBigTestElements}.
\end{proof}

\begin{remarks}
    \label{rem:adjoining-powerseries-variables}
    {\*}
    \begin{enumerate}
        \item By a proof similar to \autoref{cor:affinoid-rings-test-elements} and \autoref{cor:ideal-adic-completion-affinoid} one can deduce that if $R$ is $R_0$ and essentially of finite type over $T_n(k)$, where $k$ has characteristic $p > 0$, then any ideal-adic completion of $R$ has a big test element. We leave the details to the reader.

        \item We do not know if a more general version of \autoref{cor:ideal-adic-completion-affinoid} holds. Namely, suppose $R$ is an excellent regular ring of prime characteristic $p > 0$ that is $F$-intersection flat. Then we do not know if $R{\llbracket}x{\rrbracket}$ is also $F$-intersection flat. The problem again reduces to showing that $R{\llbracket}x{\rrbracket} \to R^{1/p}{\llbracket}x{\rrbracket}$ is intersection flat, which we do not know how to show.
    \end{enumerate}
\end{remarks}

In the proof of \autoref{thm:IF-completion-algebraic-closure} we showed that $T_n(\widehat{\overline{k}})$ is a filtered colimit of ORT $T_n(k)$-algebras whose transition maps are split faithfully flat extensions (by \autoref{thm:ORT-spherically-complete-countable-type}). Thus, one may wonder if $T_n(\widehat{\overline{k}})$ is an ORT $T_n(k)$-algebra and not just intersection flat. We now give an example to illustrate that there exist NA fields $k$ for which 
$
\Hom_{T_n(k)}(T_n(\widehat{\overline{k}}), T_n(k))) = 0,
$
and so, $T_n(\widehat{\overline{k}})$ cannot always be an ORT $T_n(k)$-algebra. Our example essentially relies on the construction given in \cite[Sec.\ 5]{DattaMurayamaTate} of a NA field $k$ of characteristic $p > 0$ for which there exist no non-zero continuous functionals $k^{1/p} \to k$. Indeed, we first claim that such a field $k$ cannot have non-zero continuous functionals $\widehat{\overline{k}} \to k$. For suppose a continuous functional $f \colon \widehat{\overline{k}} \to k$ exists and $f(x) \neq 0$. Then the composition 
\[
\widehat{\overline{k}} \xrightarrow{\qquad x \cdot \qquad} \widehat{\overline{k}}  \xrightarrow{\qquad f \qquad} k  
\]
is a continuous functional that sends $1 \mapsto f(x) \neq 0$. Since $k^{1/p}$ can be realized as a subfield of $\widehat{\overline{k}}$, one can then restrict the above functional to $k^{1/p}$ giving a continuous functional $k^{1/p} \to k$ that sends $1 \mapsto f(x) \neq 0$. But this contradicts our choice of $k$. Now the fact that 
$\Hom_{T_n(k)}(T_n(\widehat{\overline{k}}), T_n(k))) = 0$
for this choice of $k$ follows from the following observation.

\begin{proposition}(c.f. \cite[Thm.\ 3.1]{DattaMurayamaTate})
    \label{prop:Tate-solidity}
    Let $(k,|\cdot|_k) \hookrightarrow (\ell,|\cdot|_\ell)$ be an extension of NA fields. Then the following are equivalent:
    \begin{enumerate}[label=\textnormal{(\alph*)}]
        \item \label{prop:Tate-solidity:a} For all integers $n > 0$, $T_n(k) \hookrightarrow T_n(\ell)$ splits.
        \item \label{prop:Tate-solidity:b} For all integers $n > 0$, there exists a non-zero $T_n(k)$-linear map $T_n(\ell) \to T_n(k)$.
        \item \label{prop:Tate-solidity:c} There exists a non-zero $T_1(k)$-linear map $T_1(\ell) \to T_1(k)$.
        \item \label{prop:Tate-solidity:d} There exists a non-zero continuous functional $\ell \to k$.
    \end{enumerate}
\end{proposition}

\begin{proof}
    The proof is very similar to that of \cite[Thm.\ 3.1]{DattaMurayamaTate}. We clearly have $\ref{prop:Tate-solidity:a} \implies \ref{prop:Tate-solidity:b} \implies\ref{prop:Tate-solidity:c}$. We next show $\ref{prop:Tate-solidity:d} \implies\ref{prop:Tate-solidity:a}$. Let $f \colon \ell \to k$ be a non-zero continuous linear functional. Say $f(x) \neq 0$. Then the composition
    \[
        \ell \xrightarrow{\qquad x \cdot \qquad} \ell \xrightarrow{\qquad f \qquad} k \xrightarrow{\qquad f(x)^{-1}\cdot \qquad} k
    \]
    is a continuous functional that maps $1 \mapsto 1$. Thus, we may assume without loss of generality that $f$ maps $1 \mapsto 1$. Then the map
    \begin{align*}
    T_n(\ell) &\to T_n(k)\\
    \sum_{\nu \in \mathbb{Z}^n_{\geq 0}} a_\nu X^\nu &\mapsto  \sum_{\nu \in \mathbb{Z}^n_{\geq 0}} f(a_\nu) X^\nu  
    \end{align*}
    gives a splitting of $T_n(k) \hookrightarrow T_n(\ell)$.

    It remains to show $\ref{prop:Tate-solidity:c} \implies \ref{prop:Tate-solidity:d}$. We first claim that $T_1(k) \hookrightarrow T_1(\ell)$ splits. Indeed, let $\varphi \colon T_1(\ell) \to T_1(k)$ be a non-zero $T_1(k)$-linear map. Since $T_1(k)$ is a Euclidean domain (\autoref{thm:Tate-algebras-properties}~\ref{thm:Tate-Euclidean-dim1}), it is a PID. Thus, $\im(\varphi) = HT_1(k)$, for some non-zero $H \in T_1(k)$. But there is a $T_1(k)$-linear isomorphism $HT_1(k) \cong T_1(k)$ that sends $H \mapsto 1$. The upshot is that we can restrict the codomain of $\varphi$ and use this isomorphism to assume $1 \in \im(\varphi)$. Then choosing $G \in T_1(\ell)$ such that $\varphi(G) = 1$, we can pre-compose $\varphi$ by multiplication by $G$ on $T_1(\ell)$ to further get a $T_1(k)$-linear map $T_1(\ell) \to T_1(k)$ that sends $1 \mapsto 1$, that is, we get a splitting of $T_1(k)\hookrightarrow T_1(\ell)$.

    Thus, fix a splitting $\phi \colon T_1(\ell) \to T_1(k)$ of $T_1(k) \hookrightarrow T_1(\ell)$. Then consider the composition
    \[
    f \coloneqq \ell \hookrightarrow T_1(\ell) \xlongrightarrow{\phi} T_1(k) \twoheadrightarrow k,    
    \]
    where the map $T_1(k) \twoheadrightarrow k$ is the $k$-algebra homomorphism obtained by sending $X \mapsto 0$. The map $f$ is $k$-linear, and, by construction, $f(1) = 1$. Hence, it suffices to show $f$ is continuous. Assume that it is not. Then using \autoref{lem:continuous-maps-normed-spaces}, there exists a sequence $(a_i)_{i \in \mathbb{Z}_{\geq 0}}$ in $\ell$ such that $|a_i|_\ell \to 0$ as $i \to \infty$, but such that, if 
    \begin{equation}
        \label{eq:!!}
    \phi(a_i) \coloneqq \sum_{j = 0}^\infty b_{i,j}X^j,    
    \end{equation}
    then 
    \begin{equation}
        \label{eq:!!!}
    |f(a_i)|_k = |b_{i,0}|_k \geq i!.
\end{equation} 
Using the sequence $(a_i)_i$, we will now construct an element of $T_1(\ell)$ whose image under $\phi$ cannot be in $T_1(k)$, thereby giving a contradiction.

    Let $m_0 \coloneqq 0$, and for all $i \geq 1$, inductively choose $m_i \gg m_{i-1}$ such that 
    $
     \displaystyle   \max_{0\leq r\leq i-1} \{|b_{m_r,i-r}|\} < |b_{m_i,0}|.
    $
    Note that such $m_i$ exist because $|b_{i,0}| \to \infty$ as $i \to \infty$. By the non-Archimedean triangle inequality, we have
    \begin{equation}
        \label{eq:!}
    \left|\sum_{r=0}^{i-1} b_{m_r,i-r}\right|_k < |b_{m_i,0}|_k.    
    \end{equation}
    Now consider the restricted power series 
    $
    \sum_{r = 0}^\infty a_{m_r}X^r \in T_1(\ell).    
    $
    Applying the splitting $\phi \colon T_1(\ell) \to T_1(k)$ to this power series, we see that 
    \begin{align*}
        \phi\left(\sum_{r = 0}^\infty a_{m_r}X^r\right) &= \phi\left(\sum_{r = 0}^i a_{m_r}X^r\right) + \phi\left(\sum_{r = i+1}^\infty a_{m_r}X^r\right)\\
         &= \sum_{r=0}^i X^r \phi(a_{m_r}) + X^{i+1}\cdot\phi\left(\sum_{r = i+1}^\infty a_{m_r}X^{r-i-1}\right)\\
         & \stackrel{\autoref{eq:!!}}{=} \sum_{r=0}^i \left(X^r \sum_{j=0}^\infty b_{m_r,j}X^j\right) + X^{i+1}\cdot\phi\left(\sum_{r = i+1}^\infty a_{m_r}X^{r-i-1}\right).
    \end{align*}
    Comparing coefficients, we see that the coefficient of $X^i$ in $\phi\left(\sum_{r = 0}^\infty a_{m_r}X^r\right)$ is  
    $
        \sum_{r=0}^{i} b_{m_r,i-r}.
    $
    But then 
    \[
      \left|\sum_{r=0}^{i} b_{m_r,i-r}\right|_k = \left|b_{m_i,0} + \sum_{r=0}^{i-1} b_{m_r,i-r}\right|_k \stackrel{\autoref{eq:!}}{=}  |b_{m_i,0}|_k \stackrel{\autoref{eq:!!!}}{\geq} m_i!.
    \]
    Here the penultimate equality follows by \autoref{eq:!} and \autoref{lem:norm-becomes-max}. Thus, $\phi\left(\sum_{r = 0}^\infty a_{m_r}X^r\right)$ is not a restricted power series, giving us the desired contradiction.
\end{proof}

\section{Acknowledgements}
This project branched off from a project that the authors began with Takumi Murayama, and was facilitated by a SQuaRE at the American Institute of Mathematics (AIM). The authors thank AIM for providing a supportive and mathematically rich environment.

We are grateful to Takumi for allowing us to write this standalone paper and have greatly benefitted from our conversations with him. In addition, the material of \autoref{subsec:Artin-Rees} is based on unpublished results obtained by Takumi and the first author; we thank Takumi for letting us include the results in this paper. We have also benefited from conversations with Karen Smith, Mel Hochster, Yongwei Yao, Gabriel Picavet, Jay Shapiro, Johan de Jong, Alex Perry, Remy van Dobben de Bruyn and Bhargav Bhatt.

\section{Statements and Declarations}

\subsection{Funding and competing interests}

Rankeya Datta is employed by the University of Missouri.  Neil Epstein is employed by George Mason University.  Karl Schwede is employed by the University of Utah.  Kevin Tucker is employed by the University of Illinois at Chicago.  

Rankeya Datta was partially supported by NSF grant DMS \#2502333, an AMS-Simons travel grant and a grant from the Simons Foundation MP-TSM-00002400.  Neil Epstein did not receive support from any organization for the submitted work.
Karl Schwede was supported by NSF FRG Grant \#1952522, NSF Grants DMS \#2101800 and \#2501903, and a Simons Foundation Fellowship and Travel Support for Mathematicians SFI-MPS-TSM00013051.
Kevin Tucker was supported in part by NSF grants DMS \#2200716 and \#2501904.

The authors have no other relevant financial or non-financial interests to disclose.

\bibliographystyle{skalpha}
\bibliography{main,preprints}

\end{document}